\documentclass[reqno,centertags,12pt]{amsart}
\usepackage{amsmath,amsthm,amscd,amssymb,latexsym,verbatim}

\usepackage{graphicx,epsf,cite}

\usepackage{xcolor}

\usepackage{tikz}
\pgfdeclarelayer{nodelayer}
\pgfdeclarelayer{edgelayer}
\pgfsetlayers{edgelayer,nodelayer,main}
\tikzstyle{arrow}=[draw=black,arrows=-latex]

-\marginparwidth \oddsidemargin 0mm \evensidemargin \oddsidemargin

\newtheorem{theorem}{Theorem}[section]

\newtheorem{lemma}[theorem]{Lemma}

\theoremstyle{definition}

\theoremstyle{remark}

\newcounter{smalllist}

\DeclareMathOperator*{\sgn}{sgn}

\allowdisplaybreaks
\numberwithin{equation}{section}

\newcommand{\lb}{\label}

\newcommand{\supp}{\text{\rm{supp}}}

\newcommand{\beq}{\begin{equation}}
\newcommand{\eeq}{\end{equation}}

\newcommand{\bal}{\begin{align}}
\newcommand{\eal}{\end{align}}
\newcommand{\bals}{\begin{align*}}
\newcommand{\eals}{\end{align*}}

\newcommand{\bbN}{{\mathbb{N}}}
\newcommand{\bbR}{{\mathbb{R}}}

\newcommand{\bbZ}{{\mathbb{Z}}}

\newcommand{\eps}{\varepsilon}

\newcommand{\bbRr}{{\mathbb{R}}}

\begin{document}
\title[Local Regularity for the Muskat Problem]
{The 2D Muskat Problem I: \\ Local Regularity on the Half-plane, Plane, and Strips}

\author{Andrej Zlato\v s}
\address{\noindent Department of Mathematics \\ UC San Diego \\ La Jolla, CA 92093, USA \newline Email:
zlatos@ucsd.edu}


\begin{abstract}
We prove local well-posedness for the Muskat problem on the half-plane, which models motion of an interface between two fluids of distinct densities (e.g., oil and water) in a porous medium (e.g., an aquifer) that sits atop an impermeable layer (e.g., bedrock).  Our result allows for the interface to touch the bottom, and hence applies to the important scenario of the heavier fluid invading a region occupied by the lighter fluid along the impermeable layer.  We use this result in the companion paper 
\cite{ZlaMuskatBlowup} 
to prove existence of finite time stable regime  singularities in this model, including for arbitrarily small initial data.  We do not require the interface and its derivatives to vanish at $\pm\infty$ or be periodic, and even allow it to be $O(|x|^{1-})$,
which is an optimal bound on the power of growth.  We also extend our results to the cases of the Muskat problem on the whole plane and on horizontal strips, where almost all previous works did impose such limiting requirements.
\end{abstract}

\maketitle

\section{Introduction and Main Results} \lb{S1}

The {\it Muskat problem} is a mathematical model for the motion of the interface between two incompressible immiscible fluids of different densities $\rho_1>\rho_0$ (such as water and oil, or salt water and fresh water) inside a porous medium (such as a sand or sandstone aquifer) \cite{Muskat, Bear}.  It has also been used as a model for cell velocity in the growth of tumors \cite{Friedman, Pozrikidis}, and is  related to the Hele-Shaw cell problem \cite{SafTay, Hele-Shaw}. If one instead considers a single fluid with continuously varying density, the corresponding model is the {\it incompressible porous medium (IPM) equation}.  In both cases one studies the transport PDE
\begin{equation} \lb{1.1}
\partial_{t} \rho + u\cdot\nabla \rho =0
\end{equation}
for fluid density $\rho$ and velocity $u$, with the former  of the form
\beq\lb{1.3a}
\rho({\bf x},t)=\rho_1- (\rho_1-\rho_0)\chi_{\Omega_t}({\bf x})
\eeq
in the case of the Muskat problem, where $\Omega_t$ is the (time-dependent) region of the lighter fluid.
The velocity $u$ is  determined via {\it Darcy's law}, which after appropriate temporal scaling --- by a factor that is the product of the gravitational constant, permeability of the medium, and the reciprocal of the fluid viscosity --- becomes
\begin{equation} \lb{1.2}
u:=-\nabla p - (0,\rho) \qquad \text{and}\qquad \nabla\cdot u=0.
\end{equation}
The term $-(0,\rho)$ models downward motion of denser regions of the fluid due to gravity (with $0$ the zero vector in one fewer dimensions than ${\bf x}$) and $p$ is the fluid pressure, caused by incompressibility and used to obtain a divergence-free $u$.  In the case of the Muskat problem,  motion of the fluid interface   $\partial\Omega_t$ determines the full dynamic, so only  $u$ on it is relevant.

In two spatial dimensions, which is the case considered in the present paper, these models have been studied extensively on the whole plane $\bbR^2$ during the last two decades and we refer the reader to the reviews \cite{GanRev, GraLaz} for an overview of the  literature.
 In particular, local regularity for the Muskat problem with arbitrarily large initial data was proved in  \cite{CorGanMuskat} in the class of fluid interfaces that are graphs of functions from $H^3(\bbR)$, as long as the lighter fluid lies above the heavier one (i.e., in the {\it stable regime}, when the Rayleigh-Taylor condition \cite{Rayleigh, SafTay} is satisfied).  Later, such results were obtained for (less regular) interfaces from $W^{2,p}(\bbR)$ for $p>1$ \cite{CGSV}, $H^2(\bbR)$  \cite{Matioc}, $H^{3/2+\eps}(\bbR)$  \cite{Matioc2},  $W^{s,p}(\bbR)$ for $p>1$ and $s\in(1+\frac 1p,2)$ \cite{AbeMat}, $\dot H^1(\bbR)\cap \dot H^{3/2+\eps}(\bbR)$  \cite{AlaLaz}, as well as from the critical spaces $H^{3/2}(\bbR)$ \cite{AlaNgu, AlaNgu2, AlaNgu4} and $C^1(\bbR)\cap L^2(\bbR)$ \cite{CNX} (we note that spaces $\dot W^{1+1/p,p}(\bbR)$ are critical in this setting).  Meanwhile, stable regime global well-posedness results for small initial data in various spaces appeared in \cite{CGS,SCH,CCGS,CGSV,GGNP, GGHP, Nguyen, Cameron, CorLaz}.  Many other works also studied the Muskat problem in  settings  with surface tension and for fluids with distinct viscosities (see, e.g.,  \cite{EscMat,CCG, Matioc, Matioc2, AbeMat}).
 
 We note that the results in \cite{AlaLaz,GGNP, GGHP} do not require solutions to vanish as $|x|\to\infty$, although they do have to ``locally'' flatten there (this is also the case for the global existence of weak solutions result for monotone $W^{1,\infty}(\bbR)$ initial data from \cite{DLL}).  However, the main result of the preprint \cite{Cameron2}  is global well-posedness for the Muskat problem on $\bbR^3$ with initial data that are small in $\dot W^{1,\infty}(\bbR^2)$ but may grow sublinearly as $|x|\to\infty$.  Hence, these solutions do not need to flatten in any sense (nor be periodic, see  \cite{CorGanMuskat, SCH}), and taking initial data depending only on one variable yields the same result on $\bbR^2$.

 When the interface is in the {\it unstable regime}, with the heavier fluid lying above the lighter one in some region, the Muskat problem was proved to be ill-posed  in Sobolev spaces in \cite{CorGanMuskat} (unstable regime solutions forming a finite time singularity were showed to exist earlier in  \cite{SCH}).
  The question whether an interface can pass from the stable regime to the unstable one was answered in the affirmative in \cite{CCFGL} (such {\it overturning} was later studied in numerous other papers, both analytically and numerically, see e.g. \cite{BCG, GomGra, CGZ, CGZ2,CCFG}).  This work therefore provided an example of  finite time breakdown of local  well-posedness for the Muskat problem on $\bbR^2$, even though the constructed interface remains smooth up to the time when it develops a vertical tangent line.  In fact, because it is an analytic planar curve, one can continue this solution uniquely past that time, but only in the space of analytic curves.   Moreover, it was shown in \cite{CCFGL} that all stable regime solutions with initial interfaces from $H^4(\bbR)$
instantly become analytic, and lower order instant smoothing for less regular initial data (with corners) was also shown to hold in the stable regime \cite{GGNP, GGHP}.  On the other hand, existence of finite time singularities for the Muskat problem on $\bbR^2$ in the stable regime remains an outstanding open problem.

%

However, since aquifers typically sit on top of (or in-between) impermeable rocky layers, of  particular importance are the cases when the domain is the half-plane 
$\bbR\times\bbR^+$ or the strip $\bbR\times(0,l)$ for some $l>0$.  On these domains the fluid velocity must also satisfy the {\it no-flow boundary condition} $u_2\equiv 0$ on the flat bottom $\bbR\times\{0\}$ (as well as on the top $\bbR\times\{l\}$ in the strip case).  Moreover, these settings  allow one to model the physically relevant and dynamically interesting scenario of {\it invasion} by one fluid of a region initially occupied by another fluid.  In this case we have $\{a\le x_1\le b\}\subseteq \Omega_{0}$ for some interval $[a,b]$ and the heavier fluid may then flow into the region $\{a\le x_1\le b\}$ underneath the lighter fluid, along the impermeable bottom layer  (and on the strip, the lighter fluid can similarly invade a region occupied by the heavier one along the  top boundary).

While well-posedness for the Muskat problem on the half-plane or on the strip follows via straightforward adjustments of the whole plane proof when the fluid interface  stays away from the domain boundary, all the relevant estimates blow up as this distance decreases.  (We note that the currently available local regularity results for IPM on strips apply only to solutions that are constant on the top and bottom \cite{CCL}, which is analogous to the Muskat interface not touching the boundary.  In fact, vanishing of higher normal derivatives of the solutions on the boundary  is required in \cite{CCL}, which is needed to avoid complicated boundary layer terms even for solutions that are just constant on the boundary \cite{DGL}.)  In particular, the important invasion scenarios above cannot be modeled using existing results.  Our first motivation is to remedy this by showing local well-posedness for the Muskat problem on $\bbR\times\bbR^+$ and on $\bbR\times(0,l)$ for all sufficiently smooth initial interfaces, including those that touch the domain boundary. We do this in Theorems~\ref{T.1.1} and \ref{T.1.3} below.

Our second motivation is to demonstrate existence of {\it stable regime singularity formation} for the fluid interface.  Although the interface overturning examples on $\bbR^2$ in \cite{CCFGL} involve breakdown of local well-posedness in Sobolev spaces, and the fluid interface acquires infinite slope in finite time, it  does not develop a singularity as a curve at that time.  
Nevertheless, it was proved in  \cite{CCFG} that such interfaces can lose analyticity (and even $C^4$ regularity) in finite time, although this happens in the unstable regime (see \cite{CGZ2} for examples of interfaces that remain analytic and eventually return into the stable regime).  
%

In contrast, we  show in the companion paper \cite{ZlaMuskatBlowup}, using the main results of the present paper, that $H^3$ fluid interfaces on the half-plane can indeed develop stable regime finite time singularities, even for arbitrarily small initial data.  The mechanism for this is invasion by the heavier fluid of a region occupied by the lighter fluid along the impermeable bottom layer, described above, specifically when the invasion proceeds from both directions.  In this scenario,  arguments in \cite{ZlaMuskatBlowup} suggest that the fluid interface develops a singularity at the meeting point of the two invading ``fronts''.  We note that $H^3$ interfaces (and hence also these invading fronts) must have zero contact angles where they touch the boundary, but this is a natural (pre-singularity) assumption in the stable regime.  In fact,  if an initial interface is the graph of a function that equals $\psi(x):=\max\{\alpha x,0\}$ for $|x|\le 1$ (with $\alpha>0$) and is smooth and bounded elsewhere, a simple computation of the corresponding velocity \eqref{1.2} via formulas \eqref{1.3} and \eqref{1.4b} below shows that it satisfies  $u(x,\psi(x))= \frac{(\rho_1-\rho_0)\alpha^2}{\pi(1+\alpha^2)} (\ln x,0) + O(1)$ as $x\to 0$.  This means that even if we had a local well-posedness theory for such initial data, the contact angle of the corresponding solution would have to vanish instantaneously.  However, the small data singularity result in \cite{ZlaMuskatBlowup} and the proof of our main result below (vs.~the one in  \cite{CorGanMuskat}) suggest a potential low regularity theory to be much more involved
 than on the plane.

Finally, our third motivation is to generalize existing theory for the Muskat problem on $\bbR^2$ to fluid interfaces that need not be spatially periodic or vanishing/flattening at $\pm\infty$, which is the current state of the art except for the small data result in \cite{Cameron2}.  To better model real world situations, one should be able to include at least general bounded interfaces, with uniformly bounded Sobolev norms on unit spatial intervals.  We do this here and go even further, allowing interfaces with $O(|x|^{1-})$ growth as $|x|\to\infty$. 
This is an optimal result (see Remark 2 after Theorem \ref{T.1.1}), and we establish it  on the whole plane, half-plane, as well as on strips (of course, in the latter case all solutions are bounded).  In addition, we obtain here maximum principles for such fluid  interfaces on all these domains.  

The treatment of such general classes of interfaces will require proper definitions of norms of the functions involved.  We provide these next, after presenting the interface PDE for the Muskat problem on the half-plane, and then state our main results.

\vskip 3mm
\noindent
{\bf Basic setup.} 
We will perform all of our analysis on the half-plane  $\bbR\times\bbR^+$.  The whole plane case, which is a greatly simplified version of the former, and the strip case, which will add some extra difficulties, will be treated in the last section.
From \eqref{1.2} we see that the fluid {\it vorticity} $\nabla^\perp\cdot u$ equals $\rho_{x_1}$, where we use the convention $\nabla^\perp:=(\partial_{x_2},-\partial_{x_1})$ (which is the negative of the usual definition but it will be more convenient here).
It follows that
\begin{equation} \lb{1.3}
u({\bf x},t):= \nabla^\perp \Delta^{-1} \rho_{x_1}({\bf x},t) = \frac 1{2\pi} \int_{\bbR\times\bbR^+} \left( \frac{({\bf x}-{\bf y})^\perp}{|{\bf x}-{\bf y}|^2} - \frac{({\bf x}-\bar {\bf y})^\perp}{|{\bf x}-\bar {\bf y}|^2} \right) \rho_{x_1}({\bf y},t) \, d{\bf y},
\end{equation}
where $\Delta$ is the Dirichlet Laplacian on $\bbR\times\bbR^+$, $\bar {\bf y}:=(y_1,-y_2)$, and ${\bf y}^\perp:=(y_2,-y_1)$.

In the stable regime on the half-plane, the interface of the two fluids at time $t\ge 0$ is the graph of some (sufficiently regular) non-negative function $x_1\mapsto f(x_1,t)$, with the lighter fluid lying above the heavier one (so $\Omega_t=\{x_2>f(x_1,t)\}$).  It is then not difficult to derive from \eqref{1.1}--\eqref{1.2} a PDE for the motion of this interface, which we do for the reader's convenience on both the half-plane and the plane in the next section (see Section \ref{S12} for the strip case).  This then obviously fully determines the dynamic of the fluid as long as the interface remains the graph of a (regular-enough) non-negative  function.  We can assume without loss  that $\rho_1-\rho_0=2\pi$ (otherwise we scale time by a factor of $\frac {\rho_1-\rho_0}{2\pi}$), when the resulting PDE is
\begin{equation} \lb{1.5}
f_t(x,t) = PV \int_\bbR \left[ \frac{y\, (f_x(x,t)-f_x(x-y,t))}{y^2+(f(x,t)-f(x-y,t))^2} + \frac{y\, (f_x(x,t)+f_x(x-y,t))}{y^2+(f(x,t)+f(x-y,t))^2} \right]  dy
\end{equation}
(we also replaced $x_1$ by $x$ here). Dropping the second fraction similarly yields the corresponding PDE on the plane, which is \eqref{1.13} below, while the strip case results in \eqref{12.1}.

In all three cases, the regularity level of $f(\cdot,t)$ considered here will be $H^3$, as in \cite{CorGanMuskat}, but  we will only require uniform boundedness of $H^2$-norms of $f_x(\cdot,t)$ on unit  intervals and even allow $O(|x|^{1-})$ growth of $f$ as $|x|\to\infty$.  This motivates the following definitions.
For  $k=0,1,\dots $ define the  ``locally uniform'' norms  (similar to those from \cite{Kato})
\begin{align*}
\|g\|_{\tilde L^2(\bbR)}  := \sup_{x\in\bbR} \|g\|_{L^2([x-1,x+1]) } 
\qquad\text{and}\qquad
 \|g\|_{\tilde H^k(\bbR)}  := \sum_{j=0}^k \|g^{(j)}\|_{\tilde L^2(\bbR)} 
\end{align*}
on the spaces of those $g\in L^2_{\rm loc}(\bbR)$ for which these are finite. Then for  $k\ge 1$ and any $\gamma\in[0,1]$ define the seminorms
\begin{align*}
 \|g\|_{\ddot C^\gamma(\bbR)}  := \sup_{|x- y|\ge 1} \frac{|g(x)-g(y)|}{|x-y|^\gamma} 
 \qquad\text{and}\qquad
 \|g\|_{\tilde H^k_\gamma(\bbR)}  :=  \|g'\|_{\tilde H^{k-1}(\bbR)} 
+ \|g\|_{\ddot C^{1-\gamma}(\bbR)} ,
\end{align*}
which vanish for constant functions.    
In particular, $\|g\|_{\tilde H^k_\gamma(\bbR)}<\infty$ allows $g$ to have $O(|x|^{1-\gamma})$ growth at $\pm\infty$.
Our reason for using $1-\gamma$ here is that for all $g\in L^2_{\rm loc}(\bbR)$ we now have
\beq\lb{1.6a}
 \|g\|_{\tilde H^{k}_{\gamma'}(\bbR)}  \le \|g\|_{\tilde H^{k}_\gamma(\bbR)} 
+ \|g'\|_{L^\infty(\bbR)}  \le 2  \|g\|_{\tilde H^{k}_\gamma(\bbR)}
\eeq
when $k\ge 1$ and $0\le\gamma'\le\gamma\le  1$, which agrees with a similar inequality for the norms $\|\cdot\|_{C^{k,\gamma}}$ (and will also be convenient  for the analogous continuous seminorms from Section \ref{S2}).

\vskip 3mm
\noindent
{\bf Main results.} 
We can now state our main results.  We do this  for the half-plane case in the first two theorems, and then extend them to the plane and strip cases in the third one.  The first result includes local well-posedness for \eqref{1.5} as well as an important blow-up criterion.

\begin{theorem} \lb{T.1.1}
If $\gamma\in(0,1]$ and $\|\psi\|_{\tilde H^3_\gamma(\bbR)}<\infty$ for some $\psi\ge 0$,  there is $\gamma$-independent
\beq\lb{1.10}
T_\psi \ge C_\gamma \min \left\{ \|\psi\|_{\tilde H^3_\gamma(\bbR)}^{-4}, \,1+ \left| \ln \|\psi\|_{\tilde H^3_\gamma(\bbR)} \right| \right\}
\eeq
(with some $C_\gamma>0$ depending only on $\gamma$)
such that the following hold.  

(i) 
There is a  classical solution $f\ge 0$ to \eqref{1.5} on $\bbR\times[0,T_\psi)$ with $f(\cdot,0)\equiv\psi$
such that 
\beq \lb{1.7}
\sup_{t\in[0,T]} \|f(\cdot,t)\|_{\tilde H^3_\gamma(\bbR)} <\infty
\eeq
and 
\beq \lb{1.7a}
\sup_{t\in[0,T]} \|f_t(\cdot,t)\|_{W^{1,\infty}(\bbR)}<\infty
\eeq
 for any $T\in[0,T_\psi)$.  And if $T_\psi<\infty$, then for each $\gamma'\in(0,1]$ we have
\beq \lb{1.8}
\int_0^{T_\psi} \|f_x(\cdot,t)\|_{C^{1,\gamma'}(\bbR)}^4 \, dt =\infty.
\eeq

(ii)
For each $T\in[0,T_\psi)$, $f$ from (i) is the unique non-negative classical solution to \eqref{1.5} on $\bbR\times[0,T]$ that satisfies \eqref{1.7}.  We also have
\beq \lb{1.9}
\sup_{t\in[0,T]} \|f(\cdot,t)-\psi\|_{\tilde H^3(\bbR)} <\infty,
\eeq
and if there are $\psi_{\pm\infty}\in\bbR$ such that
$\psi-\psi_{\pm\infty}\in H^3(\bbR^\pm)$, then even 
\beq \lb{1.11}
\sup_{t\in[0,T]}  \|f(\cdot,t)-\psi\|_{H^3(\bbR)}  <\infty.
\eeq
\end{theorem}

{\it Remarks.}
1.  The last claim in (i) means that $f$ develops a singularity at $T_\psi$ if the latter is finite.  This could be either overturning, when $f_x$ blows up, or an interface curve singularity.
\smallskip

2. One may ask whether these results extend to the case $\gamma=0$, when even  linear growth of $f(\cdot,t)$ is allowed at $\pm\infty$.  If, e.g.,  $\psi'=\chi_{[0,\infty)}$ outside of some bounded interval, it is easy to see that the PV integral in \eqref{1.5} is not defined at any $x\in\bbR$ with $\psi'(x)\neq 0$ when $t=0$, while the one in \eqref{1.13} is not defined at $x$ with $\psi'(x)\neq -1$.  
This shows that Theorem \ref{T.1.1} is optimal with respect to the power of the allowable growth of $f(\cdot,t)$ at $\pm\infty$.
\smallskip

3. The last claim in (ii) is a version of \cite[Theorem 1.2]{DLL}.
\smallskip

Our second main result contains several important  estimates for \eqref{1.5}.  These include an $L^\infty$ maximum principle for $f$ (which was proved in \cite{CorGanMuskat} for $H^3(\bbR)$ interfaces on the whole plane), the claim that the interface cannot ``peel off''  the bottom without a loss of regularity, and an estimate on the difference of two solutions that are initially close on a large interval.  The latter result, which we will  use to obtain the maximum principle on strips, uses norms $\|\cdot\|_{C^{2,\gamma}_\gamma}$ defined in Section \ref{S2} as well as
\[
\|g\|_{\tilde L^{2,\mu}(\bbR)} :=  \left\|   \sqrt{1+(\mu-|x|)_+} \, g(x) \right\|_{\tilde L^2(\bbR)}.
\]

\begin{theorem} \lb{T.1.2}
Let $\gamma,\psi,f$ be as in Theorem \ref{T.1.1}.

(i)
Then $\sup f(\cdot,t)$ is non-increasing and $\inf f(\cdot,t)$ is non-decreasing on $[0,T_\psi)$.

(ii)
If $\inf \psi=0$, then $\inf f(\cdot,t)=0$ for all $t\in[0,T_\psi)$. 

(iii)
If $\|\tilde \psi\|_{\tilde H^3_\gamma(\bbR)}<\infty$ and  $\tilde f\ge 0$ is a solution to \eqref{1.5} with $\tilde f(\cdot,0) \equiv \tilde \psi\ge 0$, then
\[
\left\| f(\cdot,t)-\tilde f(\cdot,t) \right\|_{\tilde L^{2,\mu}(\bbR)} \le C_\gamma \exp \left[  C_\gamma \int_0^t \left(1+  \|f(\cdot,s)\|_{C^{2,\gamma}_\gamma}^2 + \|\tilde f(\cdot,s)\|_{C^{2,\gamma}_\gamma}^2 \right) ds \right] \left\| \psi-\tilde \psi \right\|_{\tilde L^{2,\mu}(\bbR)}
\]
holds for each 
$\mu\ge 0$ and $t\in[0,\min\{T_\psi,T_{\tilde \psi}\})$, where $C_\gamma$ only depends on $\gamma$.
\end{theorem}

{\it Remarks.}
1. Since we have \eqref{1.7}, part
(iii) also yields a local $L^\infty$ bound on $f-\tilde f$.
\smallskip

2. In \cite{ZlaMuskatBlowup} we also obtain an $L^2$ maximum principle for  $f$ and an $L^\infty$ maximum principle for $f_x$, which were proved for $H^3(\bbR)$ solutions on the whole plane in \cite{CorGanMuskatMax,CCGS}.  These will be the key ingredients in the proof of existence of finite time stable regime singularities for \eqref{1.5}.
\smallskip

We finally extend the above results to the Muskat problem on the plane and on horizontal strips.  The derivation of \eqref{1.5} in Section \ref{S2} also gives
\begin{equation} \lb{1.13}
f_t(x,t) = PV \int_\bbR  \frac{y\, (f_x(x,t)-f_x(x-y,t))}{y^2+(f(x,t)-f(x-y,t))^2}  dy
\end{equation}
in the former case (with \eqref{1.3} now involving the Laplacian on $\bbR^2$).
Similarly, in the case of the strip $\bbR\times(0,l)$ we obtain
\begin{equation} \lb{12.1}
f_t(x,t) = \sum_\pm PV \int_\bbR (f_x(x,t)\pm f_x(x-y,t))\, \Theta_l \big(y,f(x,t)\pm f(x-y,t)\big)  \, dy, 
\end{equation}
where
\begin{equation} \lb{12.1a}
\Theta_l(y,r):=\frac \pi{2l}\, \frac{\sinh \frac{\pi y}l} {\cosh \frac{\pi y}l - \cos \frac {\pi r}l}
\eeq
(see Section \ref{S12} below for the derivation of \eqref{12.1}, which is from \cite{CorGraOri}).
Our main results extend to both these PDE with the natural adjustments.

\begin{theorem} \lb{T.1.3}
(i)
Theorem \ref{T.1.1} and Theorem \ref{T.1.2}(i,iii) hold for \eqref{1.13} in place of \eqref{1.5}, without requiring that $f,\psi,\tilde f, \tilde\psi\ge 0$.

(ii) 
Theorems \ref{T.1.1} and \ref{T.1.2} hold for \eqref{12.1} in place of \eqref{1.5}, with $0\le f,\psi,\tilde f, \tilde\psi\le l$.  In addition, if $\sup \psi=l$, then $\sup f(\cdot,t)=l$ for all $t\in[0,T_\psi)$. 
\end{theorem}

{\it Remark.}
Note that Theorem \ref{T.1.3}(i) includes existing results from \cite{CorGanMuskat} for $\psi\in H^3(\bbR)$ (see \eqref{1.11}) and for periodic  $\psi\in H^3_{\rm loc}(\bbR)$, since in the latter case uniqueness of solutions implies periodicity of $f(\cdot,t)$ for each $t\in[0,T_\psi)$ (with the same period).

\vskip 3mm
\noindent
{\bf Discussion of the proof of Theorem \ref{T.1.1} and organization of the paper.} 
Most of the rest of this paper, namely Sections \ref{S2}--\ref{S10}, forms the proof of Theorem \ref{T.1.1}.  
When one works in $H^3(\bbR)$ and on the plane (so without a bottom and $f(\cdot,t)$ vanishes at $\pm\infty$), local well-posedness for the Muskat problem in the stable regime can be obtained via a simple adjustment of a proof from \cite{CorGanMuskat} (where in fact local well-posedness in $H^4(\bbR^2)$ was obtained for the Muskat problem on $\bbR^3$).  The crux of that proof is an estimate on the rate of change of the $L^2$ norm of $f_{xxx}(\cdot,t)$.  

In the case of Theorem \ref{T.1.1}, the crucial argument will be an analogous estimate, this time involving  {\it local} $L^2$ norms of $f_{xxx}(\cdot,t)$ and including terms corresponding to the second fraction in \eqref{1.5}.  The former requires introduction of appropriate cut-off functions $h$, and then we split the formula for the rate of change of the resulting local norms into eight integrals $I_j^\pm$, four corresponding to each term in \eqref{1.5} (see  \eqref{2.0b} below).

In the proof in \cite{CorGanMuskat}, which only involved the $I_j^-$ integrals and no cut-off functions, the main challenge was to estimate  the integral corresponding to our $I_1^-$.  This is because it  includes the most singular integrand (involving  $\partial_x^4 f$), and we make use of the arguments from \cite{CorGanMuskat} when estimating $I_1^-$ and the other integrals $I_j^-$ originating from the first fraction in \eqref{1.5}.  Our inclusion of the cut-offs $h$, and potential growth of $f(\cdot,t)$ at $\pm\infty$, unsurprisingly introduce a number of non-trivial  complications in these  arguments as well as in the rest of our proofs.  (We note that it is not possible to use the result from \cite{CorGanMuskat} to obtain Theorem~\ref{T.1.1} on the plane via some approximation scheme, even for just $\tilde H^3(\bbR)$ interfaces with no growth as $x\to\pm \infty$.)

However, working on the half-plane  increases the difficulty level much more dramatically.
Note that if $\inf_{x\in\bbR} f(x,0)>0$, then the second fraction in \eqref{1.5} is sufficiently regular and the local well-posedness proof on $\bbR^2$  extends to this case, but the obtained constants {\it blow up} as $\inf_{x\in\bbR} f(x,0)$ approaches 0.  As a result, this approach fails when  $\inf_{x\in\bbR} f(x,0)= 0$, which is the most important case, as we explained above.  
In our proof,  integrals $I_j^+$, which originate from the second fraction in \eqref{1.5}, will be much more challenging to estimate in the region where $f$ is small.  In fact, estimates on $I_3^+$ and $I_4^+$ (whose counterparts $I_3^-$ and $I_4^-$ are easily bounded) form over two thirds of the proof of our main a priori bound \eqref{2.3} in  Sections~\ref{S3}--\ref{S6}, even though these terms do not involve the highest derivatives of $f$!  

The reason for this  is the following.  When $f$ is a smooth function, the inside integrands in all the integrals $I_j^-$ are bounded functions of $y$ because their numerators vanish to at least  the same degree as their denominators.  Since $f_{xxx}(\cdot,t)$ is only $L^2_{\rm loc}$ and the integrands involve derivatives of order up to 4, obtaining estimates on these integrals in terms of only $f_{xxx}(\cdot,t)$ and lower order derivatives is far from trivial and requires various symmetrization arguments as well as leveraging of cancellations coming from adding the integrands for $y$ and $-y$ (for $y>0$).  On the other hand, denominators of the inside integrands in $I_j^+$ only vanish when $f(x,t)=0$, and in that case such cancellations still work because $f_x(x,t)=0$  (since $f\ge 0$).  What makes these integrals much more difficult to estimate is the fact that their integrands become very large (even after the $y$ vs.~$-y$ cancellation) when $f(x,t)$ is {\it close to  but not equal to 0} and $y$ is small.  And this is even more pronounced for larger $j$ (when the numerators and denominators include more factors), due to all the $\pm$ signs being $+$ and hence most of the factors not vanishing at $y=0$, which is why estimates on $I_3^+$ and $I_4^+$ will be the most involved.  

Related to (and an illustration of) this is also the following observation.  
In Section \ref{S11} we show that
 \eqref{1.5} is the same as
\begin{equation} \lb{1.6}
f_t(x,t) = \pi \chi_{(0,\infty)}(f(x,t)) +  \sum_\pm PV \int_\bbR  \frac{y\, f_x(x,t)-(f(x,t)\pm f(x-y,t))}{y^2+(f(x,t)\pm f(x-y,t))^2}  dy .
\end{equation}
The derivation of \eqref{1.6} shows that the last integral with $\pm$ being $+$, which comes from the second fraction in \eqref{1.5}, is {\it discontinuous} at all $x\in\partial\{f(\cdot,t)=0\}$, even when $f\in C_c^\infty(\bbR)$.  This also suggests that proving local well-posedness for \eqref{1.5} via \eqref{1.6} would likely  be even more complicated than our proof below.  Nevertheless, we do use \eqref{1.6} when proving Theorem \ref{T.1.2}(i) and a maximum principle for $f_x$ in \cite{ZlaMuskatBlowup}.

It follows from \eqref{2.13a} and \eqref{2.7} below that the integrand in $I_3^+$ is $O\left(\frac {|y|f(x,t)}{\max\{|y|, f(x,t)  \}^4}\right)$ when $f(x,t)\le 1$, $f_{xx}(x,t)$ is not small, and $|y|\le \sqrt{f(x,t)}$ (the integrand in $I_4^+$ is similar, albeit more complex), and hence it is far from being integrable uniformly in the size of $f(x,t)$.  
Even after adding the integrand at $y$ and $-y$, it remains
too large as we explain just before \eqref{2.34}.
However, we are able to identify here
somewhat unexpected 
additional cancellations in these integrals.  Namely, after isolating (appropriately defined)  leading order terms in the integrands, we find these terms to be $y$-derivatives of certain special functions with large variations (see \eqref{2.34}, \eqref{2.45}, and \eqref{2.61} below). Integrating these allows us to  bound both $I_3^+$ and $I_4^+$ in a way that ultimately yields the crucial a priori estimate \eqref{2.3}. 

 With \eqref{2.3} in hand, we  obtain analogous estimates for an approximating family of more regular $\eps$-dependent models  \eqref{7.1} in Section \ref{S7}.  In Sections \ref{S8} and  \ref{S9} we then derive related bounds on the rate of change of local $L^2$ norms of differences of two solutions to either \eqref{1.5} or \eqref{7.1} (these rates for single solutions may not be bounded due to possible growth of $f$ as $x\to\pm\infty$).  This proves uniqueness for \eqref{1.5}, and we then use all these estimates in Section~\ref{S10} to obtain local-in-time $\tilde H^3_\gamma$ solutions to \eqref{1.5} as $\eps\to 0$ limits of solutions to \eqref{7.1} (the latter are found via fairly standard methods), as well as finish the proof of Theorem~\ref{T.1.1}.  

Finally, the proofs of Theorems \ref{T.1.2} and \ref{T.1.3} appear in Sections \ref{S11} and \ref{S12}, respectively.  Having the proof of Theorem \ref{T.1.1}, these will involve mostly  straightforward arguments, although the extension of Theorem \ref{T.1.1} to horizontal strips  will require some care.



\section{Derivation of \eqref{1.5} and the Start of the Proof of Theorem \ref{T.1.1}} \lb{S2}

In this and all the following sections,  $C$ will  be some universal constant that may change from line to line and always depends only on $\gamma$.  We will also assume without loss that $\gamma\in(0,\frac 12]$, so that \eqref{2.0a} below holds.  We can do this because if we apply Theorem \ref{T.1.1} with  some $\gamma'<\gamma$ in place of $\gamma$ (using \eqref{1.6a} with $g:=\psi$), then \eqref{1.9} guarantees that \eqref{1.7} also holds. 

Before starting the proof of Theorem \ref{T.1.1}, let us derive \eqref{1.5} and show that this derivation is valid in the setting considered here. 

\vskip 3mm
\noindent
{\bf Derivation of \eqref{1.5} and \eqref{1.13}.} 
From \eqref{1.3a} we obtain
\beq\lb{1.4b}
\rho_{x_1} (\cdot,t)= (\rho_1-\rho_0) \frac{f_{x}(P(\cdot),t)}{\sqrt{1+f_{x}(P(\cdot),t)^2}} \delta_{\Gamma_f(t)},
\eeq
where $\Gamma_f(t):=\{(x,f(x,t))\,|\,x\in\bbR\}$ is the graph of $f(\cdot,t)$ and $P(x_1,x_2):=x_1$.  This 
and 
assuming that ${\rho_1-\rho_0}={2\pi}$ (which we do without loss)
now turns \eqref{1.3} into
\begin{equation} \lb{1.4}
u(x,f(x,t),t) =  PV \int_\bbR \sum_\pm \left[ \mp \frac{(f(x,t)\pm f(y,t), y-x)}{(x-y)^2+(f(x,t)\pm f(y,t))^2} 
\right] f_{x}(y,t) \, dy
\end{equation}
for points on the curve $\Gamma_f(t)$ at each time $t\ge 0$.
We note that the tangential component of $u(\cdot,t)$ at $\Gamma_f(t)$ (as a function on $\bbR^2$) is in fact discontinuous at all points $(x,f(x,t))$ with $f_x(x,t)\neq 0$ (it has a jump of size $(\rho_1-\rho_0)\frac{|f_{x}(x,t)|}{\sqrt{1+f_{x}(x,t)^2}} $ across the curve), but the normal component is continuous when $f_x(\cdot,t)$ is and it equals the normal component of \eqref{1.4}.  Since the tangential component does not affect the motion of the curve $\Gamma_f$  as a set, which fully determines the overall evolution, this discrepancy can be neglected.

If we want to recast this evolution in terms of a PDE for $f$, we can do that by adding  some multiple of the tangent vector $(1,f_x(x,t))$ to \eqref{1.4}, which again does not change the motion of $\Gamma_f$ provided $f(\cdot,t)$ is continuously differentiable.  We choose this multiple  so that the resulting vector's first coordinate vanishes.  That is,  we replace $u$ by
\begin{align*}
v(  x, f(x,t),t):=  u(x,f(x,t),t) 
 +   PV \int_\bbR  \sum_\pm  \frac{x-y}{(x-y)^2+(f(x,t)\pm f(y,t))^2} 
 dy \, (1,f_x(x,t)),
\end{align*}
whose first coordinate is
\[
- \frac 12 \, PV \int_\bbR   \sum_\pm \frac d{dy} \log \left[ (x-y)^2+(f(x,t)\pm f(y,t))^2 \right]
 dy =0.
\]
Hence 
$\Gamma_f$ for $f(\cdot,t)\ge 0$ moves with velocity $u$ precisely when $f_t(x,t)=v_2(x,f(x,t),t)$ for all $(x,t)\in\bbR\times\bbR^+$, which after the change of variables $y\leftrightarrow x-y$ becomes \eqref{1.5}.  Note also that this change of variables and \eqref{2.1} below show that the multiple of $(1,f_x(x,t))$ that we added above is in fact bounded when $\|f(\cdot,t)\|_{C^2_\gamma(\bbR)}<\infty$ (see the definition below), which holds in the setting of Theorem \ref{T.1.1}.  In the case of the whole plane, this argument similarly gives \eqref{1.13}.

\vskip 3mm
\noindent
{\bf Continuous norms of $f$ with power growth.}   
For  $k=0,1,\dots $ and $\gamma\in[0,1]$ let 
\begin{align*}
 \|g\|_{\dot C^\gamma(\bbR)}  := \sup_{x\neq y} \frac{|g(x)-g(y)|}{|x-y|^\gamma} 
 \qquad\text{and}\qquad
 \|g\|_{C^{k,\gamma}(\bbR)}  :=  \sum_{j=0}^k \|g^{(j)}\|_{L^\infty(\bbR)}+ \|g^{(k)}\|_{\dot C^{\gamma}(\bbR)}
\end{align*}
when $g\in L^2_{\rm loc}(\bbR)$
(note that then $C^{k,0}(\bbR)=W^{k,\infty}(\bbR)$),
and for $k\ge 1$ let
\begin{align*}
 \|g\|_{C^{k,\gamma}_\gamma(\bbR)}  := \|g'\|_{C^{k-1,\gamma}(\bbR)} 
+ \|g\|_{\ddot C^{1-\gamma}(\bbR)}  
 \qquad\text{and}\qquad
\|g\|_{C^{k}_\gamma(\bbR)}:=\|g'\|_{C^{k-1,0}(\bbR)} 
+ \|g\|_{\ddot C^{1-\gamma}(\bbR)} .
\end{align*}
All but the second of these are  seminorms that vanish on constant functions.  
Then 
\begin{align*}
\|g\|_{C^{k,\gamma'}_{\gamma'}(\bbR)}  \le \|g\|_{C^{k,\gamma}_\gamma(\bbR)} 
+ \|g'\|_{L^\infty(\bbR)} + 2\|g^{(k)}\|_{L^\infty(\bbR)} \le 4  \|g\|_{C^{k,\gamma}_\gamma(\bbR)} 
\end{align*}
again holds when $k\ge 1$ and $0\le\gamma'\le\gamma\le  1$,
and we have $\|g \|_{C^{k,0}_{0}(\bbR)}<\infty$  iff $ g'\in W^{k-1,\infty}(\bbR)$.  Finally, we note that when $\gamma\in(0,\frac 12]$ (which we assume here) and $k\ge 1$, then
\beq\lb{2.0a}
\|g\|_{C^{k-1,\gamma}_\gamma(\bbR)}\le C_{k,\gamma} \|g\|_{\tilde H^{k}_\gamma(\bbR)}
\eeq
holds for all $g\in L^2_{\rm loc}(\bbR)$, with some constant $C_{k,\gamma}<\infty$.

\vskip 3mm
\noindent
{\bf Basic bounds on $f$ and \eqref{1.5}.}    
We start the proof of Theorem \ref{T.1.1} by showing that the right-hand side of \eqref{1.5} remains bounded as long as $\|f(\cdot,t)\|_{\tilde H^3_\gamma(\bbR)}<\infty$.
Since most of our arguments  apply at some fixed $t$, we will mostly drop $t$ from the notation for the sake of simplicity.   We then denote $g':=\partial_x g$, and all (semi)norms below will involve functions defined on $\bbR$ unless we specify otherwise.

We first claim that
\beq\lb{2.1}
\left| PV \int_{S\cup(-S)} \frac{y}{y^2+(f(x)\pm f(x-y))^2} \, dy \right| \le C \|f\|_{C^2_\gamma} .
\eeq
for any $S\subseteq[0,\infty)$.  The left-hand side is clearly  bounded by
\beq\lb{2.4a}
\int_S \frac{y|f(x+y)- f(x-y)| \, |f(x+y)+ f(x-y) \pm 2f(x)| }{[y^2+(f(x)\pm f(x+y))^2][y^2+(f(x)\pm f(x-y))^2]} \, dy .
\eeq
When $\pm$ is $-$, from $|f(x+y)- f(x-y)|\le |f(x+y)- f(x)|+|f(x-y)- f(x)|$ we see that 
\[
y|f(x+y)- f(x-y)| \le  [y^2+(f(x) - f(x+y))^2]^{1/2}[y^2+(f(x)- f(x-y))^2]^{1/2},
\]
so \eqref{2.4a} is indeed no more than 
\[
2 \|f''\|_{L^\infty} 
+  2\|f\|_{\ddot C^{1-\gamma}}  \int_1^\infty y^{-1-\gamma}\,dy 
\le C \|f\|_{C^2_\gamma} .
\]
 When $\pm$ is $+$, the same estimate holds for
 \[
 \int_S \frac{y|f(x+y)- f(x-y)| \, |f(x+y)+ f(x-y) - 2f(x)| }{[y^2+(f(x)+ f(x+y))^2][y^2+(f(x)+ f(x-y))^2]} \, dy,
 \] 
 while the remainder of \eqref{2.4a}  is bounded by
 \[
 \|f'\|_{L^\infty} \int_S \frac{ 4 |f(x)| }{y^2+f(x)^2} \, dy \le C  \|f'\|_{L^\infty}
 \]
because $f>0$.  This also yields \eqref{2.1} in this case.
We also have
\begin{align*}
\left|  \int_{-1}^1 \frac{y\, (g'(x)- g'(x-y))}{y^2+(f(x)\pm f(x-y))^2} \, dy \right|
\le C\|g'\|_{\dot C^{\gamma}},
\end{align*}
and
integration by parts in $y$ yields
\begin{align}
 PV & \int_{|y|\ge 1}  \frac{y\, g'(x-y)}{y^2+(f(x)\pm f(x-y))^2} \, dy  
 =  \frac{g(x-1)-g(x)}{1+(f(x)\pm f(x-1))^2} + \frac{g(x+1)-g(x)}{1+(f(x)\pm f(x+1))^2}  \notag
\\ & +  \int_{|y|\ge 1} (g(x-y)-g(x)) \frac{ (f(x)\pm f(x-y))^2 \pm 2yf'(x-y)(f(x)\pm f(x-y)) -y^2}{[y^2+(f(x) \pm f(x-y))^2]^2} \, dy.  \lb{2.1a}
\end{align}
The last fraction is bounded by $(1+ \|f'\|_{L^\infty})|y|^{-2}$, so
\[
\left|   PV  \int_{|y|\ge 1}  \frac{y\, g'(x-y)}{y^2+(f(x)\pm f(x-y))^2} \, dy  \right| 
\le C (1+ \|f'\|_{L^\infty}) (\|g'\|_{L^\infty} + \|g\|_{\ddot C^{1-\gamma}}).
\]
The above bounds thus imply
\begin{align*}
 \left|  PV \int_{\bbR} \frac{y\,  g'(x-y)}{y^2+(f(x)\pm f(x-y))^2} \, dy \right| \le C  (1+ \|f'\|_{L^\infty}) \|g\|_{C^{1,\gamma}_\gamma},
\end{align*}
which together with \eqref{2.1} yields
\beq \lb{2.70}
 \left|  PV \int_{\bbR} \frac{y\, ( g'(x)\pm g'(x-y))}{y^2+(f(x)\pm f(x-y))^2} \, dy \right| \le C (1+ \|f'\|_{L^\infty}) \|g\|_{C^{1,\gamma}_\gamma} + C \|f\|_{C^2_\gamma}  \|g'\|_{L^\infty}.
\eeq
In particular, 
\beq \lb{2.70a}
 \left|  PV \int_{\bbR} \frac{y\, ( f'(x)\pm f'(x-y))}{y^2+(f(x)\pm f(x-y))^2} \, dy \right| \le C (1+ \|f'\|_{L^\infty}) \|f\|_{C^2_\gamma} .
\eeq

Next we derive a pair of basic bounds on $f\ge 0$ in the region where it is small.  Note that $|f'(z)|\ge \frac 12 {|f'(x)|}$ whenever $|z-x|\le \frac  {|f'(x)|} {2 \|f''\|_{L^\infty}}$, so $f\ge 0$ forces
\[
\frac  {|f'(x)|} {2 \|f''\|_{L^\infty}} \frac  {|f'(x)|} {2} \le |f(x)|.
\]
Hence  
\beq \lb{2.5a}
|f'(x)| \le 2 \|f''\|_{L^\infty}^{1/2} \sqrt{f(x)} \le (1+\|f''\|_{L^\infty}) \sqrt{f(x)},
\eeq
and then for all $|y|\le \sqrt{f(x)}$ we have
\[
|f'(x-y)| \le (1+ 2\|f'\|_{C^1}) \min\left\{\sqrt{f(x)},1 \right\} \qquad\text{and}\qquad f(x-y) \le 2(1+ \|f'\|_{C^1}) f(x).
\]
When $|y|\ge \sqrt{f(x)}$, from \eqref{2.5a} we obtain
\[
|f'(x-y)| \le (1+ 2\|f'\|_{C^1})  \min \{|y|,1\} \qquad\text{and}\qquad f(x-y) \le 2(1+ \|f'\|_{C^1}) y^2 .
\]
So 
we have
\beq \lb{2.7}
\begin{split}
|f'(x-y)| &\le 2(1+ \|f'\|_{C^1}) \max\left\{ y, \sqrt{ f(x)} \right\},
\\  f(x-y) &\le 2(1+ \|f'\|_{C^1}) \max\left\{ y^2, f(x) \right\}
\end{split}
\eeq
for all $x,y\in\bbR$ when $f\ge 0$.  We will use these two estimates extensively below.

\vskip 3mm
\noindent
{\bf A priori estimates for local norms of $f_{xxx}$.} 
 In this and the next four sections, we will assume that $f\ge 0$ and  it is smooth on 
 $\bbR \times[0,T]$ for some $T>0$.
We will eventually apply the {\it a priori} arguments from these sections to  solutions of a mollified version \eqref{7.1} of \eqref{1.5} in Section \ref{S7}, and we will show in Section \ref{S10} that the latter exist and converge to a solution of \eqref{1.5} as the mollification parameter $\eps\to 0$.
We now fix some  $h_0\in C^2(\bbR)$ such that $\chi_{[-1,1]} \le h_0 \le \chi_{[-4,4]}$ and $\max\{\|h'_0\|_{L^\infty}, \|h''_0\|_{L^\infty} \} \le  1$.  Then
\beq\lb{2.0}
\| f_{xxx}(\cdot,t)\|_{\tilde L^2(\bbR)} \le  \sup_{x_0\in\bbR}   \left\| \sqrt{h_0(\cdot-x_0)}\,  f_{xxx}(\cdot,t) \right\|_{L^2(\bbR)}
\eeq
clearly holds for all $t\in[0,T]$, so to estimate $\|f(\cdot,t)\|_{\tilde H^3_\gamma(\bbR)}$, it will mainly be important to obtain appropriate uniform-in-$x_0$ bounds on the right-hand side of \eqref{2.0}.  We do this via the estimate \eqref{2.3} below. We note that an analogous estimate for the left-hand side of \eqref{2.0} seems to {\it fail}, which is why we have to introduce the more regular cut-off functions $h_0(\cdot-x_0)$.
Let us also fix an arbitrary $x_0\in\bbR$ and denote $h:=h_0(\cdot -x_0)$ in the arguments below.   

To estimate the right-hand side of \eqref{2.0}, from \eqref{1.5} (and after dropping $t$) we obtain
\beq\lb{2.0b}
\frac d{dt}  \left\| \sqrt{h}\, f''' \right\|_{L^2}^2 
 = 2( I_1^+ + I_1^- + I_4^+ + I_4^- ) + 6(I_2^+ + I_2^- + I_3^+ + I_3^-)  
\eeq
for smooth $f$, where with $A^\pm(x,y):= [y^2+(f(x)\pm f(x-y))^2]^{-1}$ we have
\begin{align*}
I_1^\pm & := 
\int_\bbRr  h(x) f'''(x) \, PV \int_\bbR \frac{y\, (f''''(x)\pm f''''(x-y))}{y^2+(f(x)\pm f(x-y))^2} \, dydx ,
\\  I_2^\pm & := 
\int_\bbRr h(x)  f'''(x) \, PV \int_\bbR  y\, (f'''(x)\pm f'''(x-y)) \, \partial_x A^\pm(x,y) \, dydx ,
\\  I_3^\pm & := 
\int_\bbRr  h(x) f'''(x) \, PV \int_\bbR  y\, (f''(x)\pm f''(x-y)) \, \partial_{x}^2 A^\pm(x,y) \, dydx ,
\\  I_4^\pm & := 
\int_\bbRr  h(x) f'''(x) \, PV \int_\bbR  y\, (f'(x)\pm f'(x-y)) \, \partial_{x}^3 A^\pm(x,y) \, dydx .
\end{align*}
Note that the PV integrals above all converge as $|y|\to\infty$ (and all but $I_1^\pm$ do so absolutely).  As for convergence  at $y=0$, we see that the integrands are all bounded when $\pm$ is $-$ and $f$ is smooth enough, and this is clearly also true for the ones with $+$ when also $f(x)>0$.  When $f(x)=0$, then $f'(x)=0$ because $f\ge 0$, so $f'(x-y)=O(|y|)$ and $f(x-y)=O(y^2)$. This shows that the inside integrand for $I_2^+$ is still bounded, while the ones for $I_1^+$, $I_3^+$, and $I_4^+$ are  $\frac Cy+O(1)$.  Moreover,  this $C$ is non-zero only when $f''(x)\neq 0$ ($=f(x)$).  Therefore ``PV'' can be removed from $I_2^\pm$, $I_3^-$, and $I_4^-$, and it is only needed in $I_3^+$, and $I_4^+$ at countably many points, which makes it irrelevant (in fact, we will eventually apply the arguments below to $f>0$, 
so then ``PV'' will only be needed in $I_1^\pm$ as $|y|\to\infty$ anyway).    As a result, there will be  few PV integrals in our analysis of the terms $I_2^\pm$, $I_3^\pm$, and $I_4^\pm$.  However, analysis of the integrands in $I_3^+$ and $I_4^+$ when $0<f(x)\ll 1$ (e.g., when $x$ is near a point where $f$ vanishes) will turn out to be the most challenging task in our derivation of  a priori estimates for \eqref{1.5}.

We bound the above eight integrals in the next four sections, and obtain estimates \eqref{2.10}, \eqref{2.11}, \eqref{2.10c}, \eqref{2.14a}, and \eqref{2.66} that collectively yield the a priori estimate
\beq\lb{2.3}
\frac d{dt}  \left\| \sqrt{h}\, f''' \right\|_{L^2}^2 
\le C(1+  \|f\|_{C^{2,\gamma}_\gamma}^4 ) \, (\|f'\|_{\tilde L^2 }^2 + \|f''\|_{\tilde L^2 }^2 + \|f'''\|_{\tilde L^2 }^2) 
\eeq
(which holds uniformly in $x_0$).
Note that we do not prove \eqref{2.3} with the left-hand side multiplied by $-1$, due to \eqref{2.10} below.  In fact, that inequality would hold in the unstable heavier-fluid-above scenario, but \eqref{2.3} then fails as stated and causes ill-posedness for \eqref{1.5} even in $H^3(\bbR)$ in that setting \cite{CorGanMuskat}.
With \eqref{2.3} in hand, we finish the proof of Theorem \ref{T.1.1} in Sections \ref{S7}--\ref{S10} as described at the end of the introduction.



\section{Estimates on  $I_1^\pm$} \lb{S3}

Let us first estimate $I_1^\pm$, denoting $g:=f'''$.  For later reference, we will do this in terms of
\beq\lb{2.4g}
 \|g\|_{\tilde L^2_{x_0}(\bbR)}  := \left[\int_\bbR g(x)^2 \min\{1,|x-x_0|^{-2} \} dx \right]^{1/2}
\qquad (\le 2 \|g\|_{\tilde L^2(\bbR)})
\eeq
rather than $\|g\|_{\tilde L^2}$ (we use the latter in the following sections).
We have $I_1^\pm=J_1^\pm \pm J_2^\pm$,
where
\begin{align*}
J_1^\pm & := \int_\bbRr h(x)  g(x) \, PV \int_\bbR \frac{y\, g'(x)}{y^2+(f(x)\pm f(x-y))^2} \, dydx ,
\\ J_2^\pm &  :=  \int_\bbRr h(x)  g(x) \, PV \int_\bbR \frac{(x-y)\, g'(y)}{(x-y)^2+(f(x)\pm f(y))^2} \, dydx .
\end{align*}
Integration by parts in $x$ yields
\beq \lb{2.4}
\begin{split}
J_1^\pm = & \int_\bbRr  h(x)  g(x)^2 \, PV \int_\bbR \frac{y\, (f(x)\pm f(x-y)) (f'(x)\pm f'(x-y))}{[y^2+(f(x)\pm f(x-y))^2]^2} \, dydx
\\ & - \frac 12 \int_\bbRr  h'(x)  g(x)^2 \, PV \int_\bbR \frac{y}{y^2+(f(x)\pm f(x-y))^2} \, dydx =: K_1^\pm + K_2^\pm.
\end{split}
\eeq
From \eqref{2.1} we see that
\beq\lb{2.4b}
|K_2^\pm| \le C\|f\|_{C^2_\gamma} \left\| \sqrt{|h'|}\, g \right\|_{L^2}^2 \le C\|f\|_{C^2_\gamma} \| g\|_{\tilde L^2_{x_0}}^2.
\eeq


As for $K_1^\pm$, we have
\beq \lb{2.5}
|K_1^\pm| \le C \left( \|f'\|_{L^\infty} + \|M^\pm \|_{L^\infty} \right) \| g\|_{\tilde L^2_{x_0}}^2,
\eeq
where
\[
M^\pm(x): = PV \int_{-1}^1 \frac{y\, (f(x)\pm f(x-y)) \, (f'(x)\pm f'(x-y))}{[y^2+(f(x)\pm f(x-y))^2]^2} \, dy .
\]
Here we estimated both $|y|$ and $|f(x)\pm f(x-y)|$ by $[y^2+(f(x)\pm f(x-y))^2]^{1/2}$, and we will do so multiple times below (however, at other times we will estimate $|f(x)\pm f(x-y)|$ via  \eqref{2.7}, including via its first line when $\pm$ is $-$).
To estimate $M^-$, we note that from
 \[
| f'(x+y)+f'(x-y)-2f'(x)| \le 2\|f''\|_{\dot C^{\gamma}} |y|^{1+\gamma}
 \]
and
\beq\lb{2.6a}
\frac {ab}c- \frac {a'b'}{c'} = \frac {(a-a')b}{c}+\frac{a'(b-b')}{c}+\frac{a'b'(c'-c)}{cc'}
\eeq
we obtain
\begin{align*}
& \left| \frac{ (f(x) - f(x-y)) \, (f'(x) - f'(x-y))}{[y^2+(f(x) - f(x-y))^2]^2} - \frac{ (f(x+y) - f(x)) \, (f'(x+y) - f'(x))}{[y^2+(f(x) - f(x+y))^2]^2} \right|
\\  & \qquad \le  \frac {C \|f''\|_{L^\infty}^2 |y|^3 } {y^4}   + \frac {C  \|f'\|_{ C^{1,\gamma}}^2  |y|^{2+\gamma}} {y^4}   
+ \frac {C  \|f''\|_{L^\infty}^2 |y|^3} {y^4} 
 \le  \frac {C \|f'\|_{C^{1,\gamma}}^2 |y|^{2+\gamma}} {y^4} ,
\end{align*}
for $|y|\le 1$ (we again used $| f(x+y)+f(x-y)-2f(x)| \le \|f''\|_{L^\infty}y^2$).  
Hence
\beq\lb{2.4d}
\|M^- \|_{L^\infty} \le \int_0^1 \frac {C \|f'\|_{C^{1,\gamma}}^2 y^{3+\gamma}} {y^4} \, dy = C \|f'\|_{C^{1,\gamma}}^2 .
\eeq

We can bound $M^+$ via \eqref{2.6a}
and \eqref{2.7}, which in particular imply
\[
|f(x+y)-f(x-y)|\le  4(1+ \|f'\|_{C^1})\sqrt{ f(x)}\,  |y| 
\]
for $|y|\le\sqrt{f(x)}$, and yield for these $y$
\begin{align*}
& \left| \frac{ (f(x) + f(x-y)) \, (f'(x) + f'(x-y))}{[y^2+(f(x) + f(x-y))^2]^2} - \frac{ (f(x) + f(x+y)) \, (f'(x) + f'(x+y))}{[y^2+(f(x) + f(x+y))^2]^2} \right|
\\  & \hskip 5cm  \le  \frac {C(1+ \|f'\|_{C^1}^2) f(x) |y|} {\max\{|y|,f(x)\}^4}  
 \le  \frac {C(1+ \|f'\|_{C^1}^2) f(x) } {\max\{|y|,f(x)\}^3} .
\end{align*}
For $|y|\ge \sqrt{f(x)}$ we instead obtain 
\[
\left| \frac{ (f(x) + f(x-y)) \, (f'(x) + f'(x-y))}{[y^2+(f(x) + f(x-y))^2]^2} \right| \le  \frac {C(1+ \|f'\|_{C^1}^2) |y|^3 }{ |y|^4},
\]
 so it follows that
\beq\lb{2.4e}
|M^+(x)| \le C(1+ \|f'\|_{C^1}^2) \int_{0}^1 \left( \frac{yf(x)}{\max\{y,f(x)\}^3} +1\right) dy  \le C(1+ \|f'\|_{C^1}^2) .
\eeq
This, \eqref{2.4b}, 
\eqref{2.5}, and \eqref{2.4d} thus prove 
\beq \lb{2.8}
|J_1^\pm| \le C(1+ \|f\|_{C^{2,\gamma}_\gamma}^2 ) \|g\|_{\tilde L^2_{x_0}}^2 .
\eeq

Next we estimate $J_2^\pm$.  Integration by parts in $y$ yields
\beq \lb{2.91}
J_2^\pm   =  \int_\bbRr h(x)  g(x) \, PV \int_\bbR \frac{(x-y)\, \partial_y (g(y)-g(x))}{(x-y)^2+(f(x)\pm f(y))^2} \, dydx = K_3^\pm +  K_4^\pm ,
\eeq
where
\begin{align*}
K_3^\pm &: = \int_\bbRr h(x)  g(x) \, PV \int_\bbR \frac{ g(x)-g(y)}{(x-y)^2+(f(x)\pm f(y))^2} \, dydx ,
\\ K_4^\pm &: = - 2\int_\bbRr h(x)  g(x)  \int_\bbR \frac { (g(x)-g(y)) (f(x)\pm f(y)) \big[ f(x)\pm f(y)\pm   f'(y)(x-y) \big] } {[(x-y)^2+(f(x)\pm f(y))^2]^2} \, dydx 
\end{align*}
(note that the term obtained when $\partial_y$ is applied to $x-y$ in the numerator is $-K_3^\pm$, so $K_4^\pm$ is in fact the  term obtained when $\partial_y$ is applied to the denominator minus $2K_3^\pm$).
Symmetrization now shows that  $K_3^\pm:=\frac 12 (L_1^\pm+ L_2^\pm)$, where
\begin{align*}
L_1^\pm & := \int_{\bbR^2} \frac{h(y)( g(x)-g(y))^2}{(x-y)^2+(f(x)\pm f(y))^2} \, dydx ,
\\ L_2^\pm & :=  \int_{ \bbR^2}   \frac{(h(x)-h(y)) g(x) ( g(x)-g(y))}{(x-y)^2+(f(x)\pm f(y))^2} \, dydx
\\ & = \int_{ \bbR} g(x)^2 \, PV \int_{ \bbR}  \frac{h(x)-h(y) }{(x-y)^2+(f(x)\pm f(y))^2} \, dydx
\\ & = \int_{ \bbR} g(x)^2 \, PV \int_{ \bbR}  \frac{h(x)-h(x-y) }{y^2+(f(x)\pm f(x-y))^2} \, dydx
\end{align*}
and the third equality follows also by symmetrization.  Note that symmetrization here  uses
\[
\int_\bbR \, PV \int_\bbR F(x,y)\,dxdy = \lim_{\eps\to 0} \int_{|x-y|\ge \eps} F(x,y) \,dxdy = \lim_{\eps\to 0} \int_{|x-y|\ge \eps} \frac{F(x,y)+F(y,x)}2 \,dxdy.
\]
We then obviously have 
\beq\lb{2.90}
L_1^+ - L_1^-\le 0 \le  L_1^\pm.
\eeq
On the other hand, for $|x-x_0|\le 5$ we can estimate the inside integral in $L_2^\pm$ by $C (1+\|f\|_{C^2_\gamma}) $ using 
$|h(x)-h(x-y)-h'(x)y|\le y^2$ for $|y|\le 1$ and then \eqref{2.1} with $S=[-1,1]$, while for $|x-x_0|\ge 5$ this integral is bounded by $\frac 8{(|x-x_0|-4)^2}$ because $\supp \, h\subseteq [x_0-4,x_0+4]$.  It follows that
\[
|L_2^\pm|\le C(1+\|f\|_{C^2_\gamma})  \| g\|_{\tilde L^2_{x_0}}^2,
\]
and therefore
\beq \lb{2.9}
\begin{split}
K_3^\pm & \ge  - C(1+\|f\|_{C^2_\gamma})  \| g\|_{\tilde L^2_{x_0}}^2,
\\ K_3^+ - K_3^- & \le C(1+\|f\|_{C^2_\gamma})  \| g\|_{\tilde L^2_{x_0}}^2.
\end{split}
\eeq

We now write $K_4^-= - L_3^- + L_4^- + L_5^-$, where
\begin{align*}
L_3^- &: = 2\int_\bbRr h(x)  g(x)^2 \, PV  \int_\bbR \frac {  (f(x)- f(y)) \big[ f(x) - f(y) -   f'(y)(x-y) \big]  } {[(x-y)^2+(f(x) - f(y))^2]^2} \, dydx ,
\\ L_4^- &: =  \frac 12\int_\bbRr h(x)  g(x) \, PV \int_\bbR \frac {  g(y) (f(x)- f(y)) ( f''(x) + f''(y)) \min\{(x-y)^2,1\}   } {[(x-y)^2+(f(x) - f(y))^2]^2} \, dydx ,
\\ L_5^- &: = 2\int_\bbRr h(x)  g(x) \, PV \int_\bbR  \frac {g(y) (f(x)- f(y)) D(x,y)} {[(x-y)^2+(f(x) - f(y))^2]^2}\, dydx ,
\end{align*}
where
\[
D(x,y):=   f(x) - f(y) -   f'(y)(x-y) - \frac   {f''(x) + f''(y)} 4 \min\{ (x-y)^2,1\} .
\]
After changing variables $y\leftrightarrow x-y$, we can treat $L_3^-$ similarly to $K_1^-$, writing 
\[
(f(x)- f(x-y)) [ f(x) - f(x-y) -  f'(x-y)y] = - (f(x+|y|)- f(x)) [ f(x) - f(x+|y|) +  f'(x+|y|)|y|] 
\]
for $y\in (-1,0)$ and then also estimating
\[
\big| [ f(x) - f(x-y) -  f'(x-y)y] - [ f(x) - f(x+y) +  f'(x+y)y] \big| = \frac 12 |f''(z_1)-f''(z_2)| y^2 \le \|f''\|_{\dot C^{\gamma}} y^{2+\gamma}
\]
for $y\in(0,1)$, with some $z_1\in[x-y,x]$ and $z_2\in[x,x+y]$.
We thus obtain
\[
|L_3^- |\le C(1+ \|f\|_{C^{2,\gamma}_\gamma}^2)  \| g\|_{\tilde L^2_{x_0}}^2.
\]

On the other hand, symmetrization and 
\beq \lb{2.9b}
 | g(x) g(y)|\le \frac {g(x)^2 + g(y)^2} 2
 \eeq
  show that
\begin{align*}
|L_4^-| &  = \left| \frac 14\int_\bbRr   g(x)  \int_\bbR \frac {  g(y) (h(x)-h(y)) (f(x)- f(y)) ( f''(x) + f''(y))\min\{ (x-y)^2,1\}  } {[(x-y)^2+(f(x) - f(y))^2]^2} \, dydx \right| 
\\ & \le \frac 14\int_\bbRr   g(x)^2  \int_\bbR \frac {  |h(x)-h(y)| \, |f(x)- f(y)| \, | f''(x) + f''(y)| \min\{ (x-y)^2,1\}   } {[(x-y)^2+(f(x) - f(y))^2]^2} \, dydx.
\end{align*}
The last integral in $y$ is bounded by $C\|f''\|_{L^\infty}$ when $|x-x_0|\le 5$ and by  $\frac {16\|f''\|_{L^\infty} }{(|x-x_0|-4)^3}$ when $|x-x_0|\ge 5$, so
\[
|L_4^- |\le C(1+ \|f''\|_{L^\infty})  \| g\|_{\tilde L^2_{x_0}}^2.
\]
And since we clearly have
\[
|4D(x,y)| = | 2f''(z) -  f''(x) - f''(y)| (x-y)^2 \le  4 \|f''\|_{C^{\gamma}} |x-y|^{2+\gamma} 
\]
with some $z$ between $x$ and $y$ when $|x-y|\le 1$, from \eqref{2.9b} we again see that
\[
|L_5^-| \le C(1+ \|f'\|_{C^{1,\gamma}})  \| g\|_{\tilde L^2_{x_0}}^2.
\]
Therefore we proved 
\beq \lb{2.9c}
|K_4^- |\le C(1+ \|f\|_{C^{2,\gamma}_\gamma}^2)  \| g\|_{\tilde L^2_{x_0}}^2.
\eeq

Next write $K_4^+= - L_3^+ - L_4^+ + L_5^+$, where
\begin{align*}
L_3^+ &: = 2\int_\bbRr h(x)  g(x) \,  \int_\bbR \frac { (g(x)-g(y)) (f(x)+ f(y))^2  } {[(x-y)^2+(f(x) + f(y))^2]^2} \, dydx ,
\\ L_4^+ &: =  2\int_\bbRr h(x)  g(x)^2 \,  \int_\bbR \frac {  (f(x) + f(x-y))    f'(x-y)y  } {[y^2+(f(x) + f(x-y))^2]^2} \, dydx ,
\\ L_5^+ &: =  2\int_\bbRr h(x)  g(x) \,  \int_\bbR \frac { g(y)  (f(x) + f(y))    f'(y)(x-y)  } {[(x-y)^2+(f(x) + f(y))^2]^2} \, dydx ,
\end{align*}
and in $L_4^+$ we changed variables $y\leftrightarrow x-y$.  We can now estimate $L_3^+$ in the same way as $K_3^+$, although this time  a simpler argument can be used to estimate
the term analogous to $L_2^+$.
That is because after changing variables $y\leftrightarrow x-y$, this term becomes
\beq\lb{2.9a}
\int_{ \bbR} g(x)^2  \int_{ \bbR}  \frac{(h(x)-h(x-y) )(f(x)+ f(x-y))^2 }{[y^2+(f(x)+ f(x-y))^2]^2} \, dydx,
\eeq
and \eqref{2.7} shows that the absolute value of the inside integral is bounded by
\[
C (1+ \|f'\|_{C^1}^2) \left( \int_{-1}^1 \frac{  \max\{ y^2,  f(x) \}^{2} |y|  }{\max\{|y|,f(x)\}^4 } \, dy + \int_{|y|\ge1} \frac{dy} {y^{2}}  \right)
\le C (1+ \|f'\|_{C^1}^2)
\]
when $|x-x_0|\le 5$.  Hence similarly to  \eqref{2.9} we obtain
\[
L_3^+  \ge  - C(1+\|f'\|_{C^1}^2)  \| g\|_{\tilde L^2_{x_0}}^2
\]

The argument bounding $K_1^+$ shows  also that
\[
|L_4^+|\le C(1+ \|f'\|_{C^1}^2)  \| g\|_{\tilde L^2_{x_0}}^2 .
\]
Finally, symmetrization again shows that
\begin{align*}
 L_5^+  = &  \int_\bbRr h(x) g(x) \,  \int_\bbR \frac { g(y)  (f(x) + f(y))   ( f'(y)-f'(x)) (x-y)  } {[(x-y)^2+(f(x) + f(y))^2]^2} \, dydx 
 \\   & +  \int_\bbRr g(x) \,  \int_\bbR \frac { g(y) (h(x)-h(y)) (f(x) + f(y))   f'(x) (x-y)  } {[(x-y)^2+(f(x) + f(y))^2]^2} \, dydx,
\end{align*}
so  \eqref{2.9b} and a change of  variables yield
\begin{align*}
|L_5^+|  \le & \int_{x_0-5}^{x_0+5}  g(x)^2  \int_\bbR \frac {  |f(x) + f(x-y)|   \, | f'(x)-f'(x-y)| \, |y|  } {[y^2+(f(x) + f(x-y))^2]^2} \, dydx 
\\ & + \int_{|x-x_0|\ge 5}  g(x)^2  \int_{ x-x_0- 4}^{ x-x_0+ 4}
\frac {  |f(x) + f(x-y)|   \, | f'(x)-f'(x-y)| \, |y|  } {[y^2+(f(x) + f(x-y))^2]^2} \, dydx
 \\   & +  \int_\bbRr g(x)^2 \,  \int_\bbR \frac {  |h(x)-h(x-y)| (f(x) + f(x-y))   (|f'(x)|+|f'(x-y)|) |y|  } {[y^2+(f(x) + f(x-y))^2]^2} \, dydx.
\end{align*}
The integral in $y$ in the second term is bounded by $\frac {16\|f'\|_{L^\infty} }{(|x-x_0|-4)^2}$.  In the first term, the integral in $y$ over $\bbR\setminus[-1,1]$ is bounded by $C \|f'\|_{L^\infty}$, and the same bound works for the integral over $[-1,1]$ when $f(x)\ge 1$.  When $f(x)\le 1$, then \eqref{2.7} shows that the latter integral is bounded by
\[
C  (1+ \|f'\|_{C^1}^2) \int_{-1}^1 \frac {  \max\{y^2,f(x)\}  \, y^2  } {\max\{|y|,f(x)\}^4} \, dy
\le C (1+ \|f'\|_{C^1}^2 ).
\]
The third term in the bound on $|L_5^+|$ can be estimated in the same way, so the above bounds yield
\[
K_4^+ \le C(1+ \|f'\|_{C^1}^2)  \| g\|_{\tilde L^2_{x_0}}^2  .
\]

This combines with \eqref{2.9} and \eqref{2.9c} to yield
\[
J_2^+-J_2^-\le C(1+  \|f\|_{C^{2,\gamma}_\gamma}^2)  \| g\|_{\tilde L^2_{x_0}}^2,
\]
and adding  \eqref{2.8} to this finally concludes that
\beq \lb{2.10}
I_1^+ + I_1^- \le C(1+  \|f\|_{C^{2,\gamma}_\gamma}^2 ) \| f'''\|_{\tilde L^2_{x_0}}^2   .
\eeq


\section{Estimates on  $I_2^\pm$} \lb{S4}

We note that \eqref{2.4} and a change of variables show that
\[
I_2^\pm = \pm 2 K_1^\pm \mp 2 \int_\bbRr h(x) f'''(x) \, PV  \int_\bbR \frac { f'''(y) (x-y)  (f(x) \pm f(y))   ( f'(x)\pm f'(y))  } {[(x-y)^2+(f(x) \pm f(y))^2]^2} \, dydx .
\]
Symmetrization shows that the second term equals
\[
\int_\bbRr f'''(x)   \int_\bbR \frac {(h(x)  - h(y)) f'''(y) (x-y)  (f(x) \pm f(y))   ( f'(x)\pm f'(y))  } {[(x-y)^2+(f(x) \pm f(y))^2]^2} \, dydx.
\]
The absolute value of this integral is bounded by 
\beq\lb{2.10a}
\int_\bbR f'''(x)^2  \int_\bbR \frac {|h(x)  - h(y)| \, |x-y| \, |f(x) \pm f(y)| \,  | f'(x)\pm f'(y)|  } {[(x-y)^2+(f(x) \pm f(y))^2]^2} \, dydx
\eeq
When we replace $\pm$ by $-$,  the inside integrand is bounded by $\|f'\|_{C^1} \min\{1,\frac 1{(|x-x_0|-4)^2}\}$, 
so this, \eqref{2.5}, and \eqref{2.4d} yield
\beq \lb{2.11}
|I_2^-| \le C \|f'\|_{C^{1,\gamma}}^2  \|f'''\|_{\tilde L^2}^2  .
\eeq

When we replace $\pm$ by $+$, then \eqref{2.10a} becomes (after changing variables)
\beq\lb{2.10b}
\int_\bbR f'''(x)^2  \int_\bbR \frac {|h(x)  - h(x-y)| \, |y| \,  |f(x) + f(x-y)|  \, | f'(x) + f'(x-y)|  } {[y^2+(f(x) + f(x-y))^2]^2} \, dydx
\eeq
The inside integral is bounded  in the same way as \eqref{2.9a}, so from that bound, \eqref{2.5}, and \eqref{2.4e} we obtain
\beq \lb{2.10c}
|I_2^+| \le C(1+ \|f'\|_{C^1}^2)  \| f'''\|_{\tilde L^2}^2   .
\eeq


\section{Estimates on  $I_3^\pm$} \lb{S5}

We can write $I_3^-=2 (J_3^- -  J_4^-)$, where  
\begin{align*}
J_3^- & := \int_\bbRr h(x) f'''(x) \int_\bbR  \frac {y\, (f''(x) - f''(x-y)) (f'(x) - f'(x-y))^2 [3(f(x) - f(x-y))^2 -  y^2] } { [y^2+(f(x) - f(x-y))^2]^3 }  \, dydx ,
\\ J_4^- & :=  \int_\bbRr h(x) f'''(x) \int_\bbR  \frac {y\, (f''(x) - f''(x-y))^2 (f(x) - f(x-y)) } { [y^2+(f(x) - f(x-y))^2]^2 } \, dydx .
\end{align*}
Since $f''(x) - f''(x-y)= y \int_0^1 f'''(x-sy) ds$ we obtain
\beq\lb{2.11a}
\begin{split}
|J_3^-| 
& \le C \|f'\|_{C^1}^2  \int_{ \bbR^2 \times [0,1]} h(x)  |f'''(x)| \, |f'''(x-sy)| \min\{1,|y|^{-2} \} \, dxdyds 
\\ & \le C \|f'\|_{C^1}^2  \|f'''\|_{\tilde L^2}^2 \int_{\bbR \times [0,1]}   \min\{1,|y|^{-2} \} \, dyds 
 \le C \|f'\|_{C^1}^2 \|f'''\|_{\tilde L^2}^2 .
\end{split}
\eeq
We use the same method to estimate $J_4^-$,  bounding the second factor $f''(x) - f''(x-y)$ via $|f''(x) - f''(x-y)|\le \|f''\|_{\dot C^{\gamma}} |y|^{\gamma}$ for 
$|y|\le 1$.
This yields
\begin{align*}
|J_4^-|   \le C \|f\|_{C^{2,\gamma}_\gamma}^2   \|f'''\|_{\tilde L^2}^2 \int_{\bbR \times [0,1]}   \min\{|y|^{\gamma-1},|y|^{-1-\gamma} \} \, dyds  
\le C  \|f\|_{C^{2,\gamma}_\gamma}^2   \|f'''\|_{\tilde L^2}^2,
\end{align*}
and it follows that
\beq \lb{2.12}
|I_3^-| \le C  \|f\|_{C^{2,\gamma}_\gamma}^2   \|f'''\|_{\tilde L^2}^2 .
\eeq

We now turn to the term $I_3^+$, whose treatment will be much more complicated.  First we split $I_3^+=I_3'+I_3''+I_3'''$, where with $R_x:=\bbR\setminus[-\sqrt{f(x)}, \sqrt{f(x)}\,]$ we have
\begin{align*}
 I_3' & :=  \int_\bbRr \chi_{(1,\infty)}(f(x))\, h(x) f'''(x) \int_{\bbR}   y\, (f''(x)+ f''(x-y)) \, \partial_{x}^2 A^+(x,y) \, dydx ,
\\ I_3'' & := \int_\bbRr \chi_{[0,1]}(f(x))\, h(x) f'''(x)  \, PV  \int_{R_x}  y\, (f''(x)+ f''(x-y)) \, \partial_{x}^2 A^+(x,y) \, dydx,
\\ I_3''' & :=  \int_\bbRr \chi_{[0,1]}(f(x))\, h(x) f'''(x) \int_{-\sqrt{f(x)}}^{\sqrt{f(x)}}   y\, (f''(x)+ f''(x-y)) \, \partial_{x}^2 A^+(x,y) \, dydx .
\end{align*}
The inside integrand in all three integrals is twice the difference of inside integrands from $J_3^-$ and $J_4^-$ but with all $f^{(j)}(x)-f^{(j)}(x-y)$ replaced by $f^{(j)}(x)+f^{(j)}(x-y)$.  It is clear that
\beq\lb{2.11b}
\begin{split}
|I_3'| &\le  C \|f'\|_{C^1}^2  \int_{ \bbR^2} h(x)  |f'''(x)| \left(|f''(x)| + |f''(x-y)|\right) \min\{1,|y|^{-2} \} \, dxdyds 
\\ & \le C \|f'\|_{C^1}^2  \|f''\|_{\tilde L^2} \|f'''\|_{\tilde L^2} \int_{\bbR}   \min\{1,|y|^{-2} \} \, dyds 
 \le C \|f'\|_{C^1}^2 ( \|f''\|_{\tilde L^2}^2 + \|f'''\|_{\tilde L^2}^2) .
\end{split}
\eeq

We leave the most challenging term $I_3'''$ for later and deal with $I_3''$ next.
We decompose it similarly to $I_3^-$, which gives us terms like $J_3^-$ and $J_4^-$ but with each $f^{(j)}(x)-f^{(j)}(x-y)$ replaced by $f^{(j)}(x)+f^{(j)}(x-y)$.  Then we decompose these terms further in two ways.  First, we replace one factor $f''(x) + f''(x-y)$ in each integral by the sum of $2f''(x)$ and $f''(x-y) - f''(x)$, and then in the first ``subterm'' with $2f''(x)$ we also split $3(f(x) + f(x-y))^2 -  y^2$ into the sum of $3[y^2+(f(x) + f(x-y))^2]$ and $-4y^2$ in order to simplify the formulas.  We thus obtain
\[
I_3''=2 (J_3^+ -  J_4^+ + 6J_5^+ - 8 J_6^+ - 2J_7^+),
\]
 where with $\tilde h(x):=\chi_{[0,1]}(f(x))\, h(x)$ we have
\begin{align*}
J_3^+ & := \int_\bbRr \tilde h(x) f'''(x) \int_{R_x}  \frac {y\, (f''(x-y) - f''(x)) (f'(x) + f'(x-y))^2 [3(f(x) + f(x-y))^2 -  y^2] } { [y^2+(f(x) + f(x-y))^2]^3 }  \, dydx ,
\\ J_4^+ & :=  \int_\bbRr \tilde h(x)  f'''(x) \int_{R_x}  \frac {y\, (f''(x-y) - f''(x)) (f''(x) + f''(x-y)) (f(x) + f(x-y)) } { [y^2+(f(x) + f(x-y))^2]^2 } \, dydx ,
\\ J_5^+ & := \int_\bbRr \tilde h(x)  f'''(x) f''(x)  \, PV \int_{R_x}  \frac {y\,  (f'(x) + f'(x-y))^2  } { [y^2+(f(x) + f(x-y))^2]^2 }  \, dydx ,
\\ J_6^+ & := \int_\bbRr  \tilde h(x) f'''(x) f''(x)  \, PV \int_{R_x}  \frac {y^3\,  (f'(x) + f'(x-y))^2  } { [y^2+(f(x) + f(x-y))^2]^3 }  \, dydx ,
\\ J_7^+ & :=  \int_\bbRr  \tilde h(x) f'''(x) f''(x)  \, PV \int_{R_x}  \frac {y\,  (f''(x) + f''(x-y)) (f(x) + f(x-y)) } { [y^2+(f(x) + f(x-y))^2]^2 } \, dydx .
\end{align*}
We will estimate these integrals one by one.

First, the argument for $J_3^-$ gives us also
\[
|J_3^+|  \le C (1+\|f'\|_{C^1}^2 ) \|f'''\|_{\tilde L^2}^2
\]
because $|f'(x) + f'(x-y)| \le  4(1+ \|f'\|_{C^1}) |y|$ for $|y|\in[ \sqrt{f(x)},1]$ by \eqref{2.7}.   From \eqref{2.7} we also have $|f(x-y)| \le  2(1+ \|f'\|_{C^1}) y^2$ for $|y|\in[ \sqrt{f(x)},1]$, which can be used in the argument for $J_4^-$ even though now we only have $|f''(x) + f''(x-y)|\le 2\|f''\|_{L^\infty}$. We thus obtain
\begin{align*}
|J_4^+|   \le C  (1+\|f'\|_{C^1}^2 ) \|f'''\|_{\tilde L^2}^2.
\end{align*}

To estimate the last three integrals, we will use that 
\beq\lb{2.12b}
\int_{R_x} y\phi(y)dy = \int_{\sqrt{f(x)}}^\infty y\,[\phi(y)-\phi(-y)]\,dy
\eeq
and assume below without loss that $x=0$ (and of course then also $f(0)\le 1$). 
For $J_5^+$, note that from \eqref{2.6a}  we obtain
\begin{align*}
&y^5 \left|   \frac {y(f'(0) + f'(-y))^2  } { [y^2+(f(0) + f(-y))^2]^2 } - \frac {y(f'(0) + f'(y))^2  } { [y^2+(f(0) + f(y))^2]^2 } \right|
\\ & \qquad \qquad \qquad \le y^2|2f'(0) + f'(y)+f'(-y)|\,| f'(y)-f'(-y)|  
\\ & \qquad \qquad \qquad \quad + 2 y(f'(0) + f'(y))^2 
| f(y)-f(-y)|.
\end{align*}
The first term on the right-hand side can be estimated by 
\[
C (1+\|f'\|_{C^{1,\gamma}}^2) \,y^2  \min \left\{1,(\sqrt{f(0)} +y^{1+\gamma})y \right\}  
\]
for $y\ge\sqrt {f(0)}$ 
due to \eqref{2.7} and
\beq \lb{2.12a}
\begin{split}
|f'(y)-f'(-y)| & \le 2\|f'\|_{C^1} \min\{1,y\}, 
\\ |2f'(0) + f'(y)+f'(-y)| &  \le 4\|f'\|_{L^\infty}  \qquad \qquad \qquad \quad \text{for } y > 1,  
\\ | f'(y)+f'(-y)-2f'(0)| &  \le    2\|f''\|_{\dot C^{\gamma}} |y|^{1+\gamma}     \qquad \qquad \text{for  $y\in[0,1]$}.
\end{split}
\eeq
The second term  can be estimated by 
\[
C (1+\|f'\|_{C^1}^3) y^2 \min\{1,y^{3}\} 
\]
for $y\ge\sqrt{f(0)}$ because \eqref{2.7}  shows
\begin{align*}
 |f'(z)|  \le  2(1+\|f'\|_{C^1}) \min \{ 1,y\}
\end{align*}
for $y\ge\sqrt{f(0)}$ and all $|z|\le y$.  Hence
\begin{align*}
|J_5^+|  & \le \int_\bbRr h(x) | f'''(x)|\,| f''(x)| \int_{\sqrt{f(x)}}^\infty C (1+\|f'\|_{C^{1,\gamma}}^3) \frac{ \min\{1,(\sqrt {f(x)}+y^{1+\gamma})y\} } { y^3 } dydx 
\\ &\le C (1+\|f'\|_{C^{1,\gamma}}^3) \|f''\|_{\tilde L^2} \|f'''\|_{\tilde L^2}
\end{align*}
because the inside integral is bounded by $C  (1+\|f'\|_{C^{1,\gamma}}^3) $.
The same bound can be obtained for $J_6^+$ in exactly the same way.

Finally, $J_7^+$ can be estimated similarly via \eqref{2.12b}, using \eqref{2.6a} to get
\beq\lb{2.12c}
\begin{split}
&y^5 \left|   \frac {y\,  (f''(0) + f''(-y)) (f(0) + f(-y)) } { [y^2+(f(0) + f(-y))^2]^2 } 
-\frac {y\,  (f''(0) + f''(y)) (f(0) + f(y)) } { [y^2+(f(0) + f(y))^2]^2 }  \right|
\\ &   \qquad \le y^2|f''(y) - f''(-y) |\,( f(0)+f(-y)) +    y^2|f''(0) + f''(y) |\,| f(y)-f(-y)|
\\ &  \qquad \quad + 2 y |f''(0) + f''(y) |\, ( f(0)+f(y)) 
| f(y)-f(-y)|.
\end{split}
\eeq
Then the above bounds together with (recall \eqref{2.7} and that $f(0)\le 1$)
\begin{align*}
|f''(y)-f''(-y)| &\le  2\|f''\|_{\dot C^{\gamma}}\min\{1,y^{\gamma}\} ,
\\ |f(\pm y)| &\le 2(1+\|f'\|_{C^1}) \min \{ y,y^2\},
\\ |f(y)-f(-y)| & = \left| 2f'(0)y - \int_{-y}^y (z-y\sgn z)f''(z)dz \right|
 \le 2
  |f'(0)|y + \|f''\|_{\dot C^{\gamma}} y^{2+\gamma} 
 \\ & \le  4 (1+\|f'\|_{C^{1,\gamma}}) 
 (\sqrt {f(0)}  +  y^{1+\gamma})y 
\end{align*}
for $y\ge\sqrt{f(0)}$,
bound the right-hand side of \eqref{2.12c} by
\begin{align*}
C (1+\|f'\|_{C^{1,\gamma}}^2) & \left[ y^{3}  \min\{1,y^{1+\gamma} \}  + y^2 \min\left\{y, (\sqrt {f(0)}  +  y^{1+\gamma})y \right\} \right]
\\ & + C (1+\|f'\|_{C^{1,\gamma}}^3) \min\left\{y^3, (\sqrt {f(0)}  +  y^{1+\gamma})y^4 \right\} 
\end{align*}
and we  obtain
\begin{align*}
|J_7^+|  & \le \int_\bbRr h(x) | f'''(x)|\,| f''(x)| \int_{\sqrt{f(x)}}^\infty C (1+\|f'\|_{C^{1,\gamma}}^3) \frac{ \min\{1,\sqrt {f(x)}+y^{1+\gamma}\} } { y^{2} } dydx 
\\ & \le C (1+\|f'\|_{C^{1,\gamma}}^3) \|f''\|_{\tilde L^2} \|f'''\|_{\tilde L^2}.
\end{align*}
So the above arguments yield the bound 
\beq\lb{2.13}
|I_3''| \le C  (1+\|f'\|_{C^{1,\gamma}}^3) (\|f''\|_{\tilde L^2}^2 + \|f'''\|_{\tilde L^2}^2 ).
\eeq


Let us now turn to $I_3'''$ and consider its inside integrand, again for $x=0$. This is
\beq\lb{2.13a}
y(f''(0)+ f''(-y)) \, \partial_{x}^2 A^+(0,y) = 6yP(-y) - 8y^3Q(-y) -2yS(-y),
\eeq
where
\begin{align*}
P(y) & :=  \frac { (f''(0) + f''(y))(f'(0) + f'(y))^2  } { [y^2+(f(0) + f(y))^2]^2 },
\\ Q(y) &:= \frac { (f''(0) + f''(y)) (f'(0) + f'(y))^2  } { [y^2+(f(0) + f(y))^2]^3 },
\\ S(y) &:= \frac { (f''(0) + f''(y))^2 (f(0) + f(y)) } { [y^2+(f(0) + f(y))^2]^2 }.
 \end{align*}
Let us first deal with $S$ above.

\begin{lemma} \lb{L.2.1}
We have
\beq\lb{2.30}
\left| \int_{- \sqrt{f(0)} }^{  \sqrt {f(0)} } y S(y) dy \right| \le  C (1+\|f'\|_{C^1}^{4}) \left( |f''(0)|   +  M_f(0) \right) 
\eeq
when $f(0)\le 1$, where 
\[
M_f(x):=\sup_{y>0} \frac 1{2y} \int_{-y}^y |f'''(x+z)|dz
\]
is the Hardy-Littlewood maximal function for $f'''$.
\end{lemma}

{\it Remark.}
This result and Lemma \ref{L.2.2} below, which is its analog with  $6yP(y)-8y^3Q(y)$ in place of $yS(y)$, now yield
\beq\lb{2.14}
|I_3'''|\le C (1+\|f'\|_{C^1}^{4})  (\|f''\|_{\tilde L^2 }^2+\|f'''\|_{\tilde L^2 }^2) .
\eeq
This is because both results equally hold with any $x\in\bbR$ instead of just $x=0$, and it is well known that $\|M_f\|_{L^2} \le C \|f'''\|_{L^2}$ for $f\in L^2(\bbR)$, which easily implies $\|M_f\|_{\tilde L^2 } \le C \|f'''\|_{\tilde L^2 }$.
 because for $y\ge 1$ we have
\[
\frac 1{2y} \int_{-y}^y |f'''(x+z)|dz \le \sup_{x'\in\bbR} \frac 12 \int_{x'-1}^{x'+1} |f'''(z)|dz 
\le \|f'''\|_{\tilde L^2}.
\]
  Together with \eqref{2.12}, \eqref{2.11b}, and \eqref{2.13}, this yields
\beq \lb{2.14a}
|I_3^\pm| \le C (1+\|f\|_{C^{2,\gamma}_\gamma}^{4})  (\|f''\|_{\tilde L^2 }^2+\|f'''\|_{\tilde L^2 }^2) .
\eeq
We note that, as the proof of Lemma \ref{L.2.2} shows,  \eqref{2.30} does not hold individually for integrands $yP(y)$ and $y^3Q(y)$!
\smallskip

\begin{proof}
 Let 
\[
a:=f(0)\in [0,1], \qquad b:=f'(0)\in \left[ -2 \|f''\|_{L^\infty}^{1/2} \sqrt{a},2 \|f''\|_{L^\infty}^{1/2} \sqrt{a} \right], \qquad  c:= {f''(0)}
\]
(see \eqref{2.5a}).
Then 
\begin{align*}
f''(y)  = c + g_0(y), \qquad
 f'(y)  = b+cy + g_1(y), \qquad
 f(y)  = a+by+\frac c2 y^2 + g_2(y)
\end{align*}
for all $y\in\bbR$, where
\begin{align*}
g_0(y)  := \int_0^y f'''(z) dz, \qquad
g_1(y)  :=  \int_0^y (y-z) f'''(z) dz, \qquad 
 g_2(y)  :=  \int_0^y \frac{(y-z)^2}2 f'''(z) dz.
\end{align*}
Let also
\begin{align*}
G_j(y) & :=g_j(y) - g_j(-y) = \int_{-y}^y \frac{(y-z)^j}{j!} f'''(z) dz,
\\ F_{k} (y) & := (f(0)+f( y))^k + (f(0)+f( -y))^k.
\end{align*}

By \eqref{2.6a} we have 
\beq\lb{2.30a}
S(y)-S(-y) = S_1(y)G_0(y) + S_2(y) + S_3(y) ,
\eeq
where
\begin{align*}
S_1(y) & :=  \frac { (2f''(0)+f''(y)+f''(-y) ) \,  (f(0)+f( y))   } {[y^2+(f(0)+f( y))^2]^2},
\\ S_2(y) & := \frac {(2c+ g_0(-y))^2  (2by+G_2(y))} {[y^2+(f(0)+f( y))^2]^2 },
\\ S_3(y) & :=  - \frac {(2c+ g_0(-y))^2 (f(0)+f( -y)) \, [2y^2+ F_{2}(y) ] \, F_{1}(y) \, (2by+G_2(y))} 
{[y^2+(f(0)+f( y))^2]^2 \, [y^2+(f(0)+f( -y))^2]^2}.
\end{align*}

First consider $S_1$ and let
\[
\tilde S_1(y)  := \frac {(2f''(0)+f''(y)+f''(-y) )   (f(0)+f( y))} {[y^2+4a^2]^2 }.
\]
From \eqref{2.7} we have
\beq\lb{2.30c}
\begin{split}
|f(0)+f(\pm y)|  &\le  3(1+ \|f'\|_{C^1}) a
\\ |f(0)+f(\pm y)-2a| & =|f(\pm y)-f(0)|\le 2(1+ \|f'\|_{C^1}) \sqrt a \, y
\end{split}
\eeq
for all $y\in[0, \sqrt a]$, and so for these $y$ we have
\[
|(f(0)+f(\pm y))^2-4a^2|  \le 20(1+ \|f'\|_{C^1}^2 ) a^{3/2}y
\]
 and then
\beq\lb{2.33}
\left| \frac  {y^2+4a^2 } {y^2+(f(0)+f( \pm y))^2 } -1 \right| \le  C  (1+ \|f'\|_{C^1}^2) \frac{   a^{3/2}} {\max\{y,a\} }.
\eeq
This yields
\beq\lb{2.32a}
| S_1(y) - \tilde S_1(y)|  
\le C (1+\|f'\|_{C^1}^{4})   \frac{   a^{5/2}} {\max\{y,a\}^5}
\eeq
for all $y\in[0, \sqrt a]$.  Similarly, \eqref{2.30c} yields
\beq\lb{2.32b}
|\tilde S_1(y)| \le C(1+\|f'\|_{C^1}^{2})  \frac{a} {\max\{y,a\}^4}
\eeq
for  $y\in[0, \sqrt a]$,
so
\begin{align*}
\int_0^{ \sqrt a} y |\tilde S_1(y)G_0(y)| dy
 \le C(1+\|f'\|_{C^1}^{2}) \int_{0}^{ \sqrt a}  \frac{ay^2} {\max\{y,a\}^4}  M_{f}(0) dy 
 \le C(1+\|f'\|_{C^1}^{2}) M_{f}(0).
\end{align*}
An analogous bound can be obtained using \eqref{2.32a} in place of \eqref{2.32b}, and we thus have
\beq\lb{2.32c}
\int_0^{ \sqrt a} y |S_1(y)G_0(y)| dy  
\le  C(1+\|f'\|_{C^1}^{4}) M_{f}(0),
\eeq
which is dominated by the right-hand side of \eqref{2.30}.

Let next
\[
\tilde S_4(y):= \frac { 8bc^2y }{[y^2+4a^2]^2}  - \frac {128a^2bc^2y } {[y^2+4a^2]^3} = 8bc^2 \frac {y^3-12a^2y } {[y^2+4a^2]^3}, 
\]
which is  $S_2(y)+S_3(y)$ after dropping  all $G_2$ and $g_0$ terms and replacing all $f(0)+f(\pm y)$ by $2a$ (including those in $F_j(y)$);
this is also the leading order term of $S_2(y)+S_3(y)$ when $bc\neq 0$ and $0\le y\ll \sqrt a$.  When $|b|\sim\sqrt a$ and $c\sim 1$, in the latter region we have $y|\tilde S_4(y)|\sim \sqrt a\, y^2 \max\{y,a\}^{-4}$, which is too large to yield an upper bound on $\int_0^{\sqrt a} y|\tilde S_4(y)|dy$ that is independent of $a$ when $a\ll 1$.  However, we  observe that
\beq\lb{2.34}
y \tilde S_4(y) = -\partial_y\, \frac {8bc^2y^3}{[y^2+4a^2]^2} ,
\eeq
so this and $|b|\le 2 \|f''\|_{L^\infty}^{1/2} \sqrt{a}$ yield
\beq\lb{2.35}
\left| \int_0^{ \sqrt a} y \tilde S_4(y) dy \right| \le  C  \|f''\|_{L^\infty}^{1/2} c^2, 
\eeq
which is dominated by the right-hand side of \eqref{2.30}.

Now let
\[
S_4(y):=\frac {(2c+ g_0(-y))^2  (2by+G_2(y)) }
{[y^2+4a^2]^2} - \frac {16a^2(2c+ g_0(-y))^2  (2by+G_2(y)) }
{[y^2+4a^2]^3},
\]
which is obtained from $S_2(y)+S_3(y)$ in the same way as $\tilde S_4(y)$, but without dropping the $G_2$ and $g_0$ terms.
From  \eqref{2.30c}, 
\beq\lb{2.35b}
\begin{split}
2c+ g_0(-y) & =f''(0)+f''(-y),
\\ 2by+G_2(y) & = f(y)-f(-y),
\end{split}
\eeq
\eqref{2.7},  \eqref{2.33}, 
and $|g_0(-y)|\le 2yM_f(0)$ for $y\in[0,1]$ we obtain
\begin{align}
| S_2(y) & +S_3(y) - S_4(y)|  \notag
\\ & \le C  \|f''\|_{L^\infty} (|c| + M_f(0)\sqrt a) (1+\|f'\|_{C^1} ) \sqrt a \,y 
\left( \frac{  (1+\|f'\|_{C^1}^2) a^{3/2}} {\max\{y,a\}^5} +  \frac{(1+\|f'\|_{C^1}^2) a^{3/2}y} {\max\{y,a\}^6}   \right) \notag
\\ & \leq C  (1+\|f'\|_{C^1}^{4})   (|c|+ M_f(0)) \frac{ a^2} {\max\{y,a\}^4}  \lb{2.36}
\end{align}
for $y\in[0, \sqrt a]$, similarly to \eqref{2.32a}.  This bound follows after we first replace $f(0)+f(-y)$ by $2a$ and $F_1(y)$ by $4a$ in $S_3$, and then all the $(f(0)+f(\pm y))^2$ by $4a^2$ in both $S_2$ and $S_3$.
Thus
\beq\lb{2.37}
 \int_0^{ \sqrt a} y | S_2(y) +S_3(y) - S_4(y)| dy  \le  C (1 +\|f'\|_{C^1}^{4})  (|c|+ M_f(0)) ,
\eeq
which is again dominated by the right-hand side of \eqref{2.30}.
For $y\in[0, \sqrt a]$ we similarly have
\beq\lb{2.38}
|S_4(y) - \tilde S_4(y)| \le C(1+\|f'\|_{C^1}^{2})  \frac{ \sqrt a\, y + y^2} {\max\{y,a\}^4}  \int_{-y}^{y}|f'''(z)| dz \le C(1+\|f'\|_{C^1}^{2})  \frac{M_f(0) a} {\max\{y,a\}^3}  ,
\eeq
so  we  obtain
\beq\lb{2.39}
 \int_0^{ \sqrt a} y |S_4(y) - \tilde S_4(y)| dy
\le  C(1+\|f'\|_{C^1}^{2}) M_f(0).
\eeq
The result therefore follows from  \eqref{2.30a}, \eqref{2.32c},   \eqref{2.35}, \eqref{2.37}, and \eqref{2.39}.
\end{proof}

Next we consider the terms $P$ and $Q$.

\begin{lemma} \lb{L.2.2}
When $f(0)\le 1$, we have
\beq\lb{2.40}
\left| \int_{- \sqrt{f(0)} }^{  \sqrt {f(0)} } \left[ 3yP(y) -4y^3 Q(y) \right] dy \right| 
 \le  C (1+\|f'\|_{C^1}^{4})  \left( |f''(0)|   +  M_f(0) \right) 
\eeq
\end{lemma}

\begin{proof}
This proof will follow along the same lines as that of Lemma \ref{L.2.1} so we will skip some details.  Let $a,b,c,g_j,G_j^\pm$ be as above, and let us start with $P$.  By \eqref{2.6a} we have 
\beq\lb{2.40a}
P(y)-P(-y) = P_1(y)G_0(y) + P_2(y) + P_3(y) ,
\eeq
where
\begin{align*}
P_1(y) & :=  \frac {  (f'(0)+f'( y))^2   } {[y^2+(f(0)+f( y))^2]^2},
\\ P_2(y) & := \frac {(2c+ g_0(-y)) \, (2f'(0)+f'(y)+f'(-y)) \, (2cy+G_1(y))} {[y^2+(f(0)+f( y))^2]^2 },
\\ P_3(y) & :=  - \frac {(2c+ g_0(-y)) (f'(0)+f'( -y))^2 \, [2y^2+ F_{2}(y) ] \, F_{1}(y) \, (2by+G_2(y))} 
{[y^2+(f(0)+f( y))^2]^2 \, [y^2+(f(0)+f( -y))^2]^2}.
\end{align*}

Let
\[
\tilde P_1(y)  := \frac {  (f'(0)+f'( y))^2   } {[y^2+4a^2]^2 }.
\]
Then similarly to \eqref{2.32a} and \eqref{2.32b} we obtain
\begin{align*}
| P_1(y) - \tilde P_1(y)|  &\le C (1+\|f'\|_{C^1}^{4})   \frac{   a^{5/2}} {\max\{y,a\}^5},
\\ |\tilde P_1(y)| &\le C(1+\|f'\|_{C^1}^{2})  \frac{a} {\max\{y,a\}^4}
\end{align*}
for $y\in[0, \sqrt a]$, where we also used the first bound from \eqref{2.7}.  Just as \eqref{2.32c}, we now get
\beq\lb{2.41}
\int_0^{ \sqrt a} y |P_1(y)G_0(y)| dy  
\le  C(1+\|f'\|_{C^1}^{4}) M_{f}(0).
\eeq

Let now
\[
\tilde P_4(y):= \frac { 16bc^2y }{[y^2+4a^2]^2}  - \frac {128ab^3cy } {[y^2+4a^2]^3} = 
16bc^2 \frac {y^3-12a^2y } {[y^2+4a^2]^3} + 128abc \frac {(2ac-b^2)y } {[y^2+4a^2]^3} , 
\]
which is  $P_2(y)+P_3(y)$ after dropping  all $G_j$ and $g_0$ terms, replacing all $f(0)+f(\pm y)$ by $2a$ (including those in $F_j(y)$), and replacing all $f'(0)+f'(\pm y)$ terms by $2b$.  Then
\beq\lb{2.42}
y \tilde P_4(y) = 128abc \frac {(2ac-b^2)y^2 } {[y^2+4a^2]^3} -\partial_y\, \frac {16bc^2y^3}{[y^2+4a^2]^2} ,
\eeq
but this time the first term means that
\beq\lb{2.42a}
\left| \int_0^{ \sqrt a} y \tilde P_4(y) dy \right| \sim a^{-1/2}
\eeq
for $a\ll 1$ whenever $a\neq 0$, $|b|\not\ll \sqrt a$,  $|c|\not\ll 1$, and  $| 2ac-b^2 |\not \ll a$.  We will therefore  leave $\tilde P_4$ for now and return to it at the end of this proof.

Next let
\[
P_4(y):=\frac {4b (2c+ g_0(-y))  \, (2cy+G_1(y))} {[y^2+4a^2]^2 }
 - \frac { 32ab^2 (2c+ g_0(-y))  \, (2by+G_2(y))} 
{[y^2+4a^2]^3 }
\]
which is obtained from $P_2(y)+P_3(y)$ in the same way as $\tilde P_4(y)$, but without dropping the $G_2$ and $g_0$ terms.  Similarly to \eqref{2.36}, we now obtain
\begin{align*}
| P_2(y) & +P_3(y) - P_4(y)|  
\\ & \le C \|f''\|_{L^\infty}^2 \frac{M_f(0) y^3} {\max\{y,a\}^4} +  C (|c| + M_f(0)\sqrt a) (1+\|f'\|_{C^1}^4)  \frac{  a^{2}y^2} {\max\{y,a\}^6}  
\\ & + C (|c| + M_f(0)\sqrt a) (1+\|f'\|_{C^1}^2)  \sqrt a \, y
\left(   \frac{  (1+\|f'\|_{C^1}^2) a^{3/2}} {\max\{y,a\}^5} +    \frac{  (1+\|f'\|_{C^1}) a^{3/2}y} {\max\{y,a\}^6}   \right) 
\\ & \leq C  (1+\|f'\|_{C^1}^{4})   (|c|+ M_f(0))  \frac{   a^2+y^3} {\max\{y,a\}^4} 
\end{align*}
for $y\in[0, \sqrt a]$, where we also used $|b|\le 2 \|f''\|_{L^\infty}^{1/2} \sqrt{a}$
and 
\beq\lb{2.42b}
\begin{split}
2cy+G_1(y) & = f'(y)-f'(-y),  
\\ \left| (f'(0)+f'(\pm y))^n -(2b)^n \right|  &\le C_n (1+\|f'\|_{C^1}^n) a^{(n-1)/2}   y, 
\\ |2f'(0) + f'(y)+f'(-y)-4b| & = |g_1(y)+g_1(-y)| \le  2y^2 M_f(0)
\end{split}
\eeq
for $y\in[0, \sqrt a]$ and $n\ge 1$ (here we only need $n=2$ but  other $n$ will be useful in Lemma~\ref{L.2.3} below).
  The above estimate is obtained  if we first replace all the terms $f'(0)+f'(\pm y)$ by $2b$, where the errors are the first two terms after the first inequality above (in the first of them we use the first line in \eqref{2.35b} and the first and third line in \eqref{2.42b}; in the second we use the second line in \eqref{2.35b}, \eqref{2.7}, the second line in \eqref{2.42b}, and \eqref{2.30c}).
  Then during the following replacements we  use the bound $|b|\le 2 \|f''\|_{L^\infty}^{1/2} \sqrt{a}$, obtaining essentially the same errors as in Lemma \ref{L.2.1} (these are collected in the third term after the first inequality above; here we used the second line in \eqref{2.35b}, the first line in \eqref{2.42b}, the second line in \eqref{2.30c}, and \eqref{2.33}).
From this we see that
\beq\lb{2.43}
 \int_0^{ \sqrt a} y | P_2(y) +P_3(y) - P_4(y)| dy  \le  C (1 +\|f'\|_{C^1}^{4})  (|c|+ M_f(0)).
\eeq
By \eqref{2.35b}, the first line in \eqref{2.42b}, and \eqref{2.7}  we again have
\[
|P_4(y) - \tilde P_4(y)| \le C(1+\|f'\|_{C^1}^{2})  \frac{ a} {\max\{y,a\}^4}  \int_{-y}^{y}|f'''(z)| dz \le C(1+\|f'\|_{C^1}^{2})  \frac{M_f(0) a} {\max\{y,a\}^3}
\]
for $y\in[0, \sqrt a]$, so  we  obtain
\[
 \int_0^{ \sqrt a} y |P_4(y) - \tilde P_4(y)| dy
\le  C(1+\|f'\|_{C^1}^{2}) M_f(0).
\]
Finally we can conclude from this and \eqref{2.43} that
\beq\lb{2.44}
 \int_0^{ \sqrt a} y | P_2(y) +P_3(y) - \tilde P_4(y)| dy  \le  C (1 +\|f'\|_{C^1}^{4})  (|c|+ M_f(0)).
\eeq

The treatment of $Q$ is the same as $P$, with two adjustments due to the power 3 in the denominator of $Q$.  First, all the pointwise estimates on $Q_j$ (terms analogous to $P_j$) have an extra factor $\max\{y,a\}^{-2}$, but this is cancelled in all the integrals by the extra factor $y^2$ in front of $Q$ in \eqref{2.13a} (relative to $P$).  Second, the term $Q_3$ has an extra factor $\frac 32$ relative to $P_3$ (besides $\max\{y,a\}^{-2}$).  This means that we still get
\beq\lb{2.44b}
\int_0^{ \sqrt a} y^3 |Q_1(y)G_0(y)| dy  \le  C(1+\|f'\|_{C^1}^{4}) M_{f}(0).
\eeq
and
\beq\lb{2.44a}
 \int_0^{ \sqrt a} y^3 | Q_2(y) +Q_3(y) - \tilde Q_4(y)| dy  \le  C (1 +\|f'\|_{C^1}^{4})  (|c|+ M_f(0))
\eeq
via the same arguments as \eqref{2.41} and \eqref{2.44}, but now
 the critical term (analogous to $\tilde P_4$) is 
\[
\tilde Q_4(y):= \frac { 16bc^2y }{[y^2+4a^2]^3}  - \frac {192ab^3cy } {[y^2+4a^2]^4} = 
16bc^2 \frac {y^3-20a^2y } {[y^2+4a^2]^4} + 192abc \frac {(2ac-b^2)y } {[y^2+4a^2]^4} .
\]
Hence
\beq\lb{2.44c}
y^3 \tilde Q_4(y) = 192abc \frac {(2ac-b^2)y^4 } {[y^2+4a^2]^4} -\partial_y\, \frac {16bc^2y^5}{[y^2+4a^2]^3} ,
\eeq
which together with \eqref{2.42} yields
\begin{align}
\frac{3y\tilde P_4(y)-4y^3 \tilde Q_4(y)}{16} &= -24abc (2ac-b^2) \frac {y^4-4a^2y^2 } {[y^2+4a^2]^4} -\partial_y\, \frac {3bc^2y^3}{[y^2+4a^2]^2} + \partial_y\, \frac {4bc^2y^5}{[y^2+4a^2]^3}  \notag
\\ & = \partial_y  \frac { 8abc (2ac-b^2) y^3 } {[y^2+4a^2]^3} -\partial_y\, \frac {3bc^2y^3}{[y^2+4a^2]^2} + \partial_y\, \frac {4bc^2y^5}{[y^2+4a^2]^3} . \lb{2.45}
\end{align}
This way we essentially cancelled the problematic term in \eqref{2.42} by adding to it its equally problematic counterpart from \eqref{2.44c}, and using $|b|\le 2 \|f''\|_{L^\infty}^{1/2} \sqrt{a}$ we obtain
\beq \lb{2.46}
\left| \int_0^{ \sqrt a} \left[ 3y\tilde P_4(y)- 4y^3 \tilde Q_4(y) \right] dy \right| 
\le C (\|f''\|_{L^\infty}^{1/2} c^2 + \|f''\|_{L^\infty}^{3/2}  |c|) \le C\|f''\|_{L^\infty}^{3/2} |c| .
\eeq
In view of \eqref{2.40a},  \eqref{2.41},  \eqref{2.44}, \eqref{2.44b}, and \eqref{2.44a}, this finishes the proof.
\end{proof}


\section{Estimates on  $I_4^\pm$} \lb{S6}

Finally we will treat $I_4^\pm$.  Note that  $I_4^-= 6J_3^- -24 J_5^- + I_2^-$, where
\begin{align*}
J_5^- & := \int_\bbRr h(x)  f'''(x) \int_\bbR  \frac {y\, (f'(x) - f'(x-y))^4 (f(x) - f(x-y)) [(f(x) - f(x-y))^2 -  y^2] } { [y^2+(f(x) - f(x-y))^2]^4 }  \, dydx .
\end{align*}
Similarly to \eqref{2.11a} we now obtain
\begin{align*}
|J_5^-| & \le C \|f'\|_{C^1}^4  \int_{  \bbR^2 \times [0,1]}  h(x) | f'''(x)|\, | f''(x-sy)|   \min\{1,|y|^{-4} \} \, dxdyds 
\\ & \le C \|f'\|_{C^1}^4  \|f''\|_{\tilde L^2} \|f'''\|_{\tilde L^2} \int_{\bbR \times [0,1]}   \min\{1,|y|^{-4} \} \, dyds 
 \le C \|f'\|_{C^1}^4 (\|f''\|_{\tilde L^2}^2+ \|f'''\|_{\tilde L^2}^2) .
\end{align*}
This, \eqref{2.11}, and \eqref{2.11a} now yield
\beq \lb{2.65}
|I_4^-| \le C(1+  \|f'\|_{C^{1,\gamma}}^4)  (\|f''\|_{\tilde L^2}^2 + \|f'''\|_{\tilde L^2}^2).
\eeq

Term $I_4^+$ will again be the challenging one. As with $I_3^+$, split $I_4^+=I_4'+I_4'' + I_4'''+I_4''''$, where with $R_x:=\bbR\setminus[-\sqrt{f(x)}, \sqrt{f(x)}\,]$ and $\tilde h(x):=\chi_{[0,1]}(f(x))\, h(x)$ we have
\begin{align*}
I_4' & := \int_\bbRr \chi_{(1,\infty)}(f(x)) \, h(x) f'''(x)  \int_{\bbR} \left[ y\, (f'(x)+ f'(x-y)) \, \partial_{x}^3 A^+(x,y) +2y W(x,y) \right] \, dydx,
\\ I_4'' & := \int_\bbRr \tilde h(x) f'''(x) \, PV \int_{R_x} \left[ y\, (f'(x)+ f'(x-y)) \, \partial_{x}^3 A^+(x,y) +2y W(x,y) \right] \, dydx,
\\ I_4''' & :=  \int_\bbRr \tilde h(x) f'''(x) \int_{-\sqrt{f(x)}}^{\sqrt{f(x)}} \left[  y\, (f'(x)+ f'(x-y)) \, \partial_{x}^3 A^+(x,y) +2y W(x,y) \right] \, dydx,
\\ I_4'''' & := -  \int_\bbRr h(x) f'''(x) \int_{\bbR} 2 y W(x,y) \, dydx, 
\end{align*}
and
\[
W(x,y) := \frac { (f'''(x) + f'''(x-y)) (f(x) + f(x-y)) (f'(x) + f'(x-y))  } { [y^2+(f(x) + f(x-y))^2]^2 }.
\]
Note that $I_4''''=I_2^+$, so \eqref{2.10c} yields
\beq \lb{2.63}
|I_4''''| \le C(1+ \|f'\|_{C^1}^2)  \|f'''\|_{\tilde L^2}^2 .
\eeq


Next, the inside integrand in $I_4'$ is \eqref{2.62} below with $(x,x-y)$ in place of $(0,-y)$,  so
\beq\lb{2.63a}
\begin{split}
|I_4'| &\le  C \|f'\|_{C^1}^3  \int_{ \bbR^2} h(x)  |f'''(x)| \sum_{j=1}^2 \left(|f^{(j)}(x)| + |f^{(j)}(x-y)|\right) \min\{1,|y|^{-3} \} \, dxdyds 
\\ & 
 \le C \|f'\|_{C^1}^3 ( \|f'\|_{\tilde L^2}^2 + \|f''\|_{\tilde L^2}^2 + \|f'''\|_{\tilde L^2}^2) .
\end{split}
\eeq
We also have $I_4''= 6(J_3^+ + 6J_5^+ - 8 J_6^+) -24 J_8^+$, where
\[
J_8^+:= \int_\bbRr \tilde h(x) f'''(x) \int_{R_x}  \frac {y\, (f'(x) + f'(x-y))^4 (f(x) + f(x-y)) [(f(x) + f(x-y))^2 -  y^2]} { [y^2+(f(x) + f(x-y))^2]^4 }  \, dydx .
\]
Estimates from the last section show that
\[
|J_3^+ |+ |J_5^+| + |J_6^+| \le C  (1+\|f'\|_{C^{1,\gamma}}^3) (\|f''\|_{\tilde L^2}^2 +  \|f'''\|_{\tilde L^2}^2),
\]
while we  obtain
\begin{align*}
|J_8^+| & \le C (1+ \|f'\|_{C^1}^4)  \int_{\bbR^2 \times [0,1]}  h(x) | f'''(x)| (|f'(x)|+ | f''(x-sy)| ) \min\{1,|y|^{-4} \} \, dxdyds 
\\ & \le C (1+\|f'\|_{C^1}^4)  (\|f'\|_{\tilde L^2}^2 + \|f''\|_{\tilde L^2}^2 +  \|f'''\|_{\tilde L^2}^2)
\end{align*}
if inside the integral we split one factor $f'(x) + f'(x-y)$ into $2f'(x)$ and $f'(x-y) - f'(x)$, and then use \eqref{2.7} to estimate both integrals (the latter in the same way as $J_3^-$).  Therefore
\beq \lb{2.64}
|I_4'' | \le C(1+  \|f'\|_{C^{1,\gamma}}^4)   (\|f'\|_{\tilde L^2 }^2 + \|f''\|_{\tilde L^2 }^2+\|f'''\|_{\tilde L^2 }^2)
\eeq

It remains to estimate $I_4'''$.
Consider again $x=0$, and then we have
\beq\lb{2.62}
y(f'(0)+ f'(-y)) \, \partial_{x}^3 A^+(0,y) + 2 y W(0,y) = 18yP(-y) - 24y^3Q(-y)  - 24y U(-y) +  48y^3 V(-y)  ,
\eeq
where
\begin{align*}
U(y) &:= \frac { (f'(0) + f'(y))^4 (f(0) + f(y)) } { [y^2+(f(0) + f(y))^2]^3 },
\\ V(y) &:= \frac { (f'(0) + f'(y))^4 (f(0) + f(y)) } { [y^2+(f(0) + f(y))^2]^4 }.
\end{align*}
We can use Lemma \ref{L.2.2} to estimate the first two terms on the right-hand side of \eqref{2.62}.  Since the next lemma will deal with the other two terms, similarly to  \eqref{2.14} it follows that
\[
|I_4'''|\le C (1+\|f'\|_{C^1}^{4})  (\|f'\|_{\tilde L^2 }^2 + \|f''\|_{\tilde L^2 }^2+\|f'''\|_{\tilde L^2 }^2)
\]
by using these results for all $x\in\bbR$ instead of just $x=0$.  This, \eqref{2.65}, \eqref{2.63},  \eqref{2.63a}, and \eqref{2.64} then yield
\beq \lb{2.66}
|I_4^\pm | \le C (1+  \|f'\|_{C^{1,\gamma}}^4)   (\|f'\|_{\tilde L^2 }^2 + \|f''\|_{\tilde L^2 }^2+\|f'''\|_{\tilde L^2 }^2).
\eeq

\begin{lemma} \lb{L.2.3}
When $f(0)\le 1$, we have
\beq\lb{2.60a}
\left| \int_{- \sqrt{f(0)} }^{  \sqrt {f(0)} } \left[ 2y^3V(y) -y U(y) \right] dy \right| 
 \le  C (1+\|f'\|_{C^1}^{4}) \left( |f'(0)| + |f''(0)|   +  M_f(0) \right) .
\eeq
\end{lemma}

\begin{proof}
By  \eqref{2.6a} we again have
\beq\lb{2.60}
\begin{split}
U(y)-U(-y) & = U_1(y) + U_2(y) + U_3(y) ,
\\ V(y)-V(-y) & = V_1(y) + V_2(y) + V_3(y) ,
\end{split}
\eeq
where
\begin{align*}
U_1(y) & :=  \frac {  (f(0)+f( y))  ( 2cy+G_1(y) ) \sum_{n=0}^3 (f'(0)+f'(y))^n (f'(0)+f'(-y))^{3-n}   } {[y^2+(f(0)+f( y))^2]^3},
\\ U_2(y) & := \frac {(f'(0) + f'(-y))^4  (2by+G_2(y))} {[y^2+(f(0)+f( y))^2]^3 },
\\ U_3(y) & :=  - \frac {(f'(0) + f'(-y))^4 (f(0)+f( -y)) \, H_2(y) \, F_{1}(y) \, (2by+G_2(y))} 
{[y^2+(f(0)+f( y))^2]^3 \, [y^2+(f(0)+f( -y))^2]^3}
\end{align*}
and
\begin{align*}
V_1(y) & :=  \frac {  (f(0)+f( y))  ( 2cy+G_1(y) ) \sum_{n=0}^3 (f'(0)+f'(y))^n (f'(0)+f'(-y))^{3-n}   } {[y^2+(f(0)+f( y))^2]^4},
\\ V_2(y) & := \frac {(f'(0) + f'(-y))^4  (2by+G_2(y))} {[y^2+(f(0)+f( y))^2]^4 },
\\ V_3(y) & :=  - \frac {(f'(0) + f'(-y))^4 (f(0)+f( -y)) \, H_3(y) \, F_{1}(y) \, (2by+G_2(y))} 
{[y^2+(f(0)+f( y))^2]^4 \, [y^2+(f(0)+f( -y))^2]^4},
\end{align*}
with
\[
H_j(y):= \sum_{n=0}^j [y^2+(f(0)+f( y))^2]^n \, [y^2+(f(0)+f( -y))^2]^{j-n}
\]
(so in particular,  $H_1(y)=2y^2+F_2(y)$).
None of these terms is analogous to $S_1$ and $P_1$ because $U$ and $V$ do not have the $f''(0)+f''(y)$ term.  Instead $U_1+U_2+U_3$ and $V_1+V_2+V_3$ are treated the same way as $S_2+S_3$, $P_2+P_3$, and $Q_2+Q_3$ in Lemmas \ref{L.2.1} and \ref{L.2.2}.  We leave some of the details to the reader, but provide here the main estimates, discuss how to deal with an extra complication, and  show cancellation of the critical leading terms
 \begin{align*}
 \tilde U_4(y) &:= \frac { 128ab^3cy }{[y^2+4a^2]^3} + \frac { 32b^5y }{[y^2+4a^2]^3}  - \frac {768a^2b^5y } {[y^2+4a^2]^4}
\\ & = 32b^3(4ac+b^2) \frac { y^3-4a^2y }{[y^2+4a^2]^4} + 512a^2b^3 \frac {(2ac -b^2) y }{[y^2+4a^2]^4},
\\  \tilde V_4(y) &:= \frac { 128ab^3cy }{[y^2+4a^2]^4} + \frac { 32b^5y }{[y^2+4a^2]^4}  - \frac {1024a^2b^5y } {[y^2+4a^2]^5}
\\ & = 32b^3(4ac+b^2) \frac { 3y^3-20a^2y }{ 3[y^2+4a^2]^5} + {2048a^2b^3} \frac {(2ac -b^2) y }{3 [y^2+4a^2]^5}.
\end{align*}
These were obtained from $U_1+U_2+U_3$ and $V_1+V_2+V_3$ by dropping all the $G_j$ terms, and replacing all the $f(0)+f(\pm y)$ and $f'(0)+f'(\pm y)$ terms by $2a$ and $2b$, respectively.  

Doing the latter but keeping the $G_j$ terms defines terms $U_4$ and $V_4$, and we  first obtain an estimate on $U_1+U_2+U_3 - U_4$.  We can do this similarly to the analogous estimates for $S_2+S_3 - S_4$ and  $P_2+P_3 - P_4$, using \eqref{2.7}, \eqref{2.30c}, \eqref{2.33}, \eqref{2.35b}, and \eqref{2.42b}.  But if we simply estimate each $f'(0)+f'(\pm y)$  by $4(1+\|f'\|_{C^1})\sqrt a$ and each $b$ by $2 \|f''\|_{L^\infty}^{1/2} \sqrt a$, our final estimate will lack the term $|c|+M_f(0)$ from \eqref{2.36} because $U_2, U_3, V_2, V_3$ do not contain the factor $2c+g_0(-y)$ or $2cy+G_1(y)$, which can be estimated by $2(|c|+M_f(0)\sqrt a)$ and $2(|c|+M_f(0)\sqrt a)y$, respectively.  (This would add 1 to the last parenthesis in \eqref{2.60a}, which is not an issue when we consider solutions with $f'''\in \tilde L^2(\bbR)$ but would become one for those with $f'''\in L^2(\bbR)$, relevant to \eqref{1.11}.)  Instead, when we first replace $(f'(0)+f'(- y))^4$ by $(2b)^4$ in either of these terms and estimate the difference, we need to use
\[
|f'(-y)-f'(0)| \le 2(|c|+ M_f(0)\sqrt a )y \le 2 (|c|+ M_f(0) )y,
\]
which replaces the second line in \eqref{2.42b} by
\[
\left| (f'(0)+f'(\pm y))^n -(2b)^n \right|  \le C_n (1+\|f'\|_{C^1}^{n-1}) (|c|+ M_f(0) ) a^{(n-1)/2}   y .
\]
  In the case of $U_2$ we then also need to estimate
\[
\frac { 16b^4 (2by+G_2(y))} {[y^2+(f(0)+f( y))^2]^3 } - \frac { 16b^4  (2by+G_2(y))} {[y^2+4a^2]^3 },
\]
but doing this by only using  \eqref{2.33} and $|b|\le 2 \|f''\|_{L^\infty}^{1/2} \sqrt a$ is again insufficient.  However, we do obtain the desired estimate, with $|b|+ |c|+M_f(0)$ in place of $|c|+M_f(0)$, if we also employ the bound
\[
b^2 \le 4 (|b|+ |c| + M_f(0))a
\]
once.  This clearly holds when $|b|\le 4a$, while if $|b|> 4a$, then $f(\pm \frac {2a}b)\ge 0$ forces
\[
\frac b2\le \sup_{|y|\le 2a/|b|} |f'(y)-f'(0)| \le (|c|+ M_f(0) )2ab^{-1} .
\]
With a similar adjustment when estimating $U_3$,  we eventually obtain
\begin{align*}
| U_1(y)  +U_2(y)+U_3(y) - U_4(y)|  
 & \le C (1+\|f'\|_{C^1}^{4}) (|b|+|c| + M_f(0))     \frac{  a^2} {\max\{y,a\}^4} ,
\\ | V_1(y)  +V_2(y)+V_3(y) - V_4(y)|  
&  \le C (1+\|f'\|_{C^1}^{4}) (|b|+|c| + M_f(0))     \frac{  a^2} {\max\{y,a\}^6} 
\end{align*}
for $y\in[0,\sqrt a]$, so that
\[
 \int_0^{ \sqrt a} y | U_1(y) + U_2(y) +U_3(y) - U_4(y)| dy  \le  C (1+\|f'\|_{C^1}^{4}) (|b|+|c| + M_f(0))
\]
and
\[
 \int_0^{ \sqrt a} y^3 | V_1(y) + V_2(y) +V_3(y) - V_4(y)| dy  \le  C (1+\|f'\|_{C^1}^{4}) (|b|+|c| + M_f(0)).
\]
We easily also get
\begin{align*}
|U_4(y) - \tilde U_4(y)| & \le C(1+\|f'\|_{C^1}^{2})  \frac{ a^3} {\max\{y,a\}^6}  \int_{-y}^{y}|f'''(z)| dz 
\le C(1+\|f'\|_{C^1}^{2})  \frac{M_f(0) a} {\max\{y,a\}^3} ,
\\ |V_4(y) - \tilde V_4(y)| & \le C(1+\|f'\|_{C^1}^{2})  \frac{ a^3} {\max\{y,a\}^8}  \int_{-y}^{y}|f'''(z)| dz 
\le C(1+\|f'\|_{C^1}^{2})  \frac{M_f(0) a} {\max\{y,a\}^5}
\end{align*}
for $y\in[0,\sqrt a]$, so
\[
 \int_0^{ \sqrt a} \left( y |U_4(y) - \tilde U_4(y)| + y^3 |V_4(y) - \tilde V_4(y)| \right) dy
\le  C(1+\|f'\|_{C^1}^{2}) M_f(0).
\]

Finally, we now  have the crucial identity
\ \begin{align}
 & \frac{2 y^3  \tilde V_4(y)  - y\tilde U_4(y)}{32}  = -\partial_y \frac {  2b^3(4ac+b^2) y^5 }{3 [y^2+4a^2]^4} + \partial_y \frac {  b^3(4ac+b^2) y^3 }{3 [y^2+4a^2]^3}  
 +16a^2b^3 (2ac -b^2)\frac { 5y^4 - 12a^2y^2 }{ 3[y^2+4a^2]^5} \notag
\\ & \qquad\qquad\qquad = -\partial_y \frac {  2b^3(4ac+b^2) y^5 }{3 [y^2+4a^2]^4} + \partial_y \frac {  b^3(4ac+b^2) y^3 }{3 [y^2+4a^2]^3}  
-\partial_y \frac {  16a^2b^3 (2ac -b^2) y^3 }{3 [y^2+4a^2]^4}  , \lb{2.61}
\end{align}
hence
\[
\left| \int_0^{ \sqrt a} \left[ 2 y^3\tilde V_4(y) - y\tilde U_4(y)  \right] dy \right| 
\le  C b^2 \|f''\|_{L^\infty}^{3/2} (1+a)  \le C \|f''\|_{L^\infty}^{2} |b|.
\]
The result now follows from the above bounds and \eqref{2.60}.
\end{proof}


\section{A Priori Estimates for an Approximating Family of  Equations} \lb{S7}

We will find solutions to \eqref{1.5} as $\eps\to 0$ limits of solutions to the mollified problems
\begin{equation} \lb{7.1}
{\partial_t f_\eps} (x,t) = \phi_\eps * PV \int_\bbR \sum_\pm \frac{y\, (\partial_x(\phi_\eps*f_\eps)(x,t) \pm \partial_x(\phi_\eps*f_\eps)(x-y,t))}{y^2+((\phi_\eps*f_\eps)(x,t)\pm (\phi_\eps*f_\eps)(x-y,t))^2}   dy,
\end{equation}
where $\phi_\eps(x):=\frac 1\eps \phi(\frac x\eps)$ for $\eps\in (0,1)$ and $\phi\in C^\infty(\bbR)$ vanishes on $\bbR\setminus [-1,1]$, is even, decreasing on $[0,1]$ with $\phi(0)=1$ and $\|\phi'\|_{L^\infty}\le 2$, and $\int_{-1}^1 \phi(x)dx=1$.  
Dropping again $t$ in the arguments, denoting $g':=\partial_x g$,  and letting $F_\eps:=\phi_\eps*f_\eps$ for $f_\eps\ge 0$, we obtain
\[
\frac d{dt}  \left\| \sqrt{h}\, f_\eps''' \right\|_{L^2}^2 
 = 2( I_{1'}^+ + I_{1'}^- + I_{4'}^+ + I_{4'}^- ) + 6(I_{2'}^+ + I_{2'}^- + I_{3'}^+ + I_{3'}^-)  ,
\]
where with $A_\eps^\pm(x,y):= [y^2+(F_\eps(x)\pm F_\eps(x-y))^2]^{-1}$ we have
\begin{align*}
I_{1'}^\pm & := 
\int_\bbR  h(x) f_\eps'''(x) \, \phi_\eps*   PV \int_\bbR \frac{y\, (F_\eps''''(x)\pm F_\eps''''(x-y))}{y^2+(F_\eps(x)\pm F_\eps(x-y))^2} \, dydx ,
\\  I_{2'}^\pm & := 
\int_\bbR h(x)  f_\eps'''(x) \, \phi_\eps* \int_\bbR  y\, (F_\eps'''(x)\pm F_\eps'''(x-y)) \, \partial_x A_\eps^\pm(x,y) \, dydx ,
\\  I_{3'}^\pm & := 
\int_\bbR  h(x) f_\eps'''(x) \, \phi_\eps* \int_\bbR  y\, (F_\eps''(x)\pm F_\eps''(x-y)) \, \partial_{x}^2 A_\eps^\pm(x,y) \, dydx ,
\\  I_{4'}^\pm & := 
\int_\bbR  h(x) f_\eps'''(x) \, \phi_\eps* \int_\bbR  y\, (F_\eps'(x)\pm F_\eps'(x-y)) \, \partial_{x}^3 A_\eps^\pm(x,y) \, dydx. 
\end{align*}
As explained at the end of Section \ref{S2}, we can drop PV in the last three integrals, and below we will also drop it  in $I_{1'}^\pm$ (where it is only needed as $|y|\to\infty$).
Note that now $F_\eps$ does have the degree of regularity required for the arguments in the previous sections (it is in fact smooth), which carry over with the following adjustments.

Let us first consider $I_{1'}^\pm$.  From $\int_\bbR F(\phi*G) dx = \int_\bbR (\phi*F)G dx$ for even $\phi$ we have
\begin{align*}
I_{1'}^\pm  = J_{1'}^\pm + J_{0'}^\pm  & :=  
\int_\bbR  h(x) F_\eps'''(x)  \int_\bbR \frac{y\, (F_\eps''''(x)\pm F_\eps''''(x-y))}{y^2+(F_\eps(x)\pm F_\eps(x-y))^2} \, dydx 
\\ & + \int_{\bbR^2}  (h(z)-h(x)) \phi_\eps(x-z) f_\eps'''(z)  \int_\bbR \frac{y\, (F_\eps''''(x)\pm F_\eps''''(x-y))}{y^2+(F_\eps(x)\pm F_\eps(x-y))^2} \, dydx dz.
\end{align*}
We can estimate $J_{1'}^\pm$ in the same way as $I_1^\pm$, and similarly to \eqref{2.10} we  obtain
\beq \lb{7.2}
J_{1'}^+ + J_{1'}^- \le C(1+  \|f_\eps\|_{C^{2,\gamma}_\gamma}^2 ) \| f_\eps'''\|_{\tilde L^2}^2
\eeq
with some $\eps$-independent $C$, where we also used
\[
\| \phi_\eps *g\|_{\tilde L^2}^2 \le 2 \| g\|_{\tilde L^2}^2.
\]

On the other hand, we have 
\beq \lb{7.3}
J_{0'}^\pm = \int_{\bbR^2}  f_\eps'''(z) f_\eps'''(v)  \int_\bbR (h(z)-h(x)) \phi_\eps(x-z)  \int_\bbR \frac{y\, (\phi_\eps'(x-v)\pm \phi_\eps'(x-y-v))}{y^2+(F_\eps(x)\pm F_\eps(x-y))^2} \, dydx dzdv.
\eeq
From \eqref{2.1} 
we obtain 
\[
\left| \int_\bbR \frac{y}{y^2+(F_\eps(x)\pm F_\eps(x-y))^2} \, dy \right|
\le C \|F_\eps\|_{C^2_\gamma} \le C \|f_\eps\|_{C^2_\gamma},
\]
so 
\beq \lb{7.4}
\left| \int_\bbR \frac{y \phi_\eps'(x-v)}{y^2+(F_\eps(x)\pm F_\eps(x-y))^2} \, dy \right|
 \le \frac {C \|f_\eps\|_{C^2_\gamma}} {\eps^{2}}  \chi_{[0,\eps]}(|x-v|)
\eeq
because $\phi_\eps(x-z)=0$ when $|x-z|\ge \eps$.  When $|x-v|\ge 2\eps$, then 
\[
\left| \frac d{dy} \,\frac{y }{y^2+(F_\eps(x)\pm F_\eps(x-y))^2} \right| \le \frac {C(1+\|f_\eps'\|_{L^\infty})} {|x-v|^2}
\]
for all $y\in[x-v-\eps, x-v+\eps]$, so from  
\[
\int_{x-v-\eps}^{x-v+\eps}\phi_\eps'(x-y-v)=0 \qquad\text{and}\qquad \int_{x-v-\eps}^{x-v+\eps}|\phi_\eps'(x-y-v)|=2\eps^{-1}
\]
we get
\beq \lb{7.5}
\left| \int_\bbR \frac{y \phi_\eps'(x-y-v)}{y^2+(F_\eps(x)\pm F_\eps(x-y))^2} \, dy \right| 
\le  \frac {C(1+ \|f_\eps'\|_{L^\infty} )}{ |x-v|^2}.
\eeq
When $|x-v|\le 2\eps$, then 
from $|\phi_\eps'(x-y-v)-\phi_\eps'(x-v)| \le C \eps^{-3}|y|$ and the fact that \eqref{7.4} also holds with integration  over $[-y',y']$ for any $y'\ge 0$ instead of over $\bbR$,
\[
\left| \int_\bbR \frac{y \phi_\eps'(x-y-v)}{y^2+(F_\eps(x)\pm F_\eps(x-y))^2} \, dy \right|
= \left| \int_{-3\eps}^{3\eps} \frac{y \phi_\eps'(x-y-v)}{y^2+(F_\eps(x)\pm F_\eps(x-y))^2} \, dy \right|
 \le \frac {C(1+ \|f_\eps\|_{C^2_\gamma})} {\eps^{2}}.
\]
From this, \eqref{7.3}, \eqref{7.4}, \eqref{7.5}, and the properties of $h$ it now follows that
\[
|J_{0'}^\pm| \le \int_{x_0-4-\eps}^{x_0+4+\eps}  | f_\eps'''(z)| \int_{\bbR} |f_\eps'''(v)| \frac{C\eps (1+ \|f_\eps\|_{C^2_\gamma}) }{\max\{|z-v|,\eps\}^2} dvdz \le C (1+ \|f_\eps\|_{C^2_\gamma}) \| f_\eps'''\|_{\tilde L^2}^2
\]
because
\[
\int_\bbR \frac{\eps |f_\eps'''(v)| }{\max\{|z-v|,\eps\}^2} \,dv
\le \sum_{n=1}^\infty \frac\eps{2^{2(n-1)}\eps^2} \int_{-2^n\eps}^{2^n\eps} |f_\eps'''(z+w)| \,dw \le 4 M_{f_\eps}(z).
\]
Hence this and \eqref{7.2} show that
\beq \lb{7.6}
I_{1'}^+ + I_{1'}^- \le C(1+ \|f_\eps\|_{C^{2,\gamma}_\gamma}^2 ) \| f_\eps'''\|_{\tilde L^2}^2
\eeq

As for $I_{2'}^\pm$, similarly to \eqref{2.11} and \eqref{2.10c} we will obtain
\beq \lb{7.7}
|I_{2'}^\pm| \le C(1+ \|f_\eps'\|_{C^{1,\gamma}}^2 )  \| f_\eps'''\|_{\tilde L^2}^2   
\eeq
if we can prove the same bound for
\begin{align*}
J_{2'}^\pm & := \int_{\bbR^2}  (h(z)-h(x)) \phi_\eps(x-z) f_\eps'''(z)  \int_\bbR  y\, (F_\eps'''(x)\pm F_\eps'''(x-y)) \, \partial_x A_\eps^\pm(x,y) \, dydxdz
\\ & =  2 \int_{\bbR^2}  f_\eps'''(z) f_\eps'''(v)  \int_\bbR (h(z)-h(x)) \phi_\eps(x-z) 
\\ & \qquad  \int_\bbR \frac{y\, (\phi_\eps(x-v)\pm \phi_\eps(x-y-v))(F_\eps(x)\pm F_\eps(x-y))(F_\eps'(x)\pm F_\eps'(x-y))}{[y^2+(F_\eps(x)\pm F_\eps(x-y))^2]^2} \, dydx dzdv.
\end{align*}
Similarly to \eqref{7.3}, this will follow once we bound the inside integral by
\[
\frac {C  (1+  \|f_\eps'\|_{C^{1,\gamma}}^2 )} {\max\{|x-v|,\eps\}^{2}},
\]
and we do this analogously.  First,  the estimation of $K_1^\pm$ from \eqref{2.4} yields 
\beq \lb{7.8a}
\left| \int_{S\cup(-S)} \frac{y (F_\eps(x)\pm F_\eps(x-y))(F_\eps'(x)\pm F_\eps'(x-y))}{[y^2+(F_\eps(x)\pm F_\eps(x-y))^2]^2} \, dy \right|
\le C(1+ \|F_\eps'\|_{C^{1,\gamma}}^2 )  \le C(1+ \|f_\eps'\|_{C^{1,\gamma}}^2 ) 
\eeq
for any $S\subseteq[0,\infty)$, so 
\beq \lb{7.8}
\left| \int_\bbR \frac{y \phi_\eps(x-v) (F_\eps(x)\pm F_\eps(x-y))(F_\eps'(x)\pm F_\eps'(x-y))}{[y^2+(F_\eps(x)\pm F_\eps(x-y))^2]^2} \, dy \right|
\le   \frac {C  (1+  \|f_\eps'\|_{C^{1,\gamma}}^2 )} {\eps} \chi_{[0,\eps]}(|x-v|).
\eeq
When $|x-v|\ge 2\eps$, then 
\[
\left| \frac{y (F_\eps(x)\pm F_\eps(x-y))(F_\eps'(x)\pm F_\eps'(x-y))}{[y^2+(F_\eps(x)\pm F_\eps(x-y))^2]^2} \right| 
\le \left| \frac{F_\eps'(x)\pm F_\eps'(x-y)}{y^2+(F_\eps(x)\pm F_\eps(x-y))^2} \right| 
\le \frac {C\|f_\eps'\|_{L^\infty}} {|x-v|^2}
\]
for all $y\in[x-v-\eps, x-v+\eps]$, so
\beq \lb{7.9}
\left| \int_\bbR \frac{y \phi_\eps(x-y-v) (F_\eps(x)\pm F_\eps(x-y))(F_\eps'(x)\pm F_\eps'(x-y))}{[y^2+(F_\eps(x)\pm F_\eps(x-y))^2]^2} \, dy \right|
\le  \frac {C \|f_\eps'\|_{L^\infty}}{ |x-v|^2}.
\eeq
When $|x-v|\le 2\eps$, then from $|\phi_\eps(x-y-v)-\phi_\eps(x-v)| \le C \eps^{-2}|y|$ and \eqref{7.8} we obtain
\begin{align}
& \left| \int_\bbR  \frac{y \phi_\eps(x-y-v) (F_\eps(x)\pm F_\eps(x-y))(F_\eps'(x)\pm F_\eps'(x-y))}{[y^2+(F_\eps(x)\pm F_\eps(x-y))^2]^2} \, dy \right|   \lb{7.9a}
 \\ & \qquad = \left| \int_{-3\eps}^{3\eps} \frac{y \phi_\eps(x-y-v) (F_\eps(x)\pm F_\eps(x-y))(F_\eps'(x)\pm F_\eps'(x-y))}{[y^2+(F_\eps(x)\pm F_\eps(x-y))^2]^2} \, dy \right|
 \le \frac {C(1+  \|f_\eps'\|_{C^{1,\gamma}}^2 )} {\eps^{3/2}}. \notag
\end{align}
This is immediate when $\pm$ is $-$, while in the other case we use \eqref{2.7} to bound the last integral with $\phi_\eps(x-y-v)-\phi_\eps(x-v)$ in place of $\phi_\eps(x-y-v)$ by
\[
\frac {C (1+ \|F_\eps'\|_{C^1})}{\eps^2} \int_0^{3\eps}  \frac{ \max\{ y, \sqrt{ F_\eps (x)} \} } { \max\{ y,  F_\eps (x) \} } dy \le \frac {C (1+ \|f_\eps'\|_{C^1})}{\eps^{3/2}}.
\]
The required bound on $J_{2'}^\pm$ now follows from the above estimates.

Integral $I_{3'}^\pm$ is estimated in the same way as $I_{3}^\pm$  because we did not use integration by parts in $x$ or  symmetrization in the process; instead, estimates on all the relevant integrals immediately followed after we obtained appropriate bounds on their inside integrals or integrands.
The first inequality in \eqref{2.11a} becomes
\[
|J_{3'}^-|
 \le C \|F_\eps'\|_{C^1}^2  \int_{ \bbR^2 \times [0,1]} h(x)  |f_\eps'''(x)| \, (\phi_\eps * |F_\eps'''|)(x-sy) \min\{1,|y|^{-2} \} \, dxdyds, 
\]
which yields the same estimate on $J_{3'}^-$ as we obtained for $J_{3}^-$, with $f_\eps$ in place of $f$.  Bounds on the terms analogous to $J_{3}^+$ and $J_{4}^\pm$ are adjusted in the same way.  The term analogous to $J_5^+$ will be bounded by
\[
 C (1+\|F_\eps'\|_{C^{1,\gamma}}^3)  \int_\bbRr h(x)  |f_\eps'''(x)| (\phi_\eps*|F_\eps''|)(x)    \,dx 
 \le  C  (1+\|f_\eps'\|_{C^{1,\gamma}}^3) \|f_\eps''\|_{\tilde L^2} \|f_\eps'''\|_{\tilde L^2},
\]
those corresponding to $J_6^+$, $J_7^+$, and $I'_3$ are treated the same way, and  \eqref{2.14} becomes
\begin{align*}
|I_{3'}'''| & \le C (1+\|F_\eps'\|_{C^1}^{4})  \int_\bbRr h(x)  |f_\eps'''(x)| (\phi_\eps*(|F_\eps''|+M_{F_\eps}))(x)    \,dx
\\ & \le C (1+\|f_\eps'\|_{C^1}^{4})  (\|f_\eps''\|_{\tilde L^2 }^2 +\|f_\eps'''\|_{\tilde L^2 }^2 ) .
\end{align*}
However, in all those terms that constitute $I_{3'}^+$, the estimates on the inside integrals/integrands carry over only when $|f(x+z)-f(x)|\le\frac 12$ for all $|z|\le\eps$ because in the relevant bounds in Section \ref{S5} we assumed that $f(x)$ is away from either 0 or $\infty$.  We therefore need to assume here that $\eps\le \frac 12\|f_\eps'\|_{L^\infty}^{-1}$.

The same adjustments apply to the terms constituting $I_{4'}^\pm$  (except for $I_{2'}^\pm$, which is bounded by \eqref{7.7}) in a decomposition analogous to that of $I_4^\pm$ in Section \ref{S6}. These estimates, together with \eqref{7.6} and \eqref{7.7}, finally yield
\beq\lb{7.10}
\frac d{dt}  \left\| \sqrt{h}\, f_\eps''' \right\|_{L^2}^2 
\le C (1+  \|f_\eps\|_{C^{2,\gamma}_\gamma}^4 ) \, (\|f_\eps'\|_{\tilde L^2 }^2 + \|f_\eps''\|_{\tilde L^2 }^2 + \|f_\eps'''\|_{\tilde L^2 }^2) 
\eeq
in place of \eqref{2.3}, provided $\eps\le \frac 12\|f_\eps'\|_{L^\infty}^{-1}$.  

\section{Uniqueness of Solutions} \lb{S8}

Assume that $f_1,f_2\ge 0$ are classical solutions to \eqref{1.5} on $\bbR\times[0,T]$ that both satisfy \eqref{1.7}.  Then the second claim in Theorem \ref{T.1.1}(ii) holds for $f_j$ ($j=1,2$)  because \eqref{2.70a} yields
\beq \lb{8.7}
\|f_j(\cdot,\tau)-f_j(\cdot,0)\|_{L^\infty} \le C \int_0^\tau s_j(t)(1+ s_j(t)) \,dt
\eeq
for all $\tau\in[0,T]$, where $s_j(t):= \|f_j(\cdot,t)\|_{\tilde H^3_\gamma(\bbR)}^2$.
Therefore $f:=f_1-f_2\in L^\infty([0,T];\tilde H^3(\bbR))$ whenever $f_1(\cdot,0)- f_2(\cdot,0)\in \tilde H^3(\bbR)$, and for all $t\in[0,T)$ we have
\[
\frac d{dt}  \left\| \sqrt{h}\, f (\cdot,t) \right\|_{L^2}^2
=  2 (I_5^+ + I_5^- + I_6^+ + I_6^-  )
\]
with $h:=h_0(\cdot-x_0)$ as above, where (we again drop $t$ in the notation and let $g':=\partial_x g$)
\begin{align*}
 I_5^\pm &:= \int_\bbRr h(x)  f(x)  \, PV \int_\bbR \frac {y\, (f'(x)\pm f'(x-y))}{y^2+(f_1(x)\pm f_1(x-y))^2} \, dydx ,
 \\ I_6^\pm &:= \int_\bbRr h(x)  f(x)  \int_\bbR  y B^\pm(x,x-y) \, dydx = \int_\bbRr h(x)  f(x)  \int_\bbR  (x-y) B^\pm(x,y) \, dydx ,
\end{align*}
and with $f_3:=f_1+f_2$ we  have
\begin{align*}
B^\pm & (x,y) := \frac { f_2'(x)\pm f_2'(y)} {(x-y)^2+(f_1(x)\pm f_1(y))^2} 
- \frac {f_2'(x)\pm f_2'(y)} {(x-y)^2+(f_2(x)\pm f_2(y))^2}
\\ & =  - \frac {  (f_2'(x)\pm f_2'(y)) \,(f_3(x) \pm f_3(y) ) \, (f(x)\pm f(y))} {[(x-y)^2+(f_1(x)\pm f_1(y))^2]\, [(x-y)^2+(f_2(x)\pm f_2(y))^2]} =: B_0^\pm(x,y) \, (f(x)\pm f(y)).
\end{align*}
The arguments in Section \ref{S3} estimating $I_1^\pm$ identically apply to $I_5^\pm$ with $(f,f_1)$ in place of $(g,f)$, and just as \eqref{2.10} they now prove
\beq \lb{8.2}
I_5^+ + I_5^- \le C(1+  \|f_1\|_{C^{2,\gamma}_\gamma}^2 ) \| f\|_{\tilde L^2_{x_0}}^2   .
\eeq

Symmetrization shows that
\begin{align*}
\int_\bbRr h(x)  f(x) \, PV  \int_\bbR & (x-y) B_0^\pm(x,y)f(y) \, dydx
= \frac 12 \int_{\bbRr^2} (h(x)-h(y))  (x-y) B_0^\pm(x,y)  f(x)  f(y) \, dydx,
\end{align*}
so with $E:=[x_0-4,x_0+4]$ we have
\begin{align*}
\left| \int_\bbRr h(x)  f(x) \, PV  \int_\bbR  (x-y) B_0^-(x,y)f(y) \, dydx \right|
& \le C \|f_2'\|_{C^1} \int_{\bbRr^2} \frac{\chi_{E}(x)+ \chi_{E}(y)} {\max\{1,|x-y|^2  \}} | f(x)| \,|  f(y)| \, dydx
\\ \le C \|f_2'\|_{C^1} \int_{\bbRr^2} \frac{\chi_{E}(x)+ \chi_{E}(y)} {\max\{1,|x-y|^2  \}}f(x)^2 \, dydx
& \le C \|f_2'\|_{C^1}  \| f\|_{\tilde L^2_{x_0}}^2   .
\end{align*}

On the other hand, letting $\hat y:=\min\{|y|,1\}$, $a_j:=f_j(x)$ ($j=1,2$),  $a_3:=\max\{a_1,a_2\}$, and $\lambda:=1+\|f_1'\|_{C^1}+ \|f_2'\|_{C^1}$, 
changing variables back via $y\leftrightarrow x-y$, and using also \eqref{2.7} yields
\begin{align*}
& \left| \int_\bbRr h(x)  f(x) \, PV  \int_\bbR  (x-y) B_0^+(x,y)f(y) \, dydx \right|
\\ & \qquad \qquad \le C \lambda^2 \int_{\bbRr^2} \frac{ (\chi_E(x)+\chi_E(x-y)) \hat y  |y| \max\{\hat y,\sqrt{\hat a_2} \} \max\{\hat y|y|,a_3 \} } {\max\{|y|,a_1 \}^2  \max\{|y|,a_2 \}^2 } | f(x)| \,|  f(x-y)| \, dydx,
\end{align*}
where we used $|h(x)-h(x-y)|\le  \hat y(\chi_E(x)+\chi_E(x-y))$ and  
\beq\lb{8.2b}
f_j(x-y)\le C (1+\|f_j'\|_{C^1}) (a_j+\min\{y^2,|y|\}) \le C (1+ \|f_j'\|_{C^1}) \max\{\hat y|y|,a_j \} 
\eeq
 for $j=1,2$.  If we employ $| f(x)| \,|  f(x-y)|\le f(x)^2+f(x-y)^2$, split the integral in two accordingly, and change variables $x\leftrightarrow x-y$ in the second, we find that 
 \begin{align*}
& \left| \int_\bbRr h(x)  f(x) \, PV  \int_\bbR  (x-y) B_0^+(x,y)f(y) \, dydx \right|
\\ & \qquad \qquad \le C \lambda^2 \int_{\bbR} f(x)^2 \int_{\bbRr} \frac{ (\chi_E(x)+\chi_E(x-y)) \hat y  |y| \max\{\hat y,\sqrt{\hat a_2} \} \max\{\hat y|y|,a_3 \} } {\max\{|y|,a_1 \}^2  \max\{|y|,a_2 \}^2 }  \, dydx,
\end{align*}
 so we need to estimate the inside integral.  
 It is clearly bounded by
\[
4\int_{0}^1 \frac{ \max\{ y,\sqrt{a_3} \}^3} { \max\{y,a_3 \}^2 } dy 
+ 4\int_1^\infty  \frac{ 1 } { y^2 } dy \le C
\]
for any $x\in\bbR$,
while for $|x-x_0|\ge 5$ it is bounded by
\[
 \int_{|x-x_0|-4}^{|x-x_0|+4} \frac{   1  } {y^2}  \, dy \le \frac C{|x-x_0|^2}.
\]
From the above estimates it now follows that
\beq\lb{8.2a}
\left| \int_\bbRr h(x)  f(x) \, PV  \int_\bbR  (x-y) B_0^\pm (x,y)f(y) \, dydx \right| \le C\lambda^2 \| f\|_{\tilde L^2_{x_0}}^2 . 
\eeq

Finally, we will show that
\beq\lb{8.0}
\left| \int_\bbRr h(x)  f(x)^2  \, PV \int_\bbR  (x-y) B_0^\pm(x,y) \, dydx \right| 
\le C \lambda_\gamma^2  \| f\|_{\tilde L^2_{x_0}}^2 ,
\eeq
with  $\lambda_\gamma:=1+\|f_1\|_{C^{2,\gamma}_\gamma}+ \|f_2\|_{C^{2,\gamma}_\gamma}$
This and \eqref{8.2a} then yield
\beq\lb{8.1a}
|I_6^\pm| \le C(1+\|f_1\|_{C^{2,\gamma}_\gamma}^2+ \|f_2\|_{C^{2,\gamma}_\gamma}^2) \| f\|_{\tilde L^2_{x_0}}^2,
\eeq
which together with \eqref{8.2} implies
\beq\lb{8.1}
\frac d{dt}  \left\| \sqrt{h}\, f (\cdot,t)\right\|_{L^2}^2 \le C \left(1+  \|f_1(\cdot,t)\|_{C^{2,\gamma}_\gamma}^2 + \|f_2(\cdot,t)\|_{C^{2,\gamma}_\gamma}^2 \right) \| f(\cdot,t)\|_{\tilde L^2_{x_0}}^2   .
\eeq
Uniformity of \eqref{8.1} in $x_0\in\bbR$ (recall that $h=h_0(\cdot -x_0)=:h_{x_0}$) and 
\[
\|f(\cdot,t)\|_{\tilde L^2} \le C \sup_{x_0\in\bbR}  \left\| \sqrt{h_{x_0}}\, f(\cdot,t) \right\|_{L^2}
\]
yield
\beq\lb{8.0b}
\frac d{dt}  \sup_{x_0\in\bbR}  \left\| \sqrt{h_{x_0}}\, f(\cdot,t) \right\|_{L^2} 
\le C \left(1+  \|f_1(\cdot,t)\|_{C^{2,\gamma}_\gamma}^2 + \|f_2(\cdot,t)\|_{C^{2,\gamma}_\gamma}^2 \right) \sup_{x_0\in\bbR}  \left\| \sqrt{h_{x_0}}\, f(\cdot,t) \right\|_{L^2}   
\eeq
 for all $t\in[0,T]$.
So Gr\" onwall's inequality shows that $f_1\equiv f_2$ whenever $f_1(\cdot,0)\equiv f_2(\cdot,0)$, proving the first claim in Theorem \ref{T.1.1}(ii).  From $\|\psi\|_{\tilde H^3_\gamma}<\infty$ and \eqref{2.70a} we also obtain \eqref{1.9}.



To prove \eqref{8.0}, it clearly suffices to show 
\beq\lb{8.0a}
\left| \, PV  \int_{S\cup(-S)}  y B_0^\pm(x,x-y) \, dy \right| \le C \lambda_\gamma^2
\eeq
for any $S\subseteq[0,\infty)$.  Using a version of \eqref{2.6a} with an extra term in the  denominator yields
\beq\lb{8.4}
y |B_0^+(x,x+y)- B_0^+(x,x-y)| 
 \le C \lambda B_1(x,y) + C \lambda^2 \sum_{n=2}^4 B_n(x,y)
\eeq
for $y\ge 0$, where (recall \eqref{2.7}, $a_j:=f_j(x)$ ($j=1,2$), and  $a_3:=\max\{a_1,a_2\}$)
\begin{align*}
B_1(x,y) & := \frac{y^2 \,  (f_3(x) + f_3(x-y)) }   {\max\{ y,a_1\}^2 \max\{ y,a_2\}^2} ,
\\ B_2(x,y) & := \frac{y^2 \, \max \big\{ \hat y, \hat a_2^{1/2} \big\} \, \max \big\{ \hat y, \hat a_3^{1/2} \big\} }  {\max\{ y,a_1\}^2 \max\{ y,a_2\}^2} ,
\\ B_3(x,y) & := \frac{y^2 \, \max \big\{ \hat y, \hat a_2^{1/2} \big\} \, (f_3(x) + f_3(x+y) ) \, (2f_1(x) + f_1(x+y) + f_1(x-y)) \,  \max \big\{ \hat y, \hat a_1^{1/2} \big\}  }  {[y^2+(f_1(x)+ f_1(x-y))^2]\, [y^2+(f_1(x)+ f_1(x+y))^2] \,  [y^2+(f_2(x)+ f_2(x-y))^2]} ,
\\ B_4(x,y) & := \frac{y^2 \, \max \big\{ \hat y, \hat a_2^{1/2} \big\} \,  (f_3(x) + f_3(x+y) ) \, (2f_2(x) + f_2(x+y) + f_2(x-y)) \,  \max \big\{ \hat y, \hat a_2^{1/2} \big\}  }  {[y^2+(f_1(x)+ f_1(x+y))^2]\, [y^2+(f_2(x)+ f_2(x-y))^2] \,  [y^2+(f_2(x)+ f_2(x+y))^2]} .
\end{align*}
We have
\begin{align*}
\int_0^\infty B_1(x,y)  dy \le C\lambda \int_0^1 \frac {\max\{y^2, a_3\} } { \max\{ y,a_3\}^2}  \, dy, 
+ \int_1^\infty \frac { 2a_3 + C\lambda_\gamma y^{1-\gamma}} { \max\{ y,a_3\}^2}  \, dy & \le C\lambda_\gamma
\\ \int_0^\infty B_2(x,y)  dy \le \int_0^{\infty} \frac {\max\{\hat y^2,\hat a_3\} } { \max\{ y,a_3\}^2}  \, dy & \le C ,
\\  \int_0^\infty B_4(x,y)  dy \le \int_0^{\infty} \frac {\max\{\hat y^2,\hat a_2\}  } { \max\{ y,a_2\}^2}  \, dy &  \le C ,
\end{align*}
where in the last line we used  $\max \{ \hat yy, a_j\}\le \max \{ y, a_j\}$.
When $a_1\ge a_2$, we also obtain
\[
 \int_0^\infty B_3(x,y)  dy \le \int_0^{\infty} \frac {\max\{\hat y^2,\hat a_1\}  } { \max\{ y,a_1\}^2}  \, dy \le C ,
\]
while in the case $a_2\ge a_1$ we have
\begin{align*}
 \int_0^\infty B_3(x,y)  dy \le \int_0^{\infty} \frac {\max\{\hat y,\hat a_1^{1/2}\} \max\{\hat y,\hat a_2^{1/2}\}  } { \max\{ y,a_1\} \max\{ y,a_2\}}  \, dy \le C
\end{align*}
because in the last inequality we can assume without loss that $a_2\le 1$, and then use
\begin{align*}
\int_0^{\sqrt{a_1}} \frac {  a_1^{1/2} a_2^{1/2}  } {  \max\{ y,a_1\} \max\{ y,a_2\}}  \, dy
& \le C,
\\ \int_{\sqrt{a_1}}^{1} \frac { \max\{  y, a_2^{1/2}\}  } {  \max\{ y,a_2\}}  \, dy 
 \le \int_{0}^{1} \frac { \max\{ y,  a_2^{1/2}\}  } {  \max\{ y,a_2\}}  \, dy &\le C.
\end{align*}
These bounds and \eqref{8.4} yield \eqref{8.0a} when $\pm$ is $+$.  
To prove \eqref{8.0a} with $-$, write the numerator of $B_0^-(x,x-y)$ as
\begin{align*}
 \left[ f_2'(x) - f_2'(x-y) \right] \, \left[f_3(x) - f_3(x-y) -f_3'(x)y \right] 
+  y f_3'(x) \left[ f_2'(x) - f_2'(x-y) \right] .
\end{align*}
The first term is bounded by $C\lambda^2 \hat y^2y$, yielding the desired estimate for the fraction involving it directly.  The integral involving the second fraction is evaluated analogously to
\[
\int_S y |B_0^-(x,x+y)- B_0^-(x,x-y)| dy ,
\]
and then we use a version of \eqref{2.6a} with one term in the numerator and two in the  denominator, together with the bounds
\begin{align*}
 \left| 2f_2'(x) - f_2'(x+y)-f_2'(x-y) \right| & \le 2\|f_2'\|_{C^{1,\gamma}} |\hat y|^{1+\gamma} ,
 \\  |f_2'(x) - f_2'(x-y)| \, |2f_j(x)- f_j(x+y) - f_j(x-y)|  & \le C\lambda^2 \hat y^2y.
\end{align*}



It remains to prove the last claim in Theorem \ref{T.1.1}(ii), which we do at the end of Section~\ref{S10}.

\section{Estimates on the Difference of Approximating Solutions} \lb{S9}

We will construct solutions to \eqref{1.5} as $\eps\to 0$ limits of appropriate solutions to \eqref{7.1}, so we first need to extend the estimates from Section \ref{S8} to involve solutions $f_1$ and $f_2$ to \eqref{7.1} with $\eps\in(0,\frac 12)$ and $\eps'\in(0,\frac 12)$, respectively.  Let again $f:=f_1-f_2$ and drop $t$ in the notation.  Let
\beq\lb{9.0}
H^\pm_f (x,y):=\frac{y\, (f'(x) \pm f'(x-y))}{y^2+(f(x)\pm f(x-y))^2} ,
\eeq
so that
\[
{\partial_t f} (x) = \phi_{\eps} * PV \int_\bbR (H^-_{\phi_\eps * f_1}(x,y) + H^+_{\phi_\eps * f_1}(x,y) ) dy
\,- \, \phi_{\eps'} * PV \int_\bbR (H^-_{\phi_{\eps'} * f_2}(x,y) + H^+_{\phi_{\eps'} * f_2}(x,y) ) dy
\]
(we drop PV below).  Then
\[ 
\frac d{dt}  \left\| \sqrt{h}\, f \right\|_{L^2}^2 
 = 2( I_{5'}^+ + I_{5'}^- + I_{6'}^+ + I_{6'}^- + I_{7'}^+ + I_{7'}^-)
\]
where $F_j:=\phi_\eps*f_j$ ($j=1,2$),  $F:=\phi_\eps*f=F_1-F_2$,
\begin{align*}
 I_{5'}^\pm &:= \int_\bbRr h(x)  f(x) \,  \phi_\eps * \int_\bbR \frac {y\, (F'(x)\pm F'(x-y))}{y^2+(F_1(x)\pm F_1(x-y))^2} \, dydx ,
 \\ I_{6'}^\pm &:= \int_\bbRr h(x)  f(x) \, \phi_\eps* \int_\bbR  y B_5^\pm(x,x-y) \, dydx ,
  \\ I_{7'}^\pm &:= \int_\bbRr h(x)  f(x) 
  \left( \phi_\eps*  \int_\bbR  y B_7^\pm(x,y) \, dydx - \phi_{\eps'}*  \int_\bbR  y B_8^\pm(x,y) \, dydx \right),
\end{align*}
and with $F_3:=F_1+F_2$ and $F_{2'}:=\phi_{\eps'}*f_2$ we let
\begin{align*}
B_5^\pm & (x,y) := \frac {F_2'(x)\pm F_2'(y)} {(x-y)^2+(F_1(x)\pm F_1(y))^2} 
- \frac {F_2'(x)\pm F_2'(y)} {(x-y)^2+(F_2(x)\pm F_2(y))^2}
\\ & =  - \frac {  (F_2'(x)\pm F_2'(y)) \,(F_3(x) \pm F_3(y) ) \, (F(x)\pm F(y))} {[(x-y)^2+(F_1(x)\pm F_1(y))^2]\, [(x-y)^2+(F_2(x)\pm F_2(y))^2]} =: B_{6}^\pm(x,y) \, (F(x)\pm F(y)),
\\ B_7^\pm & (x,y) := 
 \frac {F_2'(x)\pm  F_2'(x-y)} {y^2+(F_2(x)\pm F_2(x-y))^2} ,
\\  B_8^\pm & (x,y) :=   \frac {F_{2'}'(x)\pm F_{2'}'(x-y)} {y^2+(F_{2'}(x)\pm F_{2'}(x-y))^2} .
\end{align*}
Similarly to \eqref{8.2}, we obtain
\beq \lb{9.1}
I_{5'}^+ + I_{5'}^- \le C(1+  \|f_1\|_{C^{2,\gamma}_\gamma}^2 ) \| f\|_{\tilde L^2}^2
\eeq
by applying the argument yielding \eqref{7.6} to $(f,F_1)$ in place of $(f_\eps''',\phi_\eps*f_\eps)$.

Integrals $I_{6'}^\pm$ can also be estimated in the same way as $I_6^\pm$ in Section \ref{S8} and we obtain
\beq\lb{9.2}
|I_{6'}^\pm| \le C(1+\|f_1\|_{C^{2,\gamma}_\gamma}^2+ \|f_2\|_{C^{2,\gamma}_\gamma}^2) \| f\|_{\tilde L^2}^2,
\eeq
provided we prove the same bound for the integrals
\begin{align*}
J_9^\pm & := \int_{\bbR^2}  (h(z)-h(x)) \phi_\eps(x-z) f(z)   \int_\bbR (x- y)\, B_{6}^\pm(x,y)F(y) \, dydx dz
\\ & = \int_{\bbR^2}  f(z)f(v)   \int_\bbR (h(z)-h(x)) \phi_\eps(x-z)   \int_\bbR y\, \phi_\eps(x-y-v) B_{6}^\pm(x,x-y)  \, dydx dzdv  
\end{align*}
that arise when symmetrizing the integrals at the start of the argument.  We obtain
\[
|J_9^\pm| \le C (1+  \|F_1'\|_{C^{1,\gamma}}^2 + \|F_2'\|_{C^{1,\gamma}}^2 )  \| f\|_{\tilde L^2}^2  
\le C (1+  \|f_1'\|_{C^{1,\gamma}}^2 + \|f_2'\|_{C^{1,\gamma}}^2 )  \| f\|_{\tilde L^2}^2  
\]
via the argument bounding $J_{2'}^\pm$ in Section \ref{S7}.  This involves first bounding the last $dydx$ integral above with $\phi_\eps(x-v)$ in place of $\phi_\eps(x-y-v)$  (see \eqref{7.8}, but now use \eqref{8.0a} instead of \eqref{7.8a}), then bounding the original integral when $|x-v|\ge 2\eps$  (see \eqref{7.9}), and finally bounding the integral with $\phi_\eps(x-y-v)-\phi_\eps(x-v)$ when $|x-v|\le 2\eps$ (see \eqref{7.9a}, and again use \eqref{8.0a} instead of \eqref{7.8a}).  

We finally claim that 
\beq\lb{9.3}
|I_{7'}^\pm|  \le C (1+  \|f_2\|_{C^{2,\gamma}_\gamma}^2 ) \sqrt{\eps+\eps'} \, \| f\|_{\tilde L^2}  
\eeq
(note that $\| f\|_{\tilde L^2}$ is not squared here).  Assume that $\eps\ge\eps'$, as both cases are identical.  Since
\[
\|\phi_\eps *F - \phi_{\eps'} *G\|_{L^\infty} 
\le  \|(\phi_\eps  - \phi_{\eps'}) *F\|_{L^\infty}  +  \|\phi_\eps *(F-G)\|_{L^\infty} 
\le \eps \|F'\|_{L^\infty} + \|F-G\|_{L^\infty},
\]
to prove \eqref{9.3}, it suffices to show
\begin{align}
\left| \int_\bbR  y \partial_x B_7^\pm(x,y) \, dy \right| & \le C (1+   \|f_2\|_{C^{2,\gamma}_\gamma}^2 ),  \lb{9.4}
\\ \left| \int_\bbR  y (B_7^\pm(x,y) - B_9^\pm(x,y)) \, dy \right|  &  \le C (1+   \|f_2\|_{C^2_\gamma}^2 )\,\sqrt\eps,  \lb{9.5}
\\ \left| \int_\bbR  y (B_8^\pm(x,y) - B_9^\pm(x,y)) \, dy \right|  &  \le C (1+   \|f_2'\|_{C^1}^2 )\,\sqrt\eps,  \lb{9.6}
\end{align}
where
\[
B_9^\pm (x,y) := 
 \frac {F_{2'}'(x)\pm F_{2'}'(x-y)} {y^2+(F_2(x)\pm F_2(x-y))^2}.
\]

We have
\[
y \partial_x B_7^\pm (x,y) = 
 \frac {y(F_2''(x)\pm  F_2''(x-y))} {y^2+(F_2(x)\pm F_2(x-y))^2} - 2  \frac {y(F_2'(x)\pm  F_2'(x-y))^2(F_2(x)\pm  F_2(x-y))} {[y^2+(F_2(x)\pm F_2(x-y))^2]^2},
\]
 and from \eqref{2.70} we see that
\[
 \left|   \int_{\bbR} \frac{y\, ( F_2''(x)\pm F_2''(x-y))}{y^2+(F_2(x)\pm F_2(x-y))^2} \, dy \right| \le C (1+ \|F_2\|_{C^{2,\gamma}_\gamma}^2).
\]
Hence \eqref{9.4} will follow once we show
\[
 \left|   \int_{\bbR} \frac {y(F_2'(x)\pm  F_2'(x-y))^2(F_2(x)\pm  F_2(x-y))} {[y^2+(F_2(x)\pm F_2(x-y))^2]^2} \, dy \right| \le C (1+ \|F_2'\|_{C^1}^2).
\]
This is immediate when $\pm$ is $-$, while \eqref{2.7} shows that in the other case the left-hand side is bounded by
\[
C (1+ \|F_2'\|_{C^1}^2) \int_0^\infty \frac {\max\{\hat y^2,F_2(x)\}}  {\max\{y,F_2(x)\}^2} \, dy \le C (1+ \|F_2'\|_{C^1}^2) .
\]


To prove \eqref{9.5}, we use
\[
| (F_{2}'(x)\pm F_{2}'(x-y)) - (F_{2'}'(x)\pm F_{2'}'(x-y)) | \le C \|f_2''\|_{L^\infty} \, \eps
\]
for $|y|\ge \eps$ to see that
\begin{align}
 \int_{-1}^1  |y|\, |B_7^-(x,y) - B_9^-(x,y)| \, dy & \le  C \|f_2''\|_{L^\infty} \int_0^1  \frac {y\min\{ y,\eps\}} {y^2} \,dy \le C \|f_2''\|_{L^\infty} \,\eps\,|\ln \eps|,  \notag
\\  \int_{\eps\le |y|\le 1} | y|\,  |B_7^+(x,y) - B_9^+(x,y)| \, dy & \le   C \|f_2''\|_{L^\infty} \,\eps\,|\ln \eps|,  \lb{9.6b}
\end{align}
while \eqref{2.7} shows 
that
\[
\int_{0}^\eps  y |B_j^+(x,y) - B_j^+(x,-y)| \, dy  \le   C (1+\|f_2'\|_{C^1}^2) \int_0^\eps  \frac {y \max \big\{ y, \sqrt{F_2(x)} \big\} } {\max\{y,F_2(x)\}^2} \,dy
\le C (1+\|f_2'\|_{C^1}^2) \,\sqrt\eps
\]
holds for $j=7,9$.
We note that  we gave up $\sqrt\eps$ here by only estimating $|F_{2}'(x)+ F_{2}'(x-y)|\le 2\|f_2'\|_{L^\infty}$ and $|F_{2'}'(x)+ F_{2'}'(x-y)|\le 2\|f_2'\|_{L^\infty}$ in one of the terms in the integrand, but this will not cause a problem.
Finally, we  have
\[
\left| \int_{ |y|\ge 1}    \frac {y(F_{2}'(x)- F_{2'}'(x))} {y^2+(F_2(x)\pm F_2(x-y))^2} \, dy \right|
\le C \|f_2\|_{C^2_\gamma} \|F_2'-F_{2'}'\|_{L^\infty} \le C \|f_2\|_{C^2_\gamma}^2\, \eps
\]
by \eqref{2.1}, while integration by parts as in \eqref{2.1a} yields
\[
\left| \int_{ |y|\ge 1}    \frac {y(F_{2}'(x-y)- F_{2'}'(x-y))} {y^2+(F_2(x)\pm F_2(x-y))^2} \, dy \right|
\le C \|F_2-F_{2'}\|_{L^\infty} (1+\|F_2'\|_{L^\infty}) 
\le  C (1+\|f_2'\|_{L^\infty}^2)\, \eps.
\]
These estimates collectively yield \eqref{9.5}.  

Inequality \eqref{9.6} is proved similarly, using
\[
| (F_{2}(x)\pm F_{2}(x-y)) - (F_{2'}(x)\pm F_{2'}(x-y)) | \le C \|f_2'\|_{L^\infty} \, \eps
\]
and with no need to  treat integrals over $\{|y|\ge 1\}$ separately.  But there is an extra complication in the bound corresponding to \eqref{9.6b}.
Namely, with $G_2(x,y):= F_{2}(x)+ F_{2}(x-y)$ and $G_{2'}(x,y):= F_{2'}(x)+ F_{2'}(x-y)$ we have
\[
B_8^+(x,y) - B_9^+(x,y) =  
\frac {\partial_x G_{2'}(x,y) (G_{2}(x,y)+G_{2'}(x,y)) \, (G_{2}(x,y)-G_{2'}(x,y)) }  {[y^2+G_2(x,y)^2] \, [y^2+G_{2'}(x,y)^2]}
\]
and $|G_{2}(x,y)-G_{2'}(x,y)|\le C \|f_2'\|_{L^\infty}\eps$.  Using \eqref{2.7}, it follows that
\[
\int_{ |y|\ge \eps} |y|\,  |B_8^+(x,y) - B_9^+(x,y)| \, dy  \le  C (1+ \|f_2'\|_{C^1}^2) \, \eps
\int_\eps^\infty  \frac {\max \{ \hat y, \sqrt{F_{2'}(x)} \}  } {\max\{y,F_{2}(x)\} \max\{y,F_{2'}(x)\}} \,dy.
\]
The last integral is bounded by
\[
\int_\eps^\infty \frac {\max \{ \hat y, \sqrt{F_{2'}(x)} \}  } {y \max\{y,F_{2'}(x)\}} 
\,dy \le \frac C{\sqrt\eps},
\]
which yields the desired bound.

We can now conclude from \eqref{9.1}, \eqref{9.2}, and \eqref{9.3} that
\[
\frac d{dt}  \left\| \sqrt{h}\, f \right\|_{L^2} \le  C \left(1+ \|f_1\|_{C^{2,\gamma}_\gamma}^2 + \|f_2\|_{C^{2,\gamma}_\gamma}^2 \right) \left(\sqrt{\eps+\eps'} + \| f\|_{\tilde L^2} \right)
\]
when  $f_1$ and $f_2$ solve \eqref{7.1} with $\eps\in(0,\frac 12)$ and $\eps'\in(0,\frac 12)$, respectively, and $f:=f_1-f_2$.  Similarly to \eqref{8.0b}, it follows that
\beq\lb{9.7}
\frac d{dt}   \sup_{x_0\in\bbR}  \left\| \sqrt{h_{x_0}}\, f(\cdot,t) \right\|_{L^2} \le  C \left( 1+ \sum_{j=1}^2 \|f_j(\cdot,t)\|_{C^{2,\gamma}_\gamma}^2  \right) \left(\sqrt{\eps+\eps'} +  \sup_{x_0\in\bbR}  \left\| \sqrt{h_{x_0}}\, f(\cdot,t) \right\|_{L^2} \right).
\eeq


\section{Existence of Solutions} \lb{S10}

First we claim that the operator in \eqref{9.0} satisfies
\beq\lb{10.1}
\begin{split}
\sup_{x\in\bbR} & \left| PV \int_\bbR (H_{F_1}^-(x,y) + H_{F_1}^+(x,y) )\,dy - PV \int_\bbR (H_{F_2}^-(x,y) + H_{F_2}^+(x,y) )\,dy \right| 
\\ & \qquad\qquad\qquad \le C (1+\|F_1\|_{C^2_\gamma}^2+\|F_2\|_{C^2_\gamma}^2) \|F_1-F_2\|_{C^{1,\gamma}}.
\end{split}
\eeq
The inside of the absolute value is
\[
\sum_\pm PV \int_\bbR \frac {y\, (F'(x)\pm F'(x-y))}{y^2+(F_1(x)\pm F_1(x-y))^2} \, dy 
+ \sum_\pm PV \int_\bbR  y B_5^\pm(x,x-y) \, dy,
\]
where $F:=F_1-F_2$, $F_3:=F_1+F_2$, and $B_5^\pm$ is from Section \ref{S9}.  The two integrals in the first sum are bounded by $C(1+\|F_1\|_{C^2_\gamma}+\|F_2\|_{C^2_\gamma})\|F\|_{C^{1,\gamma}}$ due to \eqref{2.70}. The same bound easily holds for the integral with $B_5^-$, as well as for the one with $B_5^+$ over $|y|\ge 1$.  We write the last integral as $\int_0^1 y|B_5^+(x,x-y)-B_5^+(x,x+y)| \, dy$,
and again use \eqref{2.7} and a version of \eqref{2.6a} with five terms.  There we always obtain $y^2$ in the numerator, and this cancels the smallest of $y^2+(F_1(x)+ F_1(x\pm y))^2$ and $y^2+(F_2(x)+ F_2(x \pm y))^2$ that appear in the denominator, so the integral is bounded by $C(1+\|F_1'\|_{C^1}^2+\|F_2'\|_{C^1}^2)\|F\|_{C^1}$ times
\begin{align*}
 \int_{0}^1 \frac {  \max\{y^2,F_1(x)\}  } {\max\{y,F_1(x)\}^2} \, dy +  \int_{0}^1 \frac {  \max\{y^2,F_2(x)\}  } {\max\{y,F_2(x)\}^2} \, dy + \int_0^{1} \frac {\max\{y,\sqrt{F_1(x)}\} \max\{y, \sqrt{F_2(x)}\}  } { \max\{ y,F_1(x)\} \max\{ y,F_2(x)\}}  \, dy \le C
\end{align*}
(see an estimate at the end of Section \ref{S8} for the last integral).
%
%
 This proves \eqref{10.1}.

Then for each $\eps>0$ and $n\in\bbN$ we have
\begin{align*}
 & \left\| \phi_\eps*PV \int_\bbR (H_{\phi_\eps*f_1}^-(x,y) + H_{\phi_\eps*f_1}^+(x,y) )\,dy - \phi_\eps* PV \int_\bbR (H_{\phi_\eps*f_2}^-(x,y) + H_{\phi_\eps*f_2}^+(x,y) )\,dy \right\|_{C^n} 
\\ & \qquad\qquad\qquad \le C_n (1+\|f_1\|_{C^2_\gamma}^2+\|f_2\|_{C^2_\gamma}^2) \, \eps^{-n-2} \|f_1-f_2\|_{L^\infty}.
\end{align*}
for some $(n,\gamma)$-dependent constant $C_n$.  This means that we can apply Picard's existence theorem to \eqref{7.1} to obtain its local-in-time  solutions $f$ for initial data from the metric space
$X_\psi:=\{g\in \phi*\psi+C^2(\bbR) \,\big|\,   \inf g>0 \}$, where $\psi\ge 0$ with $\|\psi\|_{\tilde H^3_\gamma}<\infty$ is the initial datum from Theorem \ref{T.1.1}.  From \eqref{2.70a} and \eqref{7.1} we see that these are $C^\infty$ if so are the initial data,
and they can be extended in time as long as they remain uniformly positive and $f(\cdot,t)-\phi*\psi$ remains bounded in $C^2(\bbR)$.  This is because
\[
\|g\|_{C^2_\gamma} = \|g'\|_{C^1} + \|g\|_{\ddot C^{1-\gamma}}
\le \|(g-\phi*\psi)'\|_{C^1} + \|(\phi*\psi)'\|_{C^1} +   \|\phi*\psi\|_{\ddot C^{1-\gamma}} + 2    \|g-\phi*\psi\|_{L^\infty}  
\]
and $\|\psi\|_{C^2_\gamma}<\infty$.
For any $\eps>0$, let $f_\eps$ be such $C^\infty$ solution with  
\beq\lb{10.2a}
f_\eps(\cdot,0)=\phi_\eps*\psi+2\eps\in X_\psi.
\eeq
  Let  $T_\eps\in(0,\infty]$ be its time of existence and let $F_\eps:=\phi_\eps*f_\eps$.

Recall that $h_{x_0}:=h_0(\cdot-x_0)$  for $x_0\in\bbR$ and let
\[
\|g\|_{\tilde H_\gamma}:=  \|g\|_{\ddot C^\gamma} + \sup_{x_0\in\bbR}   \left\| \sqrt{h_{x_0}}\, g''' \right\|_{L^2} .
\]
Then
\beq\lb{10.13}
\max\left\{ \|g\|_{\tilde H^3_\gamma}, \|g\|_{C^{2,\gamma}_\gamma} \right\} \le   C\|g\|_{\tilde H_\gamma} \le C \|g\|_{\tilde H^3_\gamma}
\eeq
because we assume that $\gamma\in(0,\frac 12]$ (recall \eqref{2.0a}), so
from \eqref{2.70a} and \eqref{7.10} we see that 
\beq\lb{10.3}
\frac d{dt}  \|f_\eps(\cdot,t)\|_{\tilde H_\gamma}  \le C  \left(1+\|f_\eps(\cdot,t)\|_{C^{2,\gamma}_\gamma}^4\right) \|f_\eps(\cdot,t)\|_{\tilde H_\gamma} \le C  \left(1+\|f_\eps(\cdot,t)\|_{\tilde H_\gamma}^4 \right) \|f_\eps(\cdot,t)\|_{\tilde H_\gamma}
\eeq
as long as $t<T_\eps$ and  $\eps\le \frac 12\|f_\eps'(\cdot,t)\|_{L^\infty}^{-1}$.
 Let $T>0$ be such that the solution to $z'=C(1+z^4)z$ with $z(0)=\|\psi\|_{\tilde H_\gamma}+1$ satisfies $z(T)=M:=\|\psi\|_{\tilde H_\gamma}+2$, and let $T_\eps':=\min\{T_\eps,T\}>0$.  Then \eqref{10.3} implies that for all $\eps\in(0,(CM)^{-1})$ 
 we have
\beq\lb{10.7}
\sup_{t\in[0,T_\eps')} \|f_\eps(\cdot,t)\|_{\tilde H_\gamma} \le M.
\eeq

It remains to check that $f_\eps$ stays uniformly positive for some uniform time.  Fix any small enough $\eps>0$ as above and  $t\in[0,T_\eps')$.  Then for any $x\in\bbR$ we have (again dropping $t$)
\beq\lb{10.6}
PV \int_\bbR  (H_{F_\eps}^-(x,y) + H_{F_\eps}^+(x,y) )\,dy = I_8^+(x) + I_8^-(x)  - 4 I_9(x),
\eeq
where
\begin{align*}
I_8^\pm(x) & := PV \int_\bbR \frac{y\, F_\eps'(x)}{y^2+(F_\eps(x) \pm F_\eps(x-y))^2}   dy,
\\ I_9 (x)& :=  \int_\bbR \frac{y\, F_\eps'(x-y) F_\eps(x)F_\eps(x-y)}{[y^2+(F_\eps(x)+ F_\eps(x-y))^2]\, [y^2+(F_\eps(x) - F_\eps(x-y))^2]}   dy.
\end{align*}
From \eqref{2.1} we see that
\beq\lb{10.4}
|I_8^\pm(x)| \le C \|F_\eps\|_{C^2_\gamma} |F_\eps'(x)| .
\eeq
Writing  $I_9(x)$  as $\int_0^\infty y(B(x,x-y)-B(x,x+y)) \, dy$ and again using \eqref{2.7} and a version of \eqref{2.6a} with two terms in the denominator, 
we see that
\beq\lb{10.5}
|I_9(x)| \le C(1+\|F_\eps'\|_{C^1}^5) |F_\eps(x)| .
\eeq
The two ``worst'' terms in the latter estimate are (with the factor $F_\eps(x)$ removed)
\begin{align*}
& \left| \int_0^\infty \frac {y\,  F_\eps'(x+y)F_\eps(x+y)\,(2F_\eps(x)+F_\eps(x+y)+F_\eps(x-y))\,(F_\eps(x+y)-F_\eps(x-y))} {[y^2+(F_\eps(x) + F_\eps(x-y))^2]\,[y^2+(F_\eps(x) + F_\eps(x+y))^2] \,[y^2+(F_\eps(x) - F_\eps(x-y))^2]} \,dy \right|
\\ & \qquad \qquad \qquad  \le C(1+\|F_\eps'\|_{C^1}^2) \int_0^\infty \frac { \max\big\{\hat y^2,F_\eps(x) \big\}} {\max\{y,F_\eps(x)\}^2}\, dy
\le C(1+\|F_\eps'\|_{C^1}^2) ,
\end{align*}
which holds because
\[
F_\eps(x+y)\,(2F_\eps(x)+F_\eps(x+y)+F_\eps(x-y)) \le 2 (F_\eps(x)+F_\eps(x+y))^2 + (F_\eps(x)+F_\eps(x-y))^2 ,
\]
and
\begin{align*}
& \left| \int_0^\infty \frac {y\,  F_\eps'(x+y)F_\eps(x+y)\,(2F_\eps(x)-F_\eps(x+y)-F_\eps(x-y))\,(F_\eps(x+y)-F_\eps(x-y))} {[y^2+(F_\eps(x) + F_\eps(x+y))^2]\,[y^2+(F_\eps(x) - F_\eps(x-y))^2] \,[y^2+(F_\eps(x) - F_\eps(x+y))^2]} \,dy \right|
\\ & \qquad \qquad \qquad  \le C(1+\|F_\eps'\|_{C^1}^5) \left[ \int_0^1 \frac { \max\big\{ y^2,F_\eps(x) \big\}^2 } {\max\{y,F_\eps(x)\}^2}\, dy + \int_1^\infty \frac 1{y^2}\, dy \right]
\le C(1+\|F_\eps'\|_{C^1}^5) .
\end{align*}
From \eqref{10.6}, \eqref{10.4}, and \eqref{10.5} we thus obtained
\beq\lb{10.11}
 \left| PV \int_\bbR  (H_{F_\eps}^-(x,y) + H_{F_\eps}^+(x,y) )\,dy \right| \le C(1+\|F_\eps\|_{C^2_\gamma}) \, ( \|F_\eps\|_{C^2_\gamma}^4 F_\eps(x)+|F_\eps'(x)|).
 \eeq

Let  now  $m(t):=\inf f_\eps(\cdot,t)>0$ and consider any $x'\in\bbR$ such that $f_\eps(x',t)\le m(t)+\eps^2$.  Then \eqref{2.7} applied to $f_\eps(\cdot,t)-m(t)\ge 0$ shows that
\[
\sup_{|x-x'|\le2\eps}|\partial_x f_\eps(x,t)| \le C(1+\|\partial_x f_\eps(\cdot,t)\|_{C^1})\, \eps,
\]
so \eqref{10.11}  and \eqref{7.1} yield
\beq\lb{10.9}
|\partial_t f_\eps(x',t)| \le C(1+\|f_\eps(\cdot,t)\|_{C^2_\gamma}^5) \,(m(t)+\eps).
\eeq
It follows that 
\[
m'(t)\ge - C(1+\|f_\eps(\cdot,t)\|_{C^2_\gamma}^5)\, (m(t)+\eps),
\]
which together with $m(0)\ge 2\eps$  (recall \eqref{10.2a}) yields $m(t)\ge \eps $ for all $t\le \min\{T_\eps,T'\}$, where $T':=\min\{T,\frac 1{3C(1+M^5)}\}>0$.  This, \eqref{10.7}, and \eqref{10.13} show that $f_\eps$ can be continued within $X_\psi$ past time $T_\eps$ if $T_\eps\le T'$, hence $T_\eps> T'$ for all small $\eps>0$.  

From \eqref{10.7} and \eqref{10.11} we see that $\tilde f_\eps:= f_\eps-\phi*\psi\in L^\infty([0,T']; \tilde H^3(\bbR))$.
Then \eqref{9.7} and $f_\eps\ge 0$ show that $\tilde f_\eps$ converges to some $\tilde f\ge -\phi*\psi$ in $L^\infty([0,T']; \tilde L^2(\bbR))$ as $\eps\to 0$ because we  also have $\tilde f_\eps(\cdot,0)\to \psi-\phi*\psi$ in $L^\infty(\bbR)\subseteq \tilde L^2(\bbR)$ (this yields $\tilde f(\cdot,0)= \psi-\phi*\psi$ as well).  Then \eqref{10.7} shows that $f:=\tilde f + \phi*\psi\ge 0$ satisfies
\beq\lb{10.8}
\sup_{t\in[0,T']} \|f(\cdot,t)\|_{\tilde H_\gamma} \le M,
\eeq
and $\tilde f_\eps\to \tilde f$ holds also in $L^\infty([0,T']; C^{2}(\bbR))$.  Then $F_\eps-\phi*\psi\to \tilde f$ in this space, so \eqref{10.1} shows that
\beq \lb{10.12}
\lim_{\eps\to 0} \|\partial_t f_\eps - G\|_{L^\infty(\bbR\times [0,T'])} =0,
\eeq
where $G$ is the right-hand side of \eqref{1.5}.  Thus $G\in C(\bbR\times [0,T'])$ because all $\partial_t f_\eps$ are smooth.  But then \eqref{10.12} and pointwise convergence $f_\eps\to f$ shows that $\partial_t f$ exists and equals $G$, which means that $f\ge 0$ is a classical solution to \eqref{1.5} that satsfies \eqref{10.8}.
We can now continue it on some time interval $[0,T_\psi)$ with $T_\psi\in(0,\infty]$ maximal such that
\beq\lb{10.12a}
\sup_{t\in[0,T]} \|f(\cdot,t)\|_{\tilde H_\gamma} <\infty
\eeq
for all $T\in(0,T_\psi)$.  This and \eqref{10.13} yield  \eqref{1.7}.  And since \eqref{10.3} implies
\beq\lb{10.14}
\frac d{dt}  \|f(\cdot,t)\|_{\tilde H_{\gamma}}  \le C \left(1+\|f(\cdot,t)\|_{C^{2,\gamma}_{\gamma}}^4\right) \|f(\cdot,t)\|_{\tilde H_{\gamma}}
\eeq
at $t=0$, this then holds for any $t\in[0,T_\psi)$ because  we can apply the above arguments with initial time $t$ and initial condition $f(\cdot,t)$ (the newly obtained $\tilde f_\eps$ converge to the same $\tilde f$ in $L^\infty([t,T']; C^{2}(\bbR))$  by the uniqueness claim in Theorem \ref{T.1.1}(ii), proved in Section \ref{S8}).

Next, for any $\gamma'\in(0,\gamma]$ we have $\|\cdot\|_{\tilde H_{\gamma'}}\le \|\cdot \|_{\tilde H_{\gamma}}$, so $f$ is also the unique solution to \eqref{1.5} on the time interval $[0,T_\psi)$ when $\gamma$ is replaced by $\gamma'$ (this is still the $T_\psi$ corresponding to $\gamma$).  
Then \eqref{10.14}  holds also for $\gamma'$ and any $t\in[0,T_\psi)$.
  Assume now that $T_\psi<\infty$ but \eqref{1.8} fails.  Since \eqref{2.70a} shows that for $\gamma''\in\{\gamma,\gamma'\}$ we have
\[
\frac d{dt}  \|f(\cdot,t)\|_{\ddot C^{\gamma''}} \le C  \left(1 +  \|f_x(\cdot,t)\|_{C^{1}}^2 \right) \left(1+ \|f(\cdot,t)\|_{\ddot C^{\gamma''}} \right),
\]
it follows that $\sup_{t\in[0,T_\psi)} \|f(\cdot,t)\|_{\ddot C^{\gamma''}}<\infty$.  This with $\gamma''=\gamma$, together with the definition of $T_\psi$ and \eqref{10.14} yield
\[
\lim_{t\to T_\psi} \|f_{xxx}(\cdot,t)\|_{\tilde L^2}=\infty.
\]
But then also
\[
\lim_{t\to T_\psi} \|f(\cdot,t)\|_{\tilde H_{\gamma'}} =\infty,
\]
so \eqref{10.14} with $\gamma'$ in place of $\gamma$ forces \eqref{1.8} to hold, a contradiction.  
This now also proves \eqref{1.8} for all $\gamma'\in(0,1)$, and shows that $T_\psi$ is independent of $\gamma$.

Since $f$ can clearly be continued at least as long as the solution $z$ above that defined $T$ exists,  \eqref{10.13} also yields \eqref{1.10}.  Finally, \eqref{10.1} with $F_1:=f(\cdot+z,t)$ ($z\neq 0$) and $F_2:=f(\cdot,t)$, shows that 
for any $t\in[0,T_\psi)$ we have
\[
\|f_t(\cdot,t)\|_{W^{1,\infty}} \le  \|f_t(\cdot,t)\|_{L^\infty} + C \left(1+\|f(\cdot,t)\|_{C^2_\gamma}^2\right) \|f_x(\cdot,t)\|_{C^{1,\gamma}}.
\]
Then \eqref{2.70a}, \eqref{1.7}, and \eqref{10.13} yield \eqref{1.7a},
which concludes the proof of Theorem \ref{T.1.1}(i).


\vskip 3mm
\noindent
{\bf Proof of the last claim in Theorem \ref{T.1.1}(ii).} 
All the arguments in Sections \ref{S3}--\ref{S7} trivially  extend to the case $h\equiv 1$,  with all $\|\cdot\|_{\tilde L^2}$ and $\|\cdot\|_{\tilde L^2_{x_0}}$ replaced by $\|\cdot\|_{L^2}$, because then we always have $h(x)-h(y)\equiv 0\equiv h'(x)$ and so terms with these factors vanish. Taking $\eps\to 0$ in the corresponding version of \eqref{7.10} yields (dropping $t$ again)
\beq\lb{10.15}
\frac d{dt}  \left\|  f''' \right\|_{L^2}^2 
\le C (1+  \|f\|_{C^{2,\gamma}_\gamma}^4 ) \, (\|f'\|_{ L^2 }^2 + \|f''\|_{ L^2 }^2 + \|f'''\|_{ L^2 }^2) 
\eeq
on the time interval $[0,T_\psi)$, similarly to \eqref{10.14}.

Let  next $\psi_0:=\phi*(\psi_{-\infty}\chi_{(-\infty,0)} + \psi_{\infty}\chi_{[0,\infty)})$ and $f_0(x,t):=f(x,t)-\psi_0(x)$.  Then
\begin{align*}
\frac d{dt}  \left\| f_{0}  \right\|_{L^2}^2
 & =  2 \sum_\pm \int_\bbRr  f_0(x)  \, PV \int_\bbR \frac {y\, (f_{0}'(x)\pm f_{0}'(x-y))}{y^2+(f(x)\pm f(x-y))^2} \, dydx
\\ & + 2 \sum_\pm \int_\bbRr  f_0(x)  \, PV \int_\bbR \frac {y\, (\psi_{0}'(x)\pm \psi_{0}'(x-y))}{y^2+(f(x)\pm f(x-y))^2} \, dydx .
\end{align*}
The arguments in Section \ref{S3} estimating $I_1^\pm$, with $h\equiv 1$, show that the first sum on the right-hand side  is bounded above by $C(1+  \|f\|_{C^{2,\gamma}_\gamma}^2 ) \| f_0\|_{L^2}^2$.
From $\|\psi_0\|_{C^2}\le C|\psi_\infty-\psi_{-\infty}|$ and \eqref{2.70} we see that the inside integrals in the second sum are bounded by $C(1+\|f\|_{C^2_\gamma})|\psi_\infty-\psi_{-\infty}|$ for all $x\in\bbR$.
But they are also  bounded by $C|\psi_\infty-\psi_{-\infty}||x|^{-1}$ when $|x|\ge 2$ because $\psi_0'\equiv 0$ on $\bbR\setminus(-1,1)$.  Therefore
\[
\frac d{dt}  \left\| f_{0}  \right\|_{L^2}^2 \le C(1+  \|f\|_{C^{2,\gamma}_\gamma}^2 ) \| f_0\|_{L^2} (\| f_0\|_{L^2} + |\psi_\infty-\psi_{-\infty}|),
\]
which together with \eqref{10.15} yields
\beq\lb{10.16}
\frac d{dt} ( \left\|  f''' \right\|_{L^2} +\left\| f_{0}  \right\|_{L^2} )
\le C (1+  \|f\|_{C^{2,\gamma}_\gamma}^4 ) \, ( \left\|  f''' \right\|_{L^2} +\left\| f_{0}  \right\|_{L^2}  + |\psi_\infty-\psi_{-\infty}|) .
\eeq
From this, \eqref{1.7}, and $\psi_0\in H^3(\bbR)$ we get \eqref{1.11} with $\psi_0$ in place of $\psi$, and then \eqref{1.11} holds as well.

\section{Proof of Theorem \ref{T.1.2}} \lb{S11}

\noindent
{\bf Proof of (ii).} 
Note that 
 \eqref{10.9}, $\tilde f_\eps(\cdot,t)\to \tilde f(\cdot,t)$ in $ C^{2}(\bbR)$, $\partial_t \tilde f_\eps(\cdot,t)\to \partial_t \tilde f(\cdot,t)$ in $ L^\infty(\bbR)$, and \eqref{2.70a}  show that for all $t\in[0,T_\psi)$ we  have
\beq\lb{10.10}
\left| \frac d{dt} \inf f(\cdot,t) \right| \le C \min \left\{ (1+\|f(\cdot,t)\|_{C^2_\gamma}^5) \inf f(\cdot,t), \, (1+\|f_x(\cdot,t)\|_{L^\infty}) \|f(\cdot,t)\|_{C^2_\gamma} \right\}.
\eeq
This, \eqref{10.13}, and \eqref{10.12a} imply that $\inf f(\cdot,t)=0$ for all $t\in[0,T_\psi)$ if $\inf \psi=0$.

\vskip 3mm
\noindent
{\bf Proof of (i).} Since
\[
\frac{ f(x)\pm f(x-y) \pm y\, f'(x-y)}{y^2+(f(x)\pm f(x-y))^2}  
= - \frac d{dy} \arctan  \frac{ f(x)\pm f(x-y)} y ,
\]
we see that
\[
 \int_{\bbR}  \frac{ f(x)\pm f(x-y) \pm y\, f'(x-y)}{y^2+(f(x)\pm f(x-y))^2}  \, dy=
 \begin{cases}
 \pi & \text{$\pm$ is $+$ and $f(x)>0$}
 \\ 0 &  \text{$\pm$ is $-$ or $f(x)=0$}
 \end{cases}
\]
for $ f\ge 0$ with $\| f\|_{C^2_\gamma(\bbR)}<\infty$.
Hence \eqref{1.5} can be equivalently written as \eqref{1.6}.

Assume that $M(t):=\sup f(\cdot,t)>0$ for some $t\in[0,T_\psi)$ (if $M(t)=0$, then $f(\cdot,t')\equiv 0$ for all $t'\in [t,T_\psi)$).  Let $\delta\in(0,\sqrt{M(t)})$ and consider any $x\in\bbR$  such that $f(x,t)\ge M(t)-\delta^2$.  Then \eqref{2.1} and \eqref{2.7} for the function $M(t)-f(\cdot,t)$ show that (we again drop $t$)
\beq\lb{11.1}
\left| PV \int_\bbR  \frac{y f'(x)}{y^2+(f(x)\pm f(x-y))^2}  dy \right| \le C(1+\|f\|_{C^2_\gamma}^2)\delta,
\eeq
and we clearly have
\beq\lb{11.2}
 - \int_{|y|\ge\delta}   \frac{f(x)- f(x-y)}{y^2+(f(x)- f(x-y))^2}  dy +  \int_{|y|\ge\delta}  \frac{f(x)- f(x-y)}{y^2+(f(x)+ f(x-y))^2}  dy \le \int_{|y|\ge\delta} \frac{\delta^2}{y^2} = 2 \delta .
\eeq
Writing 
\[
f(x)- f(x-y)=[f(x)- f(x-y)-f'(x)y]+f'(x)y
\]
and then using \eqref{2.1} and \eqref{2.7} yields
\beq\lb{11.3}
 \left| \int_{-\delta}^{\delta}   \frac{f(x)- f(x-y)}{y^2+(f(x) \pm f(x-y))^2}  dy  \right| \le C (1+ \|f\|_{C^2_\gamma}^2 )\delta .
\eeq
Finally,
\beq\lb{11.4}
- \int_\bbR  \frac{2f(x)}{y^2+(f(x)+ f(x-y))^2}  dy  \le - \int_\bbR  \frac{2f(x)}{y^2+(2f(x)+\delta^2)^2}  dy  
= -\pi \,\frac{2f(x)}{2f(x)+\delta^2},
\eeq
so from \eqref{11.1}--\eqref{11.4} and \eqref{1.6} we obtain
\[
 f_t(x) \le C (1+ \|f\|_{C^2_\gamma}^2 + \delta M(t)^{-1})\delta.
\]
This holds uniformly in $t\in[0,T]$ for any $T<T_\psi$, so taking $\delta\to 0$ shows that 
\[
\limsup_{t'\to t^+} \frac{M(t')-M(t)}{t'-t}\le 0
\]
for each $t\in[0,T_\psi)$.  Hence $M$ is non-increasing on $[0,T_\psi)$.

Assume now that $m(t):=\inf f(\cdot,t)>0$ for some $t\in[0,T_\psi)$ (we already proved that if $m(t)=0$, then $m(t')=0$ for all $t'\in [t,T_\psi)$).  Let $\delta\in(0,\sqrt{m(t)})$ and consider any $x\in\bbR$  such that $f(x,t)\le m(t)+\delta^2$.  Then  (dropping $t$) we again have \eqref{11.1} and \eqref{11.3}, while  \eqref{11.2} becomes
\[
 - \int_{|y|\ge\delta}   \frac{f(x)- f(x-y)}{y^2+(f(x)- f(x-y))^2}  dy +  \int_{|y|\ge\delta}  \frac{f(x)- f(x-y)}{y^2+(f(x)+ f(x-y))^2}  dy \ge -\int_{|y|\ge\delta} \frac{\delta^2}{y^2} = -2 \delta
\]
and \eqref{11.4} becomes
\[
- \int_\bbR  \frac{2f(x)}{y^2+(f(x)+ f(x-y))^2}  dy  \ge - \int_\bbR  \frac{2f(x)}{y^2+(2f(x)-\delta^2)^2}  dy  
= -\pi \,\frac{2f(x)}{2f(x)-\delta^2}.
\]
So we now obtain
\[
 f_t(x) \ge - C (1+ \|f\|_{C^2_\gamma}^2 + \delta m(t)^{-1})\delta ,
\]
and conclude as above that $m$ is non-decreasing on $[0,T_\psi)$.

\vskip 3mm
\noindent
{\bf Proof of (iii).} 
Let $f_0:=f-\tilde f$
and
\[
A(t):= \sup_{x_0\in\bbR} \left( 1+(\mu-|x_0|)_+\right) \left\| \sqrt{h_{x_0}}\, f_0 (\cdot,t)\right\|_{L^2}^2 
\]
for $t\in [0,\min\{T_\psi, T_{\tilde \psi}\})$.
 Then
\[
\| f_0(\cdot,t)\|_{\tilde L^2_{x_0}}^2 
\le  \sum_{n=1}^{\lfloor \frac 12(\mu-|x_0|)_+ \rfloor} \frac {C A(t)n^{-2}} {1+(\mu-|x_0|)_+} +  \sum_{n\ge \lfloor \frac 12(\mu-|x_0|)_+ \rfloor} C A(t) n^{-2} \le  \frac {C A(t)} {1+(\mu-|x_0|)_+}
\]
for all $x_0\in\bbR$.
From this and \eqref{8.1} we see that
\[
A'(t) \le  C \left(1+  \|f(\cdot,t)\|_{C^{2,\gamma}_\gamma}^2 + \|\tilde f(\cdot,t)\|_{C^{2,\gamma}_\gamma}^2 \right) A(t) ,
\]
and the claim follows.


\section{Proof of Theorem \ref{T.1.3}} \lb{S12}

\noindent
{\bf Proof of (i).} 
Derivation of \eqref{1.13} and the proof of Theorem \ref{T.1.3} are already contained in Sections \hbox{\ref{S2}--\ref{S11},} if we simply ignore all the arguments involving integrals resulting from the second term in \eqref{1.5} (which simplifies the proof considerably).  This includes estimates on all integrals with superscript $+$ and the argument in Section \ref{S10} showing that $f_\eps$ remains positive for a uniformly positive time.  Note that these were the only places where we used \eqref{2.7}, other than the proof of Theorem \ref{T.1.2}(i) (where it applies even for \eqref{1.13}, and where we also ignore all the integrals involving $f(x)+ f(x-y)$).


\vskip 3mm
\noindent
{\bf Proof of (ii).} 
For the Muskat Problem on the strip $\bbR\times(0,l)$, we consider \eqref{1.3} with the Laplacian on $\bbR\times(0,l)$, which means that with ${\bf y}_n:=(y_1,y_2-2ln)$ we now have
\begin{equation} \lb{12.2}
u({\bf x},t) = \frac 1{2\pi} \sum_{n\in\bbZ} \int_{\bbR\times(0,l)} \left( \frac{({\bf x}-{\bf y}_n)^\perp}{|{\bf x}-{\bf y}_n|^2} - \frac{({\bf x}-\bar {\bf y}_n)^\perp}{|{\bf x}-\bar {\bf y}_n|^2} \right) \rho_{x_1}({\bf y},t) \, d{\bf y}.
\end{equation}
The derivation of \eqref{1.5} in Section \ref{S2}, together with assuming that $\rho_1-\rho_0=2\pi$, now yields 
\eqref{12.1} with  $\Theta_l$ from \eqref{12.1a} when we also use that
\begin{equation} \lb{12.3}
 \sum_{n\in\bbZ}   \frac{y}{y^2+(r-2ln)^2}  
 = \frac \pi l \sum_{n\in\bbZ}   \frac{\frac {\pi y}l}{(\frac {\pi y}l)^2+(\frac {\pi r}l-2\pi n)^2}
 =\frac{\pi}{l} \Theta_\pi (\tfrac {\pi y}l, \tfrac {\pi r}l)  =  \Theta_l(y,r)
\eeq
(this derivation, including the proof of the middle equality in \eqref{12.3}, appears in \cite{CorGraOri}).  Let us assume for simplicity that $l=\pi$, as the general case is identical, and define
\beq\lb{12.4}
\begin{split}
\Theta^-_0(y,r) &:= \frac{y}{y^2+r^2}  
\\ \Theta^+_0(y,r) &:= \frac{y}{y^2+r^2} + \frac{y}{y^2+(2\pi-r)^2}  ,
\\ \Theta^-(y,r) &:= \Theta_\pi(y,r) - \Theta^-_0(y,r)  =  \sum_{n\in\bbZ\setminus\{0\}}   \frac{y}{y^2+(r-2 \pi n)^2}  ,
\\ \Theta^+(y,r) &:= \Theta_\pi(y,r) -  \Theta^+_0(y,r)  =  \sum_{n\in\bbZ\setminus\{0,1\}}   \frac{y}{y^2+(r-2\pi n)^2}  .
\end{split}
\eeq
The reason for this is that when $r=f(x)- f(x-y)$ (we again drop $t$ from the notation), then denominators of all the terms in the first sum in \eqref{12.3} are no less than $\pi^2$, except for the one with $n=0$.  And when $r=f(x)+ f(x-y)$, then the same is true for all $n\neq 0,1$.  

We now separate the right-hand side of \eqref{12.1} accordingly when replicating Sections \ref{S3}-\ref{S6} for \eqref{12.1}.  The integrals $I^\pm_j$ ($j=1,2,3,4$),  arising from differentiation of 
\[
PV \int_\bbR (f_x(x)\pm f_x(x-y))\, \Theta^\pm_0 (y,f(x)\pm f(x-y))  \, dy
\]
in $x$ are the same as in Section \ref{S2}, but those with $\pm$ being $+$ each have a second term that comes from $\frac{y}{y^2+(\pi - f(x)+ \pi- f(x-y))^2}$.  All the estimates in Sections \ref{S3}--\ref{S6} carry over to this term, with $\pi-f$ in place of $f$, and  for the integral
\[
L_1^0 : = \int_{\bbR^2} \frac{h(y)( g(x)-g(y))^2}{(x-y)^2+(\pi-f(x)+ \pi- f(y))^2} \, dydx  
\]
(which is analogous to $L_1^+$) 
we also obtain  (recall that $h=h_{x_0}$ and $g=f'''$)
\[
L_1^+ + L_1^0 - L_1^- \le  \int_{\bbR^2} \frac{h(y)( g(x)-g(y))^2}{(x-y)^2+ \pi^2} \, dydx  
\le C  \| g\|_{\tilde L^2_{x_0}}^2
\]
instead of the first inequality in \eqref{2.90}, which holds  because $0\le f\le\pi$ implies
\begin{align*}
\min\{f(x)+ f(y), 2\pi -(f(x)+ f(y)) \} & \ge |f(x)- f(y)| ,
\\ \max\{f(x)+ f(y), 2\pi -(f(x)+ f(y)) \} & \ge \pi .
\end{align*}

On the other hand, all the   integrals   arising from differentiation of 
\[
PV \int_\bbR (f_x(x)\pm f_x(x-y))\, \Theta^\pm (y,f(x)\pm f(x-y))  \, dy
\]
in $x$ three times can be estimated much more easily.  First, those involving fourth derivatives of $f$ (analogous to $J_1^\pm$ and $J_2^\pm$) are rewritten in terms of only third (and lower order) derivatives of $f$ via integration by parts (cf.~\eqref{2.4} and \eqref{2.91}):
\begin{align}
 \int_\bbRr h(x)  & g(x) \, PV  \int_\bbR g'(x)\,\Theta^\pm (y,f(x)\pm f(x-y)) \, dydx   \notag
\\  = & - \frac 12 \int_\bbRr  h(x)  g(x)^2 \, PV \int_\bbR \frac d{dx}\Theta^\pm (y,f(x)\pm f(x-y))\, dydx \notag
\\ & - \frac 12 \int_\bbRr  h'(x)  g(x)^2 \, PV \int_\bbR \Theta^\pm (y,f(x)\pm f(x-y)) \, dydx  \lb{12.5}
\end{align}
and 
\begin{align*}
 \int_\bbRr h(x) & g(x) \, PV \int_\bbR g'(y) \,\Theta^\pm (x-y,f(x)\pm f(y)) \, dydx 
 \\  = & -  \int_\bbRr  h(x)  g(x) \, PV \int_\bbR g(y)\,\frac d{dy} \Theta^\pm (x-y,f(x)\pm f(y))\, dydx.
\end{align*}
The other terms are left as they are, and then in each term we separately estimate the integrals over $|y|\le 1$ and over $|y|\ge 1$ (in the last integral we first change variables $y\leftrightarrow x-y$).  For the former, we use the sum form of $\Theta^\pm$ from \eqref{12.4}  because each of these sums (for each appearing derivative of $\Theta^\pm$) is clearly uniformly bounded in $|y|\le 1$.
On the other hand, the integrals over $|y|\ge 1$ are estimated using the non-sum form of $\Theta^\pm$ from \eqref{12.4}, similarly to the analogous estimates involving derivatives of $\frac y{y^2+(f(x)\pm f(x-y))^2}$ at $|y|\ge 1$.  This time we even have bounded $f$, which makes the estimation easier, with the only difference being that $\Theta_\pi(y,r)\to \pm\frac 12$ (instead of $0$) as $y\to\pm\infty$.  But since derivatives of $\Theta_\pi(y,r)$ decay rapidly to 0 as $|y|\to\infty$,  this is only relevant to the only term where $\Theta^\pm$ is not differentiated, namely \eqref{12.5}.  The part of it that comes from $\Theta_\pi$ is
\[
-\frac 12 \int_\bbRr  h'(x)  g(x)^2 \, PV \int_{|y|\ge 1} \Theta_\pi (y,f(x)\pm f(x-y)) \, dydx
\]
and this  is  bounded by $C\| g\|_{\tilde L^2_{x_0}}^2$ due to
\[
\sup_{r\in\bbR} \left| \Theta_\pi(y,r)- \frac 12 \right| \le \frac 1{e^{|y|}-2}
\]
whenever $|y|\ge 1$.

It follows that we again obtain  \eqref{2.3}, and Sections \ref{S7}--\ref{S11} then extend to \eqref{12.1} in a straightforward manner, except for the proof of Theorem \ref{T.1.2}(i).  The only adjustment needed is in \eqref{10.2a}, where we instead choose 
\[
f_\eps(\cdot,0)=(1- {4\eps} \pi^{-1})\,\phi_\eps*\psi+2\eps,
\]
 which means that we also have $\sup f_\eps(\cdot,0) \le \pi-2\eps$.  And then the proof of the last claim in Theorem \ref{T.1.3}(ii) (with $l=\pi$) is also virtually identical to that of Theorem \ref{T.1.2}(ii).
 
It remains to prove the claim of Theorem \ref{T.1.2}(i) for \eqref{12.1}.  When $\psi-\frac l2\in H^3(\bbR)$ and $\|\psi-\frac l2\|_{L^\infty}<\frac l2$, then this was proved in \cite[Theorem 5]{CorGraOri}.  To obtain the result for general  $\psi\in\tilde H^3(\bbR)$ with $0\le\psi\le l$, it suffices to apply Theorem \ref{T.1.2}(iii) for \eqref{12.1} with $\tilde\psi$ being $\tilde \psi_{n,\eps}:=(1-\eps)(\psi-\frac l2)\chi_n+\frac l2$ and $\mu:=n$, where $\chi_n$ is a smooth characteristic function of $[-n,n]$.  Then Theorem \ref{T.1.2}(i) for \eqref{12.1} follows from this and \cite[Theorem 5]{CorGraOri} applied to the solutions with initial data $\tilde \psi_{n,\eps}$ (after taking first $n\to\infty$ and then $\eps\to 0$) because we have \eqref{1.10} and \eqref{1.7} holds uniformly in $(n,\eps)$.


\end{document}